\documentclass[reqno,twoside]{amsart} 

\usepackage{amsmath,amsthm,amssymb,amsfonts,mathrsfs,color}
\usepackage{stmaryrd}
\usepackage{empheq}
\usepackage{cleveref}
\usepackage{dsfont}

\usepackage{pgfplots}
\usepackage{mathrsfs}
\usetikzlibrary{arrows}
\usetikzlibrary[patterns]
\definecolor{aqaqaq}{rgb}{0.6274509803921569,0.6274509803921569,0.6274509803921569}
\definecolor{ffffff}{rgb}{1.,1.,1.}
\definecolor{eqeqeq}{rgb}{0.8784313725490196,0.8784313725490196,0.8784313725490196}
\definecolor{uuuuuu}{rgb}{0.26666666666666666,0.26666666666666666,0.26666666666666666}
\definecolor{qqqqff}{rgb}{0.,0.,1.}
\definecolor{ffqqqq}{rgb}{1.,0.,0.}

\theoremstyle{plain}
\begingroup
\newtheorem{thm}{Theorem}[section]
\newtheorem{lem}[thm]{Lemma}
\newtheorem{prop}[thm]{Proposition}
\newtheorem{cor}[thm]{Corollary}
\endgroup

\theoremstyle{definition}
\begingroup
\newtheorem{defn}[thm]{Definition}
\newtheorem{rem}[thm]{Remark}
\endgroup

\numberwithin{equation}{section}

\newcommand{\res}{\mathop{\hbox{\vrule height 7pt width .5pt depth 0pt
\vrule height .5pt width 6pt depth 0pt}}\nolimits}

\newcommand{\ds}{\displaystyle}

\newcommand{\N}{\mathbb N} 
\newcommand{\R}{\mathbb R}

\newcommand{\Ms}{{\mathbb M}^{n{\times}n}_{\rm sym}}

\newcommand{\F}{{\mathcal F}}
\newcommand{\G}{{\mathcal G}}

\newcommand{\wto}{\rightharpoonup}
\newcommand{\e}{\varepsilon}

\newcommand{\LL}{{\mathcal L}}
\newcommand{\HH}{{\mathcal H}}
\newcommand{\M}{{\mathcal M}}

\renewcommand{\bar}[1]{\overline{#1}}
\newcommand{\norme}[1]{\Vert #1 \Vert}

\let\O=\Omega

\begin{document}
 
\title[Discrete approximation of the Griffith functional]{Discrete approximation of the Griffith functional by adaptive finite elements}
\author[J.-F. Babadjian]{Jean-Fran\c cois Babadjian}
\author[E. Bonhomme]{\'Elise Bonhomme}

\address[J.-F. Babadjian, E. Bonhomme]{Universit\'e Paris-Saclay, CNRS,  Laboratoire de math\'ematiques d'Orsay, 91405, Orsay, France.}
\email{jean-francois.babadjian@universite-paris-saclay.fr}
\email{elise.bonhomme@universite-paris-saclay.fr}

\maketitle

\begin{abstract} 
This paper is devoted to show a discrete adaptive finite element approximation result for the isotropic two-dimensional Griffith energy arising in fracture mechanics. The problem is addressed in the geometric measure theoretic framework of generalized special functions of bounded deformation which corresponds to the natural energy space for this functional. It is proved to be approximated in the sense of $\Gamma$-convergence by a sequence of discrete integral functionals defined on continuous piecewise affine functions. The main feature of this result is that the mesh is part of the unknown of the problem, and it gives enough flexibility to recover isotropic surface energies.
\end{abstract}


\section{Introduction}

\subsection{The variational approach to fracture}

The Griffith functional has been introduced in the context of brittle fracture. It finds its roots in the seminal work of Griffith \cite{Griffith} whose main ideas have been revisited in \cite{FM} (see also the monograph \cite{BFM}) into a variational evolution formulation. The main point is that, in a quasi-static setting and in presence of irreversibility, a constrained global minimization principle together with an energy balance select equilibrium states of an elastic body experiencing brittle fracture. In a nutshell, the Griffith energy is defined by
\begin{equation}\label{eq:isoG}
\mathcal G(u,K):=\int_{\Omega \setminus K} |e(u)|^2\, dx + \mu \HH^{N-1}(K),
\end{equation}
where $\Omega \subset \R^N$, a bounded open set, stands for the reference configuration of an elastic material, $K \subset \Omega$ is a codimension-one set representing the crack, $u:\Omega \setminus K \to \R^N$ is the displacement field which might be discontinuous across $K$, and its symmetric gradient $e(u):=(\nabla u + \nabla u^T)/2$ is the linearized elastic strain. The constant $\mu>0$ is a material parameter called toughness. This energy puts in competition a bulk energy, representing the elastic energy stored in the body outside the crack, and a surface energy penalizing the presence of the crack $K$ through its $(N-1)$-dimensional Hausdorff measure, henceforth denoted by $\HH^{N-1}$.

\medskip

This problem falls within the framework of so-called free discontinuity problems (according to De Giorgi's terminology), and it presents many formal analogies with its scalar counterpart,  the Mumford-Shah functional. Although, thanks to geometric measure theory, the existence theory for the latter is by now quite well understood (see e.g. \cite{AFP} and references therein), the minimization of the Griffith functional had to face serious additional difficulties. In particular, a satisfactory existence theory has only recently been solved. As for the Mumford-Shah functional, it passes through the introduction of a ``weak formulation'' where the crack is replaced by the jump set $J_u$ of $u$. A convenient functional setting to investigate this problem is that of functions of bounded deformation, $BD(\O)$, which correspond to (integrable) vector fields $u :\O \to \R^N$ whose distributional symmetric gradient $Eu$ is a bounded Radon measure. This space has been introduced in \cite{Suquet} (see also \cite{Temam,ST1}) as a natural space to formulate problems of perfect plasticity. Brittle fracture however requires a finer understanding of this space and especially the introduction of the subspace $SBD(\O)$ of special functions of bounded deformation in \cite{ACDM,BCDM}, for which the singular part of $Eu$ with respect to the Lebesgue measure is concentrated on the jump set. Unfortunately, this step forward was still not enough because of lack of control of the values of $u$ (due to the failure of Poincar\'e-Korn and/or Korn type inequalities in that space). It is only recently that the introduction of the space $GSBD(\O)$ of generalized special functions of bounded deformation in \cite{DM2} (see Section \ref{sec:2} for the precise definition) has given a satisfactory mathematical framework to investigate a well founded existence theory for the weak formulation, as well as for the original one. Some further compactness properties of that space have been investigated in \cite{CC,CC0} which has led to prove the existence of minimizers of the Griffith functional under Dirichlet boundary conditions (formulated in a relaxed sense). 

\subsection{Approximation of the Griffith energy}

The $\Gamma$-convergence approximation of free discontinuity problems (e.g. by more tractable ones from a numerical point of view) is of fundamental importance in applications. It has been proven in \cite{BDM} that it is not possible to approximate free discontinuity functionals by means of local integral functionals. To overcome this difficulty, a first possibility is to introduce an additional variable like, e.g., in phase field approximations where the sharp discontinuity is smoothened into a diffuse discontinuity. It represents one of the most popular methods which have already proven to be successful in other contexts such as the Modica-Mortola approximation of the perimeter functional \cite{MM}, or the Ambrosio-Tortorelli approximation of the Mumford-Shah functional \cite{AT}. In the context of brittle fracture, such approximations, which have a founded mechanical interpretation as a gradient damage model, have only recently been established in full generality in  \cite{CCdensity} (see also \cite{C,I2}). The main drawback is that, an additional numerical approximation would give rise to a multiscale problem with on the one hand the parameter of approximation, and on the other hand the mesh size (see e.g. \cite{BC,BBZ,CSS}). Another possibility is to use nonlocal integral functionals as e.g. in \cite{BDM,N3,SS}.

\medskip

For what concerns the numerical treatment of free discontinuity problems, the main difficulty is related to the fact that the jump set is part of the unknowns and that standard discontinuous finite element methods do not in general apply in this context. Having this problematic in mind as well as the multiscale issues arising in phase field or nonlocal approximations, one is thus tempted to find single scale discrete approximations of free discontinuity problems. There is a huge literature on this subject and, without being exhaustive, we refer to discrete-to-continuous approximations results \cite{AFG,C1,C2,Go,N1,N,N2}, nonlocal finite elements approximations \cite{LN,N} or discrete approximations based on stochastic meshes in \cite{BCR,R}. 

Let us focus on the discrete approximation result obtained in \cite{CDM} for the Mumford-Shah functional in dimension $N=2$. In that work, the classical Mumford-Shah functional
$$F(u):=\int_\O |\nabla u|^2\, dx +\mu \HH^1(J_u)$$
is approximated in the sense of $\Gamma$-convergence by a functional of the form
$$F_\e(u):=\int_\O f_\e(\nabla u)\, dx$$
putting a restriction on the functional space on which $F_\e$ is defined. The functional $F_\e$ is discrete in the sense that $u$ is (a scalar-valued) continuous function and piecewise affine on suitable $\e$-dependent meshes (see Definition \ref{def:triangulation}). It consists in an adaptive finite element approximation because there is an implicit mesh optimization whose numerical implementation has been carried out in \cite{BoCh}. The function $f_\e(\nabla u)$ takes the form $\frac1\e f(\e|\nabla u|^2)$ where $f$ is a nondecreasing function satisfying the standard properties \eqref{eq:f}. Typical examples of functions $f$ are, on the one hand the $\arctan$ function (as e.g. in \cite{Go} following a conjecture of De Giorgi) and, on the other hand $f(t)=t\wedge \kappa$. The main feature of this result is that, allowing the mesh to move gives enough flexibility to approximate isotropic surface energies. The constant $\mu$ appearing in the functional $F$ is explicit and only depends on $\kappa$ and the geometry of the triangulation.

An analogous analysis has been carried out in \cite{N1}, where the author constraints the mesh to be made either of equilateral triangles, or of right isosceles ones. In that case, the result is that the functional $F_\e$ $\Gamma$-converges to an anisotropic version of the Mumford-Shah functional
$$\int_\O |\nabla u|^2\, dx + \int_{J_u} \phi(\nu_u)\, d\HH^1,$$
for some function $\phi:\mathbb S^1 \to \R$, which can be explicitly computed, depending on the normal $\nu_u$ to the jump set $J_u$. In \cite{N2}, the same problem is addressed in the two-dimensional vectorial setting. If $f_\e$ is as before, the following approximating energy is considered
$$\int_\O  f_\e(e(u))\, dx.$$
As in \cite{N1}, the $\e$-dependent mesh is fixed and made either of equilateral triangles, or of right isosceles triangles, and the result is that this functional $\Gamma$-converges to an anisotropic version of the Griffith functional
$$\int_\O |e(u)|^2\, dx + \int_{J_u} \phi(\nu_u)\, d\HH^1,$$
where $\phi:\mathbb S^1 \to \R$ is as in \cite{N1}. Note that if $f(t)=t\wedge \kappa$, then
\begin{equation}\label{eq;euvsgu}
f_\e(e(u)) = 
\begin{cases}
\e |e(u)|^2 & \text{ if }\e |e(u)|^2 \leq \kappa,\\
\kappa & \text{ if }\e |e(u)|^2>\kappa.
\end{cases}
\end{equation}

In order to recover the isotropic Griffith energy \eqref{eq:isoG},  a similar approximation result is considered in \cite{N} where, now, the meshes are allowed to move as in \cite{CDM}, but the function $f_\e$ now depends on the full gradient $\nabla u$ (instead of the symmetric gradient) and behaves like
\begin{equation}\label{eq:negri}
f_\e(\nabla u) \sim 
\begin{cases}
\e |e(u)|^2 & \text{ if }\e |\nabla u|^2 \leq \kappa,\\
\kappa & \text{ if }\e |\nabla u|^2>\kappa
\end{cases}
\end{equation}
(compare with \eqref{eq;euvsgu}). In that case, the analysis of \cite{CDM} can be adapted to show a $\Gamma$-convergence result towards the isotropic Griffith energy \eqref{eq:isoG} with the same geometric multiplicative constant $\mu$ (as in \cite{CDM}) in front of the surface energy.

\subsection{Our result}

The objective of the present work is to generalize the previous results in the two-dimensional vectorial case to show an analogous statement as in \cite{CDM}, namely an adaptive discrete finite element approximation of the isotropic Griffith functional. To state precisely our main result, Theorem \ref{thm:GSBD}, we need to introduce some notation (we refer to Section \ref{sec:2} regarding functional spaces).

\medskip

Let $\O$ be a bounded open set of $\R^2$ with Lipschitz boundary. As in \cite{CDM}, we introduce the following class of admissible meshes.
\begin{defn}\label{def:triangulation}
A triangulation of $\O$ is a finite family of closed triangles intersecting $\O$, whose union contains $\O$, and such that, given any two triangles of this family, their intersection, if not empty, is exactly a vertex or an edge common to both triangles. Given some angle $\theta_0$ with $0<\theta_0 \leq 45^\circ-\arctan(1/2)$, and a function $\e \mapsto \omega(\e)$ with $\omega(\e) \geq 6\e$ for any $\e>0$ and $\lim_{\e \to 0^+} \omega(\e)=0$, we define, for any $\e>0$
$$\mathcal T_\e(\O):=\mathcal T_\e(\O,\omega,\theta_0)$$
as the set of all triangulations of $\O$ made of triangles whose edges have length between $\e$ and $\omega(\e)$, and whose angles are all greater than or equal to $\theta_0$. Then we consider the finite element space $V_\e(\O)$ of all continuous functions $u : \O \to \R^2$ for which there exists $\mathbf T \in \mathcal T_\e(\O)$ such that $u$ is affine on each triangle $T \in \mathbf T$.
\end{defn}

\begin{rem}
Imposing $\theta_0 > 0$ and $\omega(\e) \geq \e$ corresponds to a non-flatness condition that ensures the existence of a radius $\varrho(\theta_0) > 0$ such that for all triangle $T \in \mathbf T $, one can find a point $x \in T$ such that
$$
\bar{ B_\varrho (x)} \subset T.
$$ 
As for the conditions $\theta_0 \leq 45^\circ-\arctan(1/2)$ and $\omega(\e) \geq 6 \e$, we will later see that they are crucial to prove the existence of recovery sequences. Indeed, we use the same optimal triangulation introduced in \cite[Appendix A]{CDM}, where the authors' explicit construction makes use of triangles with edges of length $6 \e$ and angles equal to $45^\circ-\arctan(1/2)$.
\end{rem}

Let us consider a  nondecreasing continuous function $f : [0,+\infty) \to [0,+\infty)$ satisfying
\begin{equation}\label{eq:f}
f(0)=0, \quad \lim_{t \to 0^+} \frac{f(t)}{t} = 1 \quad\text{ and } \quad \lim_{t \to \infty} f(t) = \kappa,
\end{equation}
for some constant $\kappa>0$, and a symmetric fourth order tensor $\mathbf A \in \mathscr L(\mathbb M^{2 \times 2}_{\rm sym},\mathbb M^{2 \times 2}_{\rm sym})$ such that
\begin{equation}\label{eq:A}
\alpha |\xi|^2 \leq \mathbf A \xi:\xi \leq \beta |\xi|^2 \quad \text{ for all }\xi \in \mathbb M^{2 \times 2}_{\rm sym},
\end{equation}
for some constants $\alpha$, $\beta>0$. 

\medskip

Our main result is the following $\Gamma$-convergence approximation of the Griffith functional.

\begin{thm}\label{thm:GSBD}
The functional $\F_\e:L^0(\O;\R^2)  \to [0,+\infty]$ defined by
\begin{equation}\label{eq:F_e}
\F_\e(u)=
\begin{cases}
\ds \frac{1}{\e} \int_\O  f \big(\e \mathbf A e(u):e(u) \big) \, dx& \text{ if } u \in V_{\e} (\O),\\
 + \infty & \text{ otherwise}
\end{cases}
\end{equation}
$\Gamma$-converges, with respect to the $L^0(\O;\R^2)$-topology of convergence in measure, to the Griffith functional
$\mathcal F:L^0(\O;\R^2) \to [0,+\infty]$ given by
$$\F(u)=
\begin{cases}
\ds  \int_\O \mathbf A e(u):e(u) \, dx + \kappa \sin\theta_0 \HH^1(J_u)& \text{ if } u \in GSBD^2(\O), \\
+\infty & \text{ otherwise.}
\end{cases}$$
\end{thm}

\begin{rem}\label{rem:f}
As explained above, a meaningful choice is the function $f (t) = t \wedge \kappa$, for which the energy reduces to
$$\int_\O \frac{\kappa}{\e} \wedge \mathbf  A e(u):e(u)  \, dx.$$
It corresponds to the brittle damage energy of a linearly elastic material composed of two phases: an undamaged one whose elasticity coefficients are represented by the Hooke tensor $\mathbf A$, and a damaged one whose elasticity coefficients are set to $0$. The constant $\kappa/\e$ stands for the toughness of the material whose diverging character as $\e \to 0$ forces the damaged zones to concentrate on vanishingly small sets (see \cite{BIR}).
\end{rem}

\subsection{Strategy of proof}

As usual in $\Gamma$-convergence, the proof is achieved by combining a compactness result, a lower bound and an upper bound inequality. In order to describe our argument, let us assume for simplicity that $f(t)=t \wedge \kappa$ and $\mathbf A={\rm id}$.

\medskip

Our compactness result, Proposition \ref{prop:Comp} rests on the general $GSBD$ compactness result of \cite{CCminimization}. Given a sequence $\{u_\e\}_{\e>0}$ with uniformly bounded energy, one can apply \cite[Theorem 1.1]{CCminimization} to the modified function $v_\e:=u_\e \mathds{1}_{\{|e(u_\e)|^2 \leq \kappa/\e\}}$ which consists in putting the value zero inside each triangle $T$ where the (symmetric) gradient of $u_\e$ is ``large''. It might thus create a jump on the boundary of $T$ whose perimeter can be estimated by $\LL^2(T)/\e$. It leads to compactness in measure for the sequence $\{u_\e\}_{\e>0}$  (up to subtracting a sequence of piecewise rigid motions, leaving the energy unchanged), which thus justifies why it is natural to consider $\Gamma$-convergence with respect to this topology. 

\medskip

The upper bound causes no particular difficulty. It consists in using known density results in $GSBD^2(\O)$ (see \cite{CCdensity,CT}) to reduce to the case where the jump set of $u$ is made of finitely many pairwise disjoint closed line segments, and $u$ is smooth outside. Then, considering a similar optimal triangulation of $\O$ as in \cite{CDM} (whose vertices do not cross the jump set) and a piecewise affine Lagrange interpolation of $u$, it leads to the desired upper bound (see Proposition \ref{upper bound}).

\medskip

The proof of the lower bound inequality is much more delicate to address and it represents, to our opinion, the main achievement of this work. 
First of all, the blow-up method allows one to identify separately the bulk part and the singular part. The bulk part can be easily recovered by modifying $u_\e$ into a new function which vanishes in all triangles where $e(u_\e)$ is ``too large'' as in the compactness argument (see Proposition \ref{prop:leb}). The main difficulty is to get a lower bound for the singular part of the energy.

Before describing our strategy of proof, let us briefly explain why the methods of \cite{CDM} (and similarly \cite{BoCh}) fail in our situation. The idea of \cite{CDM} consists in modifying every minimizing sequence $\{u_\e\}_{\e>0}$ inside each triangle $T$ of the associated triangulation $\mathbf T^\e \in \mathcal T_\e(\O)$ according to its variations along each edge of $T$. It rests on the introduction of a jump criterion which stipulates that if the variation of $u_\e$ is large enough, it is convenient to create a jump along the edge. More precisely, if $x_1$, $x_2$ and $x_3$ stand for the vertices of the triangle $T$, it will be energetically favorable to create a jump at the middle point of the segment $[x_i,x_j]$ if $$\frac{|u_\e(x_i)-u_\e(x_j)|}{|x_i-x_j|}>\frac{\sigma}{\sqrt\e},$$
for some constant $\sigma>0$, while $u_\e$ remains unchanged on $[x_i,x_j]$ otherwise. This criterion has to be defined in such a way that:
\begin{itemize}
\item[(i)] the new function, say $w_\e$, has a jump set in each triangle $T$ which satisfies $\HH^1(J_{w_\e} \cap T) \leq \LL^2(T)/(\e\sin\theta_0)$, where $\theta_0$ is as in Definition \ref{def:triangulation}, and $w_\e$ does not jump across $\partial T$;
\item[(ii)] the absolutely continuous part of the gradient, $\nabla w_\e$, is controlled in $L^2(T)$ by the energy restricted to $T$.
\end{itemize}
This construction ensures that the variation of the new discontinuous and piecewise affine function $w_\e$ is always controlled along at least two edges of each triangle $T$, and it yields a control of the full gradient $\nabla w_\e$ of $w_\e$ inside $T$. In \cite{CDM}, this is possible thanks of the scalar nature of the problem because the gradient $\nabla{w_\e}_{|T}$ is a (constant) vector in $\R^2$ (see \cite[Remark 3.5]{CDM}).

In our case, $u_\e$ is not scalar-valued anymore, but vector-valued and the energy only depends on its symmetric gradient $e(u_\e)$. If one uses the same criterion than in \cite{CDM}, then condition (i) above will be satisfied for the new function $w_\e$ on $T$. However, one will only be able to estimate the $L^2$-norm of the (symmetric) gradient of $w_\e$ by that of the full gradient of $u_\e$ which, unfortunately, is not controlled by the energy $\F_\e(u_\e)$. Note that in \cite{N}, such a control is artificially made possible thanks to the particular form of the energy (see \eqref{eq:negri} above). This is however not natural in this linearized elasticity setting where the energy should be expressed in terms of the symmetric gradient of the displacement.

As a consequence, the jump criterion has to be modified. As the energy only depends on the symmetric part of the gradient of $u_\e$, it would be natural to consider a criterion involving the longitudinal variation of $u_\e$ along the edges of the triangle instead of the full variation. In other words, one could modify the criterion by asking that 
if 
$$\frac{|(u_\e(x_i)-u_\e(x_j))\cdot (x_i-x_j)|}{|x_i-x_j|^2}>\frac{\sigma}{\sqrt\e},$$
then we create a jump at the middle point of $[x_i,x_j]$, while $u_\e$ remains unchanged on $[x_i,x_j]$ otherwise. In that case, it is again not possible to control the symmetric gradient $e(w_\e)$ of the new function $w_\e$ by that of $u_\e$. Indeed, in a similar way as in \cite{CDM}, the previous criterion ensures that the longitudinal variation of $w_\e$ along at least two edges of each triangle is controlled by the energy restricted to $T$. If we call $\xi_1$ and $\xi_2 \in \mathbb S^1$ both (linearly independent) directions associated to these ``good'' edges, it shows that $e(w_\e)_{|T}:(\xi_1 \otimes \xi_1)$ and $e(w_\e)_{|T}:(\xi_2 \otimes \xi_2)$ are controlled by $e(u_\e)_{|T}$ which is not enough to control the full $2 \times 2$ symmetric matrix $e(w_\e)_{|T}$ which has three degrees of freedom. In addition, some (uncontrolled) discontinuities can also be created at the interface $I:=\partial T \cap \partial T'$ between two adjacent triangles $T$ and $T'$ so that condition (i) fails as well.

Overcoming these difficulties seems to be a very serious issue so that we decided to attack this problem from a different angle. First of all, the use of the blow-up method allows one to reduce to the case where $\O=B$ is the unit ball, $u$ is a step function of the form
$$u(x)=
\begin{cases}
a & \text{ if }  x\cdot \nu <0,\\
b & \text{ if }  x\cdot \nu >0,
\end{cases}$$
for some $a$, $b \in \R^2$ with $a \neq b$ and $\nu \in \mathbb S^1$ (with a jump set corresponding to the diameter of $B$ orthogonal to $\nu$), and, see Lemma \ref{lemma afterSections}, such that
\begin{equation}\label{eq:convzero}
\int_{\{|e(u_\e)|^2\leq \kappa/\e\}} |e(u_\e)|^2\, dx \to 0.
\end{equation}
To make our strategy of proof more transparent, we assume that $a \cdot \nu\neq b \cdot \nu$. A standard argument based on Fubini's Theorem shows that the one-dimensional section of $u_\e$ in the direction $\nu$ passing through the point $y$, namely $t \mapsto (u_\e)_y^\nu(t):=u_\e(y+t\nu)\cdot \nu$ converges (in measure) to the step function 
$$t \mapsto u_y^\nu(t)=a\cdot \nu \mathds{1}_{\R^-}+b\cdot \nu \mathds{1}_{\R^+}.$$ 
Let us denote by $\mathbf T^\e$ the triangulation on which $u_\e$ is (continuous and) piecewise affine. We further denote by $\mathbf T^\e_b$ the familly of all triangles $T \in \mathbf T^\e$ such that  $|e(u_\e)_{|T}|^2 > \kappa/\e$. Thanks to \eqref{eq:convzero}, we show that almost every line orthogonal to $J_u \cap B$ must cross at least one triangle $T\in \mathbf T^\e_b$ (see Lemma \ref{lemma1}). The reason is that if, for some $y \in J_u \cap B$, the line $y+\R \nu$ intersects no such triangles, then $(u_\e)_y^\nu$ would be bounded in $H^1$ (because $|((u_\e)_y^\nu)'| \leq |e(u_\e)(y+t\nu)|$) and thus, it would converge weakly in that space to a constant function, contradicting that $a \cdot\nu \neq b\cdot \nu$. This information allows one to get a bad lower bound for the surface energy with $1/2$ multiplicative factor. It suggests to improve the previous argument by showing that ``many'' lines $y+\R\nu$ passing through $y \in J_u \cap B$ must actually cross at least two triangles in $\mathbf T^\e_b$, which is the object of Lemma \ref{lemma2}. To do that, we show in Lemma \ref{lemma3} that there are very few points $y$ in $J_u \cap B$ such that the line $y+\R\nu$ crosses exactly one triangle $T \in \mathbf T^\e_b$. Indeed, in that case, up to a small error, the function $(u_\e)_y^\nu$ would have to pass from the value $a \cdot \nu$ to $b\cdot \nu$ inside $T$. Due to the particular shape of a triangle and of the fact that $u_\e$ is affine inside $T$, this could only happen for at most two values of $y$. Moreover, if $y$ is far away from these two values, the variation of $(u_\e)_y^\nu$ across the triangle $T$ is not sufficient, and it becomes necessary to cross an additional triangle $T'$ in $\mathbf T^\e_b$. With this improvement, we can now construct two disjoint families of triangles with the property that both families project onto $J_u \cap B$ into two sets of almost full $\HH^1$ measure (see Lemma \ref{lemma F_1 F_2}). It enables one to compensate the bad multiplicative factor $1/2$ in the previous argument, and obtain the expected lower bound with the correct constant corresponding to $\kappa \sin\theta_0$ (see Proposition \ref{prop Jump part}). In \cite{AT}, the right factor in the lower estimate of the jump part comes from the two transitions of the phase field approximating $J_u$, one from each side of the jump set. Similarily here, the same role is played by the characteristic function $\chi_\e = \mathds{1}_{\left\{ \left\lvert e(u_\e) \right\rvert^2 > \kappa / \e \right\}}$. Indeed, having in mind the optimal triangulation computed in the upper bound (see \cite[Appendix A]{CDM}) and knowing that almost every line orthogonal to $J_u \cap B$ crosses at least two distinct triangles of $\mathbf T^\e_b$, we expect these triangles to form a neighbourhood of $J_u$ as in Figure \ref{fig:intro}. Therefore, $\LL^1 \left( \left\{ \left( \chi_\e \right)^\nu_y  = 1 \right\} \right) = \e \sin \theta_0$ which leads to the right lower bound, independently of the amplitude of the jump $(b-a) \cdot \nu$.

\begin{figure}[hbtp]
\begin{tikzpicture}[x=3.0cm,y=3.0cm]
\clip(-2.5,-1.9) rectangle (3,1.2);
\draw [line width=1.2pt] (0.,0.) circle (1.0016155954677715);
\draw [->,line width=1.2pt] (0.,0.) -- (1.0016155954677715,0.);
\draw (1.035114496479055,0.05) node[anchor=north west] {$ \nu$};
\draw (-1.138077206481706,-0.4747998746194317) node[anchor=north west] {$ \mathbf{ B}$};
\draw [color=ffqqqq](-0.4154642018026352,1.2434132698396734) node[anchor=north west] {$ \mathbf{J_u \cap B = \Pi_\nu \cap B}$};
\draw [line width=1.pt] (0.06,0.96)-- (-0.06,0.84);
\draw [line width=1.pt] (-0.06,0.84)-- (0.06,0.72);
\draw [line width=1.pt] (0.06,0.72)-- (-0.06,0.6);
\draw [line width=1.pt] (-0.06,0.6)-- (0.06,0.48);
\draw [line width=1.pt] (0.06,0.48)-- (-0.06,0.36);
\draw [line width=1.pt] (-0.06,0.36)-- (0.06,0.24);
\draw [line width=1.pt] (0.06,0.24)-- (-0.06,0.12);
\draw [line width=1.pt] (-0.06,0.12)-- (0.06,0.);
\draw [line width=1.pt] (0.06,0.)-- (-0.06,-0.12);
\draw [line width=1.pt] (-0.06,-0.12)-- (0.06,-0.24);
\draw [line width=1.pt] (0.06,-0.24)-- (-0.06,-0.36);
\draw [line width=1.pt] (-0.06,-0.36)-- (0.06,-0.48);
\draw [line width=1.pt] (0.06,-0.48)-- (-0.06,-0.6);
\draw [line width=1.pt] (-0.06,-0.6)-- (0.06,-0.72);
\draw [line width=1.pt] (0.06,-0.72)-- (-0.06,-0.84);
\draw [line width=1.pt] (-0.06,-0.84)-- (0.06,-0.96);
\draw [line width=1.pt,color=ffqqqq] (0.,1.0016155954677715)-- (0.,-1.0016155954677715);
\draw [color=qqqqff](-0.17994588916649362,-1.15) node[anchor=north west] {$ \varepsilon \sin \theta_0 $};
\draw [line width=1.pt] (-0.9085041183316239,0.42172748079623285)-- (0.9184164051054441,0.399681257901157);
\draw [line width=1.pt] (-0.9085041183316239,-1.810043134063613)-- (-0.06,-1.8202823922470703);
\draw [line width=1.pt] (0.06,-1.8217304830170136)-- (0.918416405105444,-1.8320893569586891);
\draw [line width=1.pt] (-0.06,-1.6859428023063017)-- (0.06,-1.6873908930762447);
\draw [line width=1.pt] (-0.06,-1.6859428023063017)-- (-0.06,-1.8202823922470703);
\draw [line width=1.pt] (0.06,-1.6873908930762447)-- (0.06,-1.8217304830170136);
\draw [line width=1.pt] (0.06,0.96)-- (0.06,-0.96);
\draw [line width=1.pt,color=qqqqff] (0.06,-0.96)-- (-0.06,-0.96);
\draw [line width=1.pt,color=qqqqff] (0.06,0.96)-- (-0.06,0.96);
\draw [line width=1.pt] (-0.06,-0.96)-- (-0.06,0.96);
\draw [line width=1.2pt,color=qqqqff] (-0.065,-1.393)-- (0.065,-1.393);
\draw [line width=1.pt,color=qqqqff] (-0.062,-1.393)-- (-0.045,-1.365);
\draw [line width=1.pt,color=qqqqff] (-0.062,-1.393)-- (-0.045,-1.421);
\draw [line width=1.pt,color=qqqqff] (0.062,-1.393)-- (0.045,-1.365);
\draw [line width=1.pt,color=qqqqff] (0.062,-1.393)-- (0.045,-1.421);
\draw (-0.9,-1.6) node[anchor=north west] {$\left( \chi_\varepsilon \right)^\nu_y = 0$};
\draw (0.4,-1.63) node[anchor=north west] {$0$};
\draw (-0.08,-1.53) node[anchor=north west] {$ 1$};
\draw (0.954824162625825,0.5315056430077078) node[anchor=north west] {$\mathbf{ B}^\nu_y$};
\draw (-0.222767400554883,0.40839379776608964) node[anchor=north west] {$ y$};
\begin{scriptsize}
\draw [fill=uuuuuu] (0.,0.) circle (1.2pt);
\draw [fill=uuuuuu] (0.,0.41) circle (1.2pt);
\end{scriptsize}
\end{tikzpicture}
\caption{}
\label{fig:intro}
\end{figure}

\medskip

To conclude this introduction, let us mention that the originality of this work is twofold. First of all, we are able to provide a deterministic discrete finite element approximation result of the Griffith functional with isotropic surface energies. In particular, our approach does not require any unnatural dependence of the approximating energy with respect to the skew symmetric part of the gradient (in the context of linear elasticity) nor the use of stochastic meshes. Second, our method relies on an unusual application of the slicing method, which is rather employed in $\Gamma$-convergence analysis to reduce the dimension of the problem to a one-dimension study. Here, we instead use this method as a tool to enumerate in a non trivial way the number of triangles needed to derive the correct multiplicity in the surface energy.

\subsection{Organisation of the paper} 

In Section \ref{sec:2}, we collect useful notation and preliminary results that will be useful in the subsequent sections. Section 3 is devoted to show our main result, Theorem \ref{thm:GSBD}. It is divided into three parts: a first one consisting in a compactness result, Proposition \ref{prop:comp}, a second one corresponding to the lower bound inequality, Proposition \ref{lower bound}, and a last one for the upper bound inequality, Proposition \ref{upper bound} through the construction of a recovery sequence. Eventually, in Section \ref{sec:4}, we extend the previous $\Gamma$-convergence analysis allowing for Dirichlet boundary conditions formulated in a suitable way at the discrete and continuum levels (see Theorem \ref{prop:Gamma_CV_Dirichlet}). We  then deduce the fundamental property of $\Gamma$-convergence, Corollary \ref{thm:CV_minimizers}, in our specific setting, i.e., the convergence of minimizers as well as the minimum value.

\section{Notation and preliminaries}\label{sec:2}

\noindent \textbf{Vectors.} The Euclidean scalar product between two vectors $x$ and $y \in \R^n$ is denoted by $x \cdot y$, and the associated Euclidean norm by $|x|:=\sqrt{x\cdot x}$. For $x \in \R^n$ and $\varrho>0$, we denote by $B_\varrho(x):=\{y \in \R^n : \, |x-y|<\varrho\}$ the open ball centered at $x$ with radius $\varrho$. If $x=0$, we simply write $B_\varrho$ instead of $B_\varrho(0)$. The notation $\mathbb S^{n-1}$ stands for the unit sphere $\partial B_1$.

\medskip

\noindent \textbf{Matrices.} The set of all real $m \times n$ matrices is denoted by $\mathbb M^{m \times n}$, and the subset of symmetric real $n \times n$ matrices by $\mathbb M^{n \times n}_{\rm sym}$. It will be endowed with the Froebenius scalar product $A:B:={\rm tr}(A^TB)$ and the associated Froebenius norm $|A|:=\sqrt{A:A}$. 

Given two vectors $a$ and $b \in \R^n$, the tensor product between $a$ and $b$ is defined as $a \otimes b:= a b^T \in \mathbb M^{n \times n}$ and the symmetric tensor product by $a \odot b:=(a \otimes b + b \otimes a)/2 \in \Ms$.

\medskip

\noindent \textbf{Measures.} The Lebesgue and the $k$-dimensional Hausdorff measures in $\R^n$ are  respectively denoted by $\LL^n$ and $\HH^k$. If $U$ is a bounded open set of $\R^n$ and $Y$ is an Euclidean space, we denote by $\M(U;Y)$ the space of $Y$-valued bounded Radon measures in $U$ which, according to the Riesz Representation Theorem, can be identified to the dual of $C_0(U;Y)$ (the closure of $C_c(U;Y)$ for the sup-norm in $U$). For $\mu \in \M(U;Y)$, its total variation is denoted by $|\mu|$.

\medskip

\noindent \textbf{Functional spaces.}  We use standard notation for Lebesgue and Sobolev spaces. 
If $U$ is a bounded open subset of $\R^n$, we denote by $L^0(U;\R^m)$ the set of all $\LL^n$-measurable functions from $U$ to $\R^m$. We recall that a sequence $\{g_k\}_{k \in \N}$ in $L^0(U;\R^m)$ converges in measure to $g \in L^0(U;\R^m)$ if for all $\e > 0$, 
$$\LL^n \left( \left\{ x \in U: \, \left\lvert g_k(x) - g(x) \right\rvert > \e \right\} \right)\to 0.$$
Note that, for any fixed constant $M > 0$, we can define the following mapping 
\begin{equation}\label{eq:d_M}
d_M : (g,h) \in L^0(U;\R^m) \times L^0(U;\R^m) \mapsto \int_U M \wedge \left\lvert g-h \right\rvert \, d x \, \in \R ^+
\end{equation}
which turns out to be a distance over $L^0(U;\R^m)$, with the property that $g_k$ converges in measure to $g$ if and only if $d_M(g_k,g)\to 0$. It confers to $L^0(U;\R^m)$ a metric space structure.

\medskip

\noindent\textbf{Functions of bounded variation and sets of finite perimeter.} We refer to \cite{AFP} for an exhaustive treatment on that subject and just recall few notation. Let $U \subset \R^n$ be a bounded open set. A function $u \in L^1(U;\R^m)$ is a {\it function of bounded variation} in $U$, and we write $u \in BV(U;\R^m)$, if its distributional derivative $Du$ belongs to $\M(U;\mathbb M^{m \times n})$. We use standard notation for that space, referring to \cite{AFP} for details. We just recall that a function $u$ belongs to $SBV^2(U;\R^m)$ if $u \in SBV(U;\R^m)$ (the distributional derivative $Du$ has no Cantor part), its approximate gradient $\nabla u$ belongs to $L^2(U;\mathbb M^{m\times n})$ and its jump set $J_u$ satisfies $\HH^{n-1}(J_u)<\infty$.

A Lebesgue measurable set $A \subset \R^n$ is a {\it set of finite perimeter} in $U$ if its characteristic function $\mathds{1}_A$ belongs to $BV(U;\R^n)$. The reduced boundary of $A$ is denoted by $\partial^* A$ and the essential (or measure theoretic) boundary is denoted by $\partial_* A$. For every $t \in [0,1]$, we denote by $A^{(t)}$ the set of points where $A$ has density $t$. 

We also recall that a partition $\mathcal P=\{P_i\}_{i \in \N}$ of an open set $U$ is a {\it Cacciopoli partition} if each $P_i$ have finite perimeter in $U$, and $\sum_{i \in \N} |D\mathds{1}_{P_i}|(U)<\infty$. In that case, 
$$\bigcup_{i \in \N} (P_i)^{(1)} \cup \bigcup_{i,j \in \N, \, i \neq j} \partial^* P_i \cap \partial^* P_j$$
contains $\HH^{n-1}$-almost all of $U$ (see \cite[Section 4.4]{AFP}). In the sequel (as in \cite[Theorem 2.5]{CCminimization}), we will sometimes use the following notation for Caccioppoli partitions:
$$\mathcal P ^{(1)} := \bigcup_{i \in \N} {P_i}^{(1)}, \quad \partial^* \mathcal P := \bigcup_{i \in \N} \partial^* P_i.$$

\medskip

\noindent\textbf{(Generalized) functions of bounded deformation.} A function $u \in L^1(U;\R^n)$ is a {\it function of bounded deformation}, and we write $u \in BD(U)$, if its distributional symmetric gradient $Eu:=(Du+Du^T)/2$ belongs to $\M(U;\Ms)$. We refer to \cite{Suquet,Temam,ST1,ACDM,BCDM} for the main properties and notation of that space. The space $SBD^2(U)$ is made of all functions $u \in SBD(U)$ ($Eu$ has no Cantor part) such that the approximate symmetric gradient $e(u)$ (the absolutely continuous part of $Eu$ with respect to $\LL^n$) belongs to $L^2(U;\mathbb M^{n\times n}_{\rm sym})$ and its jump set $J_u$ satisfies $\HH^{n-1}(J_u)<\infty$.

\medskip

We now recall the definition and the main properties of the space of {\it generalized functions of bounded deformation} introduced in \cite{DM2}. We first need to introduce some notation. Let $\xi \in \mathbb S^{n-1}$, we denote by $\Pi_\xi:=\{y \in \R^n : \; y \cdot \xi=0\}$ the orthogonal space to $\xi$ and by $p_\xi$ the orthogonal projection onto $\Pi_\xi$. For every set $B \subset \R^n$, we define  for $\xi \in \mathbb S^{n-1}$ and $y \in \R^n$, 
$$B_y^\xi:=\{t \in \R : \; y+t\xi \in B\}, \quad B^\xi:=p_\xi(B)$$
and, for every (vector-valued) function $u:B \to \R^n$ and (scalar-valued) function $f:B \to \R$,
$$u_y^\xi(t):=u(y+t\xi)\cdot\xi, \quad f_y^\xi(t)=f(y+t\xi)\quad \text{ for all } y \in \R^n \text{ and all } t \in B_y^\xi.$$

\begin{defn}
Let $U \subset \R^n$ be a bounded open set and $u \in L^0(U;\R^n)$. Then, $u \in GBD(U)$ if there exists a nonnegative measure $\lambda \in \M(U)$ such that one of the following equivalent conditions holds true for every $\xi \in \mathbb S^{n-1}$:
\begin{enumerate}
\item for every $\tau \in C^1(\R)$ with $-\frac12 \leq \tau \leq \frac12$ and $0 \leq \tau' \leq 1$, the partial derivative $D_\xi(\tau(u \cdot \xi))=D(\tau(u\cdot\xi))\cdot \xi$ belongs to $\M(U)$, and
$$|D_\xi(\tau(u\cdot\xi))|(B) \leq \lambda(B) \quad \text{ for every Borel set }B \subset U;$$

\item $u_y^\xi \in BV_{\rm loc}(U_y^\xi)$ for $\HH^{n-1}$-a.e. $y \in U^\xi$, and$$\int_{\Pi_\xi} \left(|Du_y^\xi|(B_y^\xi \setminus J_{u_y^\xi}^1) + \HH^0(B_y^\xi \cap J_{u_y^\xi}^1) \right)d\HH^{n-1}(y)\leq \lambda(B)\quad \text{ for every Borel set }B \subset U,$$
where $J_{u_y^\xi}^1:=\{t \in J_{u_y^\xi} : \; |[u_y^\xi](t)|\geq 1\}$.
\end{enumerate}
The function $u$ belongs to $GSBD(U)$ if $u \in GBD(U)$ and $u_y^\xi \in SBV_{\rm loc}(U_y^\xi)$ for every $\xi \in \mathbb S^{n-1}$ and for $\HH^{n-1}$-a.e. $y \in U^\xi$.
\end{defn}

Every $u \in GBD(U)$ has an approximate symmetric gradient $e(u) \in L^1(U;\Ms)$ such that for  every $\xi \in \mathbb S^{n-1}$ and for $\HH^{n-1}$-a.e. $y \in U^\xi$,
$$e(u)(y+t\xi)\xi \cdot \xi=(u_y^\xi)'(t) \quad \text{ for $\LL^1$-a.e. $t \in U_y^\xi$.}$$
Moreover, the jump set $J_u$ of $u \in GBD(U)$, defined as the set of all $x_0 \in U$  for which there exist $(u^+(x_0),u^-(x_0),\nu_u(x_0)) \in \R^n \times \R^n \times  \mathbb S^{n-1}$ with $u^+(x_0) \neq u^-(x_0)$ such that the function
$$y\in B_1  \mapsto u_{x_0,\varrho}:=u(x_0+\varrho y)$$
converges in measure in $B_1$ as $\varrho \searrow 0$ to
$$y \in B_1 \mapsto
\begin{cases}
u^+(x_0) & \text{ if }y \cdot \nu_u(x_0)>0,\\
u^-(x_0) & \text{ if }y \cdot \nu_u(x_0)\leq 0,
\end{cases}$$
is countably $(\HH^{n-1},n-1)$-rectifiable. Finally, the energy space $GSBD^2(U)$ is defined as
$$GSBD^2(U):=\{u \in GSBD(U): \; e(u) \in L^2(U;\Ms), \, \HH^{n-1}(J_u)<\infty\}.$$

\section{Proof of the main result}\label{sec:3}

Let us introduce the $\Gamma$-lower and upper limits (with respect to the topology of convergence in measure) $\F'$ and $\F'':L^0(\O;\R^2) \to [0,+\infty]$ defined by
$$\F'(u):=\inf\left\{\liminf_{\e \to 0}\F_\e(u_\e) : \; u_\e \to u \text{ in measure  in } \O \right\},$$
and
$$\F''(u):=\inf\left\{\limsup_{\e \to 0}\F_\e(u_\e) : \; u_\e \to u \text{ in measure in } \O \right\},$$
for all $u \in L^0(\O;\R^2) $. 

\subsection{Domain of the $\Gamma$-limit}
We begin our analysis by identifying the domain of finiteness of the $\Gamma$-limit.

\begin{prop}\label{prop:comp}
Let $\{\e_k\}_{k \in \N}$ satisfying $\e_k \to 0$, $u \in L^0(\O;\R^2)$ and $\{ u_k \}_{k \in \N} \subset L^0(\O;\R^2)$ be such that $M := \sup_k \F_{\e_k}(u_k) < \infty$ and $u_k \to u $ in measure in $\O$. Then, $u \in GSBD^2(\O)$.
\end{prop}

\begin{proof}
According to the properties \eqref{eq:f} satisfied by $f$, for all $\delta >0$, there exists a constant $0<K<\kappa$ such that 
\begin{equation}\label{eq:dependance in delta}
f(t) \geq K \wedge [ (1-\delta) t ]\quad \text{ for all }t \geq 0.
\end{equation}
Indeed, since $f(t)/t \to 1$ as $t \to 0^+$, there exists $t^* >0$ such that $f(t)/t\geq 1 - \delta$ for all $t \in [0,t^*]$ and $K:=(1-\delta)t^* <\kappa$. Hence, for all $t \in [0,t^*]$, we have $f(t) \geq (1 - \delta)t $, while for all $t >t^*$, as $f$ is nondecreasing, $f(t) \geq f(t^*) \geq K$.

\medskip

By definition of $\F_{\e_k}$, there exists a triangulation $\mathbf T^k \in \mathcal T_{\e_k}(\O)$ such that $u_k \in V_{\e_k}(\O)$ is affine on each triangle $T \in \mathbf T^k$. We introduce the characteristic functions 
$$ \chi_k := \mathds{1}_{\left\{ (1-\delta) \mathbf A e(u_k):e(u_k) \, \geq \, \frac{K}{\e_k} \right\} } \in L^\infty (\O; \{ 0,1\} )$$
which are constant on each triangle $T \in \mathbf{T}^k$, so that 
$$D_k :=\{\chi_k=1\}\cap \O=\bigcup_{i=1}^{N_k} (T_i^k \cap \O)$$
 for some triangles $T_i^k \in \mathbf T^k$. Remark that this choice of $\chi_k$ implies that 
$$M\geq \F_{\e_k}(u_k) \geq  (1-\delta) \int_\O (1 - \chi_k) \mathbf A e(u_k):e(u_k) \, dx + \frac{K}{\e_k} \int_\O \chi_k \, dx,$$
forcing $\chi_k$ to converge to $0$ in $L^1(\O)$ since $ 0 \leq \int_\O \chi_k\, dx \leq  K^{-1}M\e_k \to 0$.

\medskip 

Let $v_k:= (1-\chi_k)u_k$ so that,  by \cite[Theorem 3.84]{AFP}, $v_k \in SBV^2(\O;\R^2)$ with $\nabla v_k= (1-\chi_k)\nabla u_k$ and 
$$J_{v_k}{{}}\subset \O \cap \partial D_k \subset \bigcup_{i=1}^{N_k} \partial T_i^k.$$
Note that 
\begin{equation}\label{eq:for simpler compacity}
v_k \to u \text{ in measure in } \O \text{ and } A :=\{ x \in \O : \, \lvert u_k(x) \rvert \to \infty \} \text{ is } \LL^2\text{-negligible}.
\end{equation}
Indeed, since $u_k \to u$ in measure in $\O$ and $\{u_k \neq v_k\} \subset D_k$ with $\LL^2(D_k) \to 0$, for all $\eta > 0$, we get that $\LL^2 \left( \left\{ \left\lvert v_k - u \right\rvert > \eta \right\}  \right) \leq \LL^2 \left( \left\{ \left\lvert u_k - u \right\rvert > \eta \right\} \right) + \LL^2(D_k)\to 0.$ Additionally, up to a subsequence (not relabeled), $u_k(x) \to u(x) \in \R^2$ for $\LL^2$-a.e. $x \in \O$. 

\medskip

On the one hand, using the energy bound $\F_{\e_k}(u_k) \leq M$ and the ellipticity property \eqref{eq:A} of $\mathbf A$, we infer that 
\begin{equation}\label{eq:ev_k bounded}
\int_\O |e(v_k)|^2\, dx \leq \frac{M}{(1-\delta)\alpha}.
\end{equation}
On the other hand, by definition of an admissible triangulation, the edges of each triangle $T_i^k$ have length greater than or equal to $\e_k$ and their angles are all greater than or equal to $\theta_0$, so that the heights of such triangles must be greater than or equal to $\e_k \sin \theta_0$. Therefore, for all $1 \leq i \leq N_k$,
$$ \LL^2(T_i^k) \geq \frac12 (\e_k \sin \theta_0) \frac{\HH^1(\partial T_i^k)}{3} $$
which implies that for all open subset $ U \subset \subset \O$ :
$$\HH^1(J_{v_k} \cap U ) \leq \frac{6}{\sin\theta_0}\sum_{i \in \{ 1,\ldots,N_k \} ,\, T_i^k \cap U \neq \emptyset} \frac{\LL^2(T_i^k)}{\e_k}.$$
Let $k_U \geq 1$ (depending on $U$) be such that for all $k \geq k_U$, any triangle $T \in \mathbf T^k$ intersecting $U$ is contained in $\O$, then it follows that for all $k \geq k_U$,
\begin{equation}\label{eq:Jv_k bounded}
\HH^1(J_{v_k} \cap U ) \leq \frac{6}{\e_k\sin\theta_0} \int_\O \chi_k \, dx \leq \frac{6 M}{K \sin\theta_0},
\end{equation}
where we used once more the energy bound $\F_{\e_k}(u_k) \leq M$. 

\medskip

Gathering \eqref{eq:ev_k bounded} and \eqref{eq:Jv_k bounded}, we can apply the $GSBD^2$-compactness Theorem (\cite[Theorem 1.1]{CC}). Together with \eqref{eq:for simpler compacity}, it ensures the existence of a subsequence (depending on the open subset $U$, which we do not relabel) such that $u_{|U} \in GSBD^2(U)$, $$
 e(v_k)_{|U} \wto e (u_{|U}) \, \text{ weakly in } L^2(U ;\mathbb M^{2 \times 2}_{\rm sym}) \, \text{ and } \, \HH^1(J_u \cap U ) \leq \liminf_{ k\to \infty} \, \HH^1(J_{v_k} \cap U).
$$
We then consider an exhaustion of $\O$ by a sequence of open subsets $\{U_m\}_{m \in \N}$ satisfying $U_m \subset\subset U_{m+1} \subset\subset \O$ for all $m \in \N$ and $\bigcup_m U_m = \O$. Using a diagonal extraction argument, we can find a subsequence (still denoted by $\{v_k \}_{k\in \N}$) such that for all $m \in \N$, $u_{|U_m} \in GSBD^2(U_m)$ and
\begin{equation}\label{eq:conv}
\ds e(v_k)_{|U_m} \wto e (u_{|U_m})   \text{ weakly in } L^2(U_m ;\mathbb M^{2 \times 2}_{\rm sym}) \text{ and }  \HH^1(J_u \cap U_m) \leq\liminf_{ k\to\infty} \HH^1(J_{v_k} \cap U_m).
\end{equation}

\medskip

Let us now check that $u$ belongs to $GSBD^2(\O)$. Indeed, let $\xi \in \mathbb S^1$ and $\tau \in C^1(\R)$ be such that $\left\lvert \tau \right\rvert \leq \frac12$ and $0 \leq \tau' \leq 1$. For all test function $\phi \in C^\infty_c(\O)$, there exists $m \in \N$ such that ${\rm supp} \, \phi \subset U_m$ so that, owing to the dominated convergence Theorem,
$$
\ds \langle D_\xi \left( \tau (u \cdot \xi) \right), \phi \rangle = -\int_{U_m } \tau(u \cdot \xi) D_\xi \phi \, dx = -\lim_{k\to \infty} \int_{U_m} \tau ( v_k \cdot \xi) D_\xi \phi \, dx = \lim_{k\to \infty} \langle D_\xi(\tau ( v_k \cdot \xi) ) , \phi \rangle.
$$
Since $v_k \cdot \xi \in SBV^2(\O)$, using the chain rule formula in $BV$ (\cite[Theorem 3.96]{AFP}), we get that $\tau ( v_k \cdot \xi) \in SBV^2(\O)$ with
$$
D_\xi\left( \tau ( v_k \cdot \xi) \right) = \tau'(v_k \cdot \xi) e(v_k) : (\xi \otimes \xi) \, \LL^2 \res \O + \left( \tau( v_k^+ \cdot \xi) - \tau( v_k^- \cdot \xi) \right) \nu_{v_k} \cdot \xi \, \HH^1\res J_{v_k}.
$$
Taking the variation, we infer that
$$\left|D_\xi\left( \tau ( v_k \cdot \xi) \right)\right| \leq |e(v_k)|\LL^2\res \O + \HH^1\res J_{v_k} =:\lambda_k.$$ 
As a consequence of \eqref{eq:ev_k bounded} together with \eqref{eq:Jv_k bounded}, the sequence $\{\lambda_k\}_{k \in \N}$  is bounded in $\M(U_m)$ for all $m \in \N$, with 
$$ 
\sup_{k \geq k_{U_m}} \lambda_k(U_m) \leq  \frac{M}{(1-\delta) \alpha} + \frac{6 M }{K \sin \theta_0} =: M_\delta < + \infty ,
$$
so that, up to a further diagonal extraction, $\lambda_k  \res  U_m \wto \lambda^{ (m)}$ weakly* in $\M( \O)$ for some nonnegative measure $\lambda^{(m)} \in \M( \O)$ satisfying, for all $m \in \N$,
$$ 
\lambda^{ (m)}( \O) \leq \liminf_{ k\to \infty} \lambda_k(U_m) \leq M_\delta.
$$
Therefore, we can introduce the following nonnegative measure $\lambda \in \M(\O)$ defined  by
$$ \lambda (B) := \sup_{m\in \N} \lambda^{ (m)}( B)=\lim_{m \to \infty}\lambda^{ (m)}( B) \text{ for all Borel subset } B \subset \O.$$
We thus obtain that 
$$ \lvert \langle D_\xi \left( \tau (u \cdot \xi) \right), \phi \rangle \rvert \leq \lim_{k \to \infty} \langle \lambda_k  \res U_m, \lvert \phi \rvert \rangle = \langle \lambda^{ (m)}, \lvert \phi \rvert \rangle  \leq \langle \lambda, \lvert \phi \rvert \rangle ,$$
implying both that $D_\xi \left( \tau (u \cdot \xi) \right) \in \M(\O)$ according to Riesz Representation Theorem and that
$$|D_\xi( \tau (u \cdot \xi) )| \leq \lambda \quad \text{ in }\M(\O),$$
which shows that $u \in GBD(\O)$. Using next that $u \in GSBD(U_m)$ for all $m \in \N$ and \cite[Definition 4.2]{DM2}, we deduce that $u \in GSBD(\O)$. Eventually, by locality of the definition of the approximate symmetric gradient $e(u)$ (see \cite[Formula (9.1)]{DM2}), as a consequence of \eqref{eq:ev_k bounded} together with \eqref{eq:conv}, we infer that $e(v_k) \wto e(u)$ weakly in $L^2(\O;\mathbb M^{2 \times 2}_{\rm sym})$ with $e(u) \in L^2(\O;\mathbb M^{2 \times 2}_{\rm sym})$. Passing to the limit as $m \to \infty$ in the last property of \eqref{eq:conv} and using \eqref{eq:Jv_k bounded} shows that $\HH^1(J_u)<\infty$. All of this establishes that $u \in GSBD^2(\O)$ and completes the proof of the Proposition.
\end{proof}

\begin{rem} We will later improve the previous result (see Proposition \ref{prop:Comp}) by getting rid-off the a priori knowledge that $u_k$ converges in measure in $\O$. The price to pay will be to subtract a sequence of piecewise rigid body motions. Proposition \ref{prop:Comp} will a posteriori justify why the topology of convergence in measure is the natural one to address the $\Gamma$-convergence analysis.
\end{rem}

\subsection{The lower bound}

The proof of the lower bound inequality relies on the blow up method which consists in identifying separately the Lebesgue and jump parts of the energy.

\begin{prop}\label{lower bound}
For all $ u \in L^0(\O;\R^2) $, 
$$\F(u) \leq \F'(u).$$
\end{prop}

\begin{proof}
Without loss of generality, we can assume that $\F'(u) <\infty$. For any $\zeta > 0$, there exists a sequence $\{u_\e\}_{\e >0}$ such that $u_\e \to u $ in measure in $\O$ and 
$$ \underset{\e \to 0}{\liminf} \, \F_{\e}(u_\e) \leq \F' (u) + \zeta  < \infty .$$
Let us extract a subsequence $\{u_k\} _{k \in \N} := \{u_{\e_k}\}_{k\in \N}$ from $\{u_\e\}_{\e >0}$ such that $u_k \to u$ $\LL^2$-a.e. in $\O$ and
$$\underset{k \to\infty}{\lim} \, \F_{\e_k} (u_k) = \underset{\e \to 0}{\liminf} \, \F_\e (u_\e) < \infty .$$
This implies that, for $k$ large enough, $u_k \in V_{\e_k} (\O)$ and $\sup_k \F_{\e_k} (u_k)< \infty$. By definition of the finite element space $V_{\e_k}(\O)$, there exists a triangulation $\mathbf T^k \in \mathcal T_{\e_k}(\O)$ such that $u_k$ is affine on each $T \in \mathbf T^k$.

\medskip

We first note that, according to Proposition \ref{prop:comp}, $u \in GSBD^2(\O)$. Let us show the lower bound inequality $\F'(u) \geq \F(u)$. To this aim, we introduce the following sequence of Radon measures on $\O$
$$ \lambda_k:=\frac{1}{\e_k} f\big(\e_k  \mathbf A e(u_k):e(u_k) \big)\LL^2\res \O.$$
Since the sequence $\{\lambda_k\}_{k \in \N}$ is uniformly bounded in $\M(\O)$, up to a subsequence (not relabeled), we have $\lambda_k \overset{*}{\wto} \lambda$ weakly* in $\M(\O)$ for some nonnegative measure $\lambda \in \M(\O)$. Thanks to the lower semicontinuity of weak* convergence in $\M(\O)$ along open sets, we have that
\begin{equation}\label{eq:lambda}
\F'(u) + \zeta \geq \lim_{k \to \infty}\lambda_k(\O) \geq \lambda(\O).
\end{equation}
Using that the measures $\LL^2\res \O$ and $\HH^1 \res J_u$ are mutually singular, it is enough to show that
\begin{equation}\label{eq:Lebesgue}
\frac{d\lambda}{d\LL^2} \geq \mathbf A e(u):e(u)\quad \LL^2\text{-a.e. in }\O,
\end{equation}
and
\begin{equation}\label{eq:Hausdorff}
\frac{d\lambda}{d\HH^1\res J_u} \geq \kappa \sin\theta_0 \quad \HH^1\text{-a.e. on }J_u.
\end{equation}
Indeed, once \eqref{eq:Lebesgue} and \eqref{eq:Hausdorff} are satisfied, it follows from the Radon-Nikod\'ym decomposition and the Besicovitch differentiation Theorems that
$$\lambda = \frac{d\lambda}{d\LL^2}  \LL^2 \res \O \, + \, \frac{d\lambda}{d\HH^1 \res J_u}  \HH^1 \res J_u+\lambda^s,$$
for some nonnegative measure $\lambda^s$ which is singular with respect to both $\LL^2\res \O$ and $\HH^1\res J_u$. Thus, after integration over $\O$ and recalling \eqref{eq:lambda}, we get that
$$\F'(u)+ \zeta \geq \int_\O \mathbf A e(u):e(u)\, dx + \kappa \sin\theta_0 \HH^1(J_u)=\F(u).$$
Taking the limit as $\zeta \to 0$, we obtain the desired lower bound inequality.
\end{proof}

The rest of this section is devoted to the establishment of \eqref{eq:Lebesgue} and \eqref{eq:Hausdorff}.  We start by identifying the lower bound for the bulk energy. 
\begin{prop}[{\bf Lower bound for the Lebesgue part}] \label{prop:leb}
For $\LL^2$-a.e. $x_0 \in \O$,
$$\frac{d\lambda}{d\LL^2}(x_0) \geq \mathbf A e(u)(x_0):e(u)(x_0).$$
\end{prop}

\begin{proof} 
Let $x_0 \in \O$ be such that
$$\frac{d \lambda}{d \LL^2}(x_0)=\lim_{\varrho  \searrow 0} \frac{\lambda\big( B_\varrho (x_0) \big)}{\pi \varrho^2}$$
exists and is finite, and 
$$\lim_{\varrho  \searrow 0} \frac{1}{\varrho^2}\int_{B_\varrho(x_0)}|e(u) (y)- e(u)(x_0)|^2\, dy = 0.$$
According to Besicovitch and Lebesgue differentiation Theorems, $\LL^2$-almost every point $x_0$ in $\O$ satisfies these properties. We next consider a sequence of radii $\{\varrho_j\}_{j \in \N}$ such that $\varrho_j  \searrow  0$ and $\lambda(\partial B_{\varrho_j}(x_0))=0$ for all $j \in \N$. 

\medskip
 
As in the proof of Proposition \ref{prop:comp}, according to the properties \eqref{eq:f} satisfied by $f$, for all $\delta >0$, there exists a constant $0<K<\kappa$ such that $f(t) \geq K \wedge [ (1-\delta) t ]$ for all $t \geq 0$. Moreover, using the characteristic functions 
$$ \chi_k := \mathds{1}_{\left\{  (1-\delta) \mathbf A e(u_k):e(u_k) \, \geq \, \frac{K}{\e_k} \right\} } \in L^\infty (\O; \left\{ 0,1 \right\} )$$
we have for every Borel set $B \subset \O$,
$$\lambda_k(B) \geq  (1-\delta) \int_B (1-\chi_k) \mathbf A e(u_k):e(u_k) \, dx + \frac{K}{\e_k} \int_B \chi_k \, dx.$$
Note that because $u_k$ is affine on each triangle $T \in \mathbf{T}^k$, $\chi_k$ is constant on each triangle $T \in \mathbf{T}^k$. Following the proof of Proposition \ref{prop:comp}, the sequence $v_k:=(1-\chi_k)u_k \in SBV^2(\O;\R^2)$ satisfies
$v_k \to u \text{ in measure in } \O$ and $ e(v_k) \wto e (u)$ weakly in $L^2(\O;\mathbb M^{2 \times 2}_{\rm sym})$. Then, for all $j \in \N$,
\begin{eqnarray*}
\lambda(B_{\varrho_j}(x_0))= \lim_{k \to \infty} \lambda_k(B_{\varrho_j}(x_0)) & \geq & (1- \delta) \liminf_{k \to \infty}  \int_{B_{\varrho_j}(x_0)} \mathbf A e(v_k):e(v_k)\, dx  \\
 & \geq  &(1- \delta) \int_{B_{\varrho_j}(x_0)} \mathbf A e(u) : e(u) \, dx.
\end{eqnarray*}
Dividing the previous inequality by $\pi\varrho_j^2$ and passing to the limit as $j \to \infty$ implies by the choice of the point $x_0$ that
\begin{eqnarray*}
\frac{d\lambda}{d\LL^2}(x_0) = \underset{j \to \infty}{\lim} \, \frac{\lambda(B_{\varrho_j}(x_0))}{\pi \varrho_j^2} & \geq & (1- \delta)\lim_{j \to \infty}   \frac{1}{\pi \varrho_j^2} \int_{B_{\varrho_j}(x_0)} \mathbf A e(u) : e(u)\, dx\\
& = & (1- \delta) \mathbf A e(u)(x_0): e(u)(x_0).
\end{eqnarray*}
Taking the limit as $\delta \to 0^+$ completes the proof of the lower bound for the Lebesgue part.
\end{proof}

We next pass to the lower bound inequality for the jump part of the energy which represents the most difficult and original part of our result.

\begin{prop}[{\bf Lower bound for the jump part}]\label{prop Jump part}
For $\HH^1$-a.e. $x_0 \in J_u$, 
$$\frac{d\lambda}{d\HH^1 \res J_u}(x_0) \geq \kappa \sin\theta_0.$$
\end{prop}
The proof of Proposition \ref{prop Jump part} is quite long and involved. It necessitates the introduction of some tools in order to carry out the blow-up analysis coupled with the slicing method. 

\medskip

Let $x_0 \in J_u$ be such that 
$$ \frac{d \lambda}{d \HH^1 \res J_u}(x_0)=\lim_{\varrho \searrow 0} \, \frac{\lambda\big( B_\varrho (x_0) \big)}{\HH^1\big(J_u \cap  B_\varrho (x_0) \big)}$$
exists and is finite, and
$$\lim_{\varrho \searrow 0} \, \frac{\HH^1( J_u \cap B_\varrho (x_0))}{2 \varrho}=1.$$
According to the Besicovitch differentiation Theorem and the  countably $(\HH^1,1)$-rectifiability of $J_u$ (see \cite[Theorem 2.83]{AFP}), it follows that $\HH^1$-almost every point $x_0$ in $J_u$ fulfills these conditions. The point $x_0 \in J_u$ being fixed throughout the rest of the proof of Proposition \ref{prop Jump part}, we sometimes intentionally omit to write the dependence with respect to $x_0$.

\medskip

By definition of the jump set $J_u$, there exist $\nu :=\nu_u(x_0) \in \mathbb S^1$ and $u^\pm(x_0) \in \R^2$ with $u^+(x_0) \neq u^-(x_0)$ such that the function 
$$u_{x_0,\varrho} := u( x_0 + \varrho \, \cdot )$$
converges in measure in $B := B_1(0)$ to the jump function
$$\bar{u} : y \in B \mapsto 
\begin{cases}
u^+(x_0) & \text{ if } y \cdot \nu > 0, \\
u^-(x_0) & \text{ if } y \cdot \nu < 0,
\end{cases} $$
as $\varrho \searrow 0 $ (see \cite[Definition 2.3]{DM2}). Note that, the jump set $J_{\bar u}$ in $B$ coincides with the diameter  $B^\nu=p_\nu(B)$ orthogonal to $\nu$. Moreover, since $\HH^1 \left( \left\{ \xi \in \mathbb{S}^1: \, [u](x_0) \cdot \xi = 0 \right\} \right) = 0$, for any $\eta>0$, there exists $\xi \in \mathbb{S}^1$ such that 
\begin{equation}\label{eq:xi}
\left\lvert \nu - \xi \right\rvert \leq \eta, \quad \nu \cdot \xi \geq \frac12, \quad \left\lvert \nu \cdot \xi^\perp \right\rvert \leq \eta \quad \text{and}\quad [u](x_0) \cdot \xi \neq 0,
\end{equation}
where $[u](x_0) := u^+(x_0) - u^-(x_0)$. If $[u](x_0) \cdot \nu\neq 0$, we can simply take $\xi=\nu$. We then set 
\begin{equation}\label{eq:M}
M_\eta := | u^+(x_0) \cdot \xi | + | u^-(x_0) \cdot \xi | > 0.
\end{equation}
From now on, when working with the convergence in measure, we will use the distance $d_{M_\eta}$ defined in \eqref{eq:d_M} associated to this precise value of $M_\eta$. As before, we consider a sequence of radii $\{\varrho_j\}_{j \in \N}$ such that $\varrho_j \searrow 0$ and $\lambda(\partial B_{\varrho_j}(x_0))=0 = \HH^1(J_u \cap \partial B_{\varrho_j}(x_0))$ for all $j \in \N$. 

\medskip

By our choice of $x_0$, we have    
$$  \begin{cases}
\ds \lim_{j \to \infty}\lim_{k \to \infty} u_k (x_0 + \varrho_j \, \cdot)  =\lim_{j \to \infty} u_{x_0,\varrho_j} =   \bar{u} \quad \text{ in measure in } B, \\
\ds\lim_{j \to \infty}\lim_{k \to \infty}  \frac{\lambda_k (B_{\varrho_j}(x_0))}{2\varrho_j} =\lim_{j \to \infty}\frac{ \lambda (B_{\varrho_j}(x_0))}{2\varrho_j}  = \frac{d \lambda}{d \HH^1 \res J_u}(x_0), \\
\ds\lim_{j \to \infty}\lim_{k \to \infty} \frac{\e_k}{\varrho_j} = \lim_{j \to \infty}\lim_{k \to \infty} \frac{\omega(\e_k)}{\varrho_j}= 0. 
\end{cases}   $$
The metrizability of the convergence in measure in $B$ shows the existence of an increasing sequence $\{k_j\}_{j \in \N}$ (depending on $\eta$) such that $k_j \nearrow  \infty$ as $j \to \infty$ and
\begin{subequations}\label{eq:beforeSections}
	\begin{empheq}[left=\empheqlbrace]{align}
	& v_j := u_{k_j}(x_0 + \varrho_j \, \cdot) \to \bar{u} \quad\text{ in measure in } B, \label{seq:CVmesure} \\
	& \frac{ \lambda_{k_j}( B_{\varrho_j}(x_0))}{2\varrho_j} \to \frac{d \lambda}{d \HH^1 \res J_u}(x_0), \label{seq:2h(x_0)}\\
	& \frac{\e_{k_j}}{\varrho_j} \to 0, \quad \frac{\omega(\e_{k_j})}{\varrho_j}\to 0. \label{seq:eps'}
	\end{empheq}
\end{subequations}
In particular, using a change of variables, we get that
\begin{eqnarray*}
2\frac{d \lambda}{d \HH^1 \res J_u}(x_0) & = &\lim_{j \to \infty}\frac{1}{\varrho_j\e_{k_j} } \int_{B_{\varrho_j}(x_0)} f\left(\e_{k_j} \mathbf A e(u_{k_j}):e(u_{k_j}) \right)\, dx \\
& =& \lim_{j \to \infty}\frac{\varrho_j}{\e_{k_j} } \int_{B} f\left(\frac{\e_{k_j}}{\varrho_j^2} \mathbf A e(v_j):e(v_j) \right)\, dy\\
&\geq & \limsup_{j \to \infty}\frac{\varrho_j}{\e_{k_j} } \int_{B} f\left(\frac{\e_{k_j}}{\varrho_j^2} \alpha |e(v_j)\xi\cdot\xi|^2 \right)\, dy,
\end{eqnarray*}
where, in the last inequality, we used the ellipticity property \eqref{eq:A} of $\mathbf A$, the nondecreasing character of $f$ and that $\xi \in \mathbb S^1$.

\medskip

According to the properties \eqref{eq:f} satisfied by $f$, for all $\delta \in (0,1)$, there exists a constant $A>0$ such that
$$f(t) \geq \left( At \right) \wedge [(1- \delta) \kappa] \quad \text{ for all }t \geq 0.$$ 
Indeed, since $f(t) \to \kappa$ as $t \to \infty$, there exists $t^* \geq 0$ such that for all $t \geq t^*$, $f(t) \geq (1-\delta) \kappa$. The function $f(t)/t$ being continuous over $[0,t^*]$ (extended by the value $1$ at $t=0$), it reaches its minimum value $A>0$ over this segment so that $f(t) \geq At $ for all $t \in [0,t^*]$.

\medskip

Let us introduce the characteristic functions
$$ \chi_j := \mathds{1}_{\left\{ \frac{A\alpha \e_{k_j}}{\varrho_j^2}|e(v_j)\xi \cdot\xi|^2 \geq (1-\delta) \kappa \right\}} \in L^\infty(B; \lbrace0,1 \rbrace),$$ 
so that
\begin{equation}\label{eq:lambda_k greater than}
2\frac{d \lambda}{d \HH^1 \res J_u}(x_0) \geq \limsup_{j \to \infty}\left\{\frac{\alpha A}{\varrho_j} \int_{B}  (1-\chi_j) |e(v_j)\xi\cdot\xi|^2 \, dy + \frac{(1-\delta)\kappa\varrho_j}{\e_{k_j}} \int_{B} \chi_j \, dy\right\}.
\end{equation}
We then introduce the translated and rescaled triangulations 
\begin{equation}\label{eq:new-triang}
\mathbf{T}^{x_0,j} := \frac{1}{\varrho_j} \left( \mathbf{T}^{k_j} - x_0 \right), \quad \mathbf{T}^{x_0,j}_b := \left\{ T \in \mathbf{T}^{x_0,j}: \, \frac{\alpha A}{\varrho_j} |e(v_j)_{|T} \xi \cdot \xi |^2 \geq\frac{ (1-\delta) \kappa\varrho_j}{\e_{k_j}} \right\}.
\end{equation}
Note that $v_j$ is affine on each $T \in \mathbf{T}^{x_0,j}$. Let us point out that
\begin{equation}\label{eq:chiTb}
{\chi_j} _{|T} := 
\begin{cases}
 1 & \text{ if } T \in \mathbf{T}^{x_0,j}_b, \\
 0 & \text{ otherwise.}
\end{cases}
\end{equation}
Since $( (v_j)^\xi_z ) ' (t) = e(v_j)(z + t \xi) \xi \cdot \xi$, then for $\HH^1$-a.e. $z \in B^\xi$, 
\begin{equation}\label{eq:tilde chi}
( \chi_j) ^\xi_z(t) = 
\begin{cases}
 1 & \text{ if } \frac{\alpha A}{\varrho_j} |( (v_j)^\xi_z ) ' (t)|^2 \geq  \frac{(1-\delta)\kappa\varrho_j}{\e_{k_j}}, \\
 0 & \text{ otherwise,}
\end{cases}
\qquad \text{ for $\LL^1$-a.e. $t \in B^\xi_z$.}
\end{equation}
The triangles belonging to the collection $\mathbf{T}^{x_0,j}_b$ correspond to the sets where the longitudinal slope of $v_j$ in the direction $\xi$ is ``very large''. They, roughly speaking,  represent the places where it will be energetically convenient to introduce a jump because of the sharp transition.

\medskip

The following result, which will play a major role in the proof of Proposition \ref{prop Jump part}, shows that for many points $y \in J_{\bar u} \cap B$, the one-dimensional energy on $B_y^\xi$ is arbitrarily small uniformly with respect to $y$.

\begin{lem}\label{lemma afterSections}
For all $\eta > 0$, there exist a subset $Z \subset J_{\bar u} \cap B$ with $\HH^1(Z) \leq \eta$ and a subsequence (not relabeled, depending on $x_0$) such that the following property holds : for all $\gamma >0$, there exists $j_0 = j_0(\gamma) \in \N$ such that for all $y \in J_{\bar u} \cap B \setminus Z$ and all $j \geq j_0$, 
\begin{subequations}

	\begin{empheq}[left=\empheqlbrace]{align}
	& \int_{B^\xi_y}   (\chi_j) ^\xi_y  \, dt \leq \gamma , \label{seq:boundL1chi} \\
	& \int_{B^\xi_y}  \big( 1 - ( \chi_j ) ^\xi_y  \big)  |( (v_j)^\xi_y ) ' |^2  \, dt \leq \gamma^2, \label{seq:boundL2vPrime} \\
	& \int_{B^\xi_y} M_\eta \wedge | (v_j - \bar{u})^\xi_y| \, dt \leq \gamma. \label{seq:boundCVmeasure}	
	\end{empheq}
\end{subequations}

\end{lem}

\begin{proof} 
According to Fubini's Theorem, the convergence in measure \eqref{seq:CVmesure} and \eqref{eq:lambda_k greater than}, we infer that
\begin{multline*}
\int_{B^\xi} \left( \int_{B^\xi_z} M_\eta \wedge | (v_j - \bar{u})^\xi_z| \, dt +\int_{B^\xi_z} (1-(\chi_j)_z^\xi)|((v_j)_z^\xi)'|^2\, dt +\int_{B^\xi_z} (\chi_j)_z^\xi\, dt\right)  d\HH^1(z) \\
\leq \int_B M_\eta \wedge | v_j - \bar{u}| \, dx + \int_{B}  (1-\chi_j) |e(v_j)\xi\cdot\xi|^2 \, dx + \int_{B} \chi_j \, dx \to 0.
\end{multline*}
As a consequence, up to a subsequence (not relabeled), there exists an $\HH^1$-negligible set $N \subset B^\xi$ such that
$$ \int_{B^\xi_z} M_\eta \wedge | (v_j - \bar{u})^\xi_z| \, dt +\int_{B^\xi_z} (1-(\chi_j)_z^\xi)|((v_j)_z^\xi)'|^2\, dt +\int_{B^\xi_z} (\chi_j)_z^\xi\, dt\to 0 \quad \text{ for all $z \in B^\xi \setminus N$.}$$
In order to pass from arbitrary points $z \in B^\xi$ to arbitrary points $y \in J_{\bar u} \cap B=B^\nu$, let us consider the following mapping (see Figure \ref{fig:A})
\begin{equation}\label{eq:Phi}
\Phi : z \in \R^2 \longmapsto z - \frac{\nu \cdot z}{\nu \cdot \xi} \xi \in \Pi_\nu
\end{equation}
which corresponds to the linear projection onto $\Pi_\nu$ in the direction $\xi$. Thanks to \eqref{eq:xi}, we can check that the Lipschitz constant of $\Phi$ is bounded by $\sqrt{1+4\eta^2}$. Moreover, since for all $z \in B^\xi$ we have $ B^\xi_z + \frac{\nu \cdot z}{\nu \cdot \xi}  = B^\xi_{\Phi(z)} = \left\{ s \in \R \, : \, z + \big(s - \frac{\nu \cdot z}{\nu \cdot \xi}\big) \xi \in B \right\}  $, we deduce that
\begin{multline*}
\int_{B^\xi_z} M_\eta \wedge | (v_j - \bar{u})^\xi_z| \, dt +\int_{B^\xi_z} (1-(\chi_j)_z^\xi)|((v_j)_z^\xi)'|^2\, dt +\int_{B^\xi_z} (\chi_j)_z^\xi\, dt\\
=\int_{B^\xi_{\Phi(z)}} M_\eta \wedge | (v_j - \bar{u})^\xi_{\Phi(z)}| \, ds +\int_{B^\xi_{\Phi(z)}} (1-(\chi_j)_{\Phi(z)}^\xi)|((v_j)_{\Phi(z)}^\xi)'|^2\, ds +\int_{B^\xi_{\Phi(z)}} (\chi_j)_{\Phi(z)}^\xi\, ds
\end{multline*}
thanks to the change of variables $s = t + \frac{\nu \cdot z}{\nu \cdot \xi}$. Since $B^\nu \subset \Phi(B^\xi)$,  setting $N':=\Phi(N)\subset \Pi_\nu$, we get that $\HH^1(N')=0$ and 
$$ \int_{B^\xi_y} M_\eta \wedge | (v_j - \bar{u})^\xi_y| \, ds +\int_{B^\xi_y} (1-(\chi_j)_y^\xi)|((v_j)_y^\xi)'|^2\, ds +\int_{B^\xi_y} (\chi_j)_y^\xi\, ds\to 0 \quad \text{ for all $y \in B^\nu \setminus N'$.}$$
Applying Egoroff's theorem, for all $\eta > 0$, there exists a subset $Z \subset B^\nu$ such that $\HH^1(Z) \leq \eta$ and the above convergence is uniform with respect to $y \in B^\nu \setminus Z$.
\end{proof}

Let us consider the subsequence introduced in Lemma \ref{lemma afterSections}. For all $y \in \big(B_{1-\frac\eta4}\big)^\nu = J_{\bar u} \cap B_{1-\frac\eta4}$, we define the end points of the section passing through $y$ in the direction $\xi$ (see the Figure \ref{fig:A}) :
\begin{equation}\label{eq:AB(Y)}
\begin{cases}
a(y) := \min \left\{ t \in [-2,2]: \, y + t \xi \in \bar{B_{1-\frac\eta4}} \right\} \in [-2,0], \\
 b(y) := \max \left\{ t \in [-2,2]: \, y + t \xi \in \bar{B_{1-\frac\eta4}} \right\} \in [0,2],
\end{cases}
\end{equation}
so that $\big( B_{1-\frac\eta4} \big)^\xi_y = \left( a(y),b(y) \right)$. Note that, for all $y \in J_{\bar u} \cap B_{1-\frac\eta2} \subset \big(B_{1-\frac\eta4}\big)^\nu $, 
\begin{equation}\label{eq:L(eta)}
0 < L_\eta := \sqrt{ \left( 1- \frac\eta2 \right)^2 \lvert \xi \cdot \nu^\perp \rvert ^2 + \frac{\eta (8 - 3\eta)}{16} } - \left(1 - \frac\eta2 \right) \left\lvert \xi \cdot \nu^\perp \right\rvert    \leq |a(y)|,\, |b(y)| \leq 2.
\end{equation}

\begin{figure}[hbtp]
\begin{tikzpicture} [x=3.0cm,y=3.0cm] 
\draw  (0.,0.) circle (3.0048467864033146cm);
\draw  (0.,0.) circle (2.760265513864068cm);
\draw [->,style=very thick] (0.,0.) -- (0.9805597694698831,0.20429473703800505);
\draw [->,style=very thick] (0.,0.) -- (1.0016155954677715,0.);
\draw [style=very thick, color=ffqqqq] (0.,0.8365071179654948)-- (0.,-0.8396873008006259);
\draw [->, style=thick] (-0.21600831631103196,1.0367817982805898) -- (-9.85157848665541E-4,1.081580800832876);
\draw [style=thick,color=qqqqff] (0.,0.8365071179654947)-- (0.24359026090338512,0.8872579338216052);
\draw [style=thick,color=qqqqff] (0.10661745830194527,0.8835846613713779) -- (0.1165471510710217,0.8359248078004036);
\draw [style=thick,color=qqqqff] (0.12704310983236325,0.8878402439866964) -- (0.13697280260143968,0.840180390415722);
\draw [style=thick,color=qqqqff] (0.,-0.8396873008006259)-- (-0.23694926645881473,-0.8890544985888058);
\draw [style=thick,color=qqqqff] (-0.10329696107966024,-0.8860730351725439) -- (-0.11322665384873667,-0.8384131816015694);
\draw [style=thick,color=qqqqff] (-0.12372261261007822,-0.8903286177878622) -- (-0.13365230537915465,-0.8426687642168877);
\draw (-0.1903568088883866,1.2310798188162422) node[anchor=north west] {$ \Phi$};
\draw (-0.8406261548004798,1.0259146241167054) node[anchor=north west] {$\frac{\eta}{2}$};
\draw (-1.15,0.05) node[anchor=north west] {$\frac{\eta}{4}$};
\draw (0.9989058451326608,0.30957513075900156) node[anchor=north west] {$\Large  \xi$};
\draw (1.0336796069461418,0.08006830279002843) node[anchor=north west] {$\Large  \nu$};
\draw (-0.17,-1.15) node[anchor=north west] {$\large \Pi_\nu$};
\draw  (-1.2010981881473148,0.5862643014189441)-- (1.1862918283108312,1.0836651032927107);
\draw (-1.2071650114525156,0.1484931910784999)-- (1.1917884876560465,0.6483031868850213);
\draw  (-1.00792026596579,-1.0496824608122173)-- (1.1979379877961869,-0.5901028962344792);
\draw (0,0.8365071179654948)--(0,1.2556164359457356);
\draw  (0.,-1.2)--(0,-0.8396873008006259);
\draw  (-0.24986720336177864,1.1992953459243092)-- (0.24103161935051848,-1.1568869200058223);
\draw (0.24,-1.12) node[anchor=north west] {$\large \Pi_\xi$};
\draw (-1.2,-0.45) node[anchor=north west] {$ B$};
\draw (0.015,0.6) node[anchor=north west] {$ y$};
\draw [color=qqqqff](-0.92,0.22) node[anchor=north west] {{\small $y+a(y)\xi$}};
\draw [color=qqqqff](0.01,0.38) node[anchor=north west] {$0$};
\draw [color=qqqqff](0.76,0.58) node[anchor=north west] {\small $y+b(y)\xi$};
\draw [color=qqqqff](-0.3,-0.55) node[anchor=north west] {$ L_\eta$};
\draw [<->,style=thick,color=qqqqff] (-0.25858235479351643,-0.774791464899063) -- (-0.023715730545229102,-0.725858175227442);
\draw (-0.14,0.005) node[anchor=north west] {$\large 0$};
\draw [color=ffqqqq](-0.6,-0.23) node[anchor=north west] {$ J_{\bar u} \cap  B_{1- \frac\eta2} $};
\draw [dashed,domain=-0.9:0] plot(\x,{(--0.9216617708490081--3.9737386711924394E-5*\x)/1.1017978819960574});
\draw [dashed,domain=-0.9:0] plot(\x,{(--1.107409726344132--3.9737386711924394E-5*\x)/1.1017978819960574});
\draw [<->] (-0.85,0.8364726612240576) -- (-0.85,1.0050588949505228);
\draw [dashed] (0.,0.)-- (-1.001608718707807,0.003711562038167028);
\draw [<->] (-0.920082187599316,0.003409457261806758) -- (-1.001608718707807,0.003711562038167028);

\draw [fill=uuuuuu] (0.,0.4) circle (1pt);
\draw [fill=qqqqff] (-0.894966064970131,0.21353828436194397) circle (1pt);
\draw [fill=qqqqff] (0.7352240785758172,0.5531802695493612) circle (1pt);
\end{tikzpicture}
\caption{}
\label{fig:A}
\end{figure}

\noindent
We introduce the family 
$$\mathbf{T}^{x_0,j}_{b,int} := \left\{ T \in \mathbf{T}^{x_0,j}_b:\;  T \cap \bar{B_{1-\frac\eta4}} \neq \emptyset \right\}$$
 of triangles which intersect $\bar{B_{1-\frac\eta4}}$ and where $v_j$ varies enough in the direction $\xi$. Note that for $j \in \N$ large enough (depending on $\eta$), each $T \in \mathbf{T}^{x_0,j}_{b,int}$ is contained in $B$, since the lengths of all triangles's edges are controlled by $\omega(\e_{k_j})/\varrho_j\to 0$. The collection $\mathbf{T}^{x_0,j}_{b,int}$ is introduced for technical reasons to deal with triangles which intersect the boundary of the ball $B$.

\medskip 

In the following result, we show that, for some subset of $Z' \subset J_{\bar u} \cap B_{1-\frac\eta2}$ of arbitrarily small $\HH^1$ measure, and along a subsequence (only depending on $\eta$), all the sections in the direction $\xi$ passing through $J_{\bar u} \cap B_{1-\frac\eta2}\setminus Z'$ must cross at least one triangle $T \in \mathbf{T}^{x_0,j}_{b,int}$ contained in $B$, and on which the longitudinal slope of $v_j$ in the direction $\xi$ is ``large". The formal idea of the proof consists in observing that, if for some $y \in J_{\bar u} \cap B_{1-\frac\eta2}$ the one-dimensional section $B_y^\xi$ intersects no triangle in the collection $ \mathbf{T}^{x_0,j}_{b,int}$ for infinitely many $j$'s, then the function $(v_j)_y^\xi$ would be bounded in $H^1(B_y^\xi)$. Lemma \ref{lemma afterSections} would then entail that $(v_j)_y^\xi$ converges (weakly in $H^1(B_y^\xi)$ and also $\LL^1$-a.e. in $B_y^\xi$) to a constant function. This property contradicts the fact that $(v_j)_y^\xi \to \bar u_y^\xi$ $\LL^1$-a.e. in $B_y^\xi$, where $\bar u_y^\xi$ is a step function taking two different values $u^\pm(x_0)\cdot \xi$.

\begin{lem}\label{lemma1}
For all $\eta>0$, there exist a subset $ Z' \subset J_{\bar u} \cap B$ containing $Z$ with $\HH^1(Z') \leq \eta$, and a subsequence (not relabeled) such that the following property holds : for all $y \in J_{\bar u} \cap B_{1-\frac\eta2}\setminus Z'$ and all $j \in \N$, there exists a triangle $T=T(y,j) \in \mathbf{T}^{x_0,j}_{b,int}$ such that $( \mathring{T} \cap B )^\xi_y \neq \emptyset$.
\end{lem}

\begin{proof}
Let $Z$ be the exceptional set given by Lemma \ref{lemma afterSections}. We first show the weaker result that there exists an increasing mapping $\phi : \N \to \N$ with the following property: for all $y \in J_{\bar u} \cap B_{1-\frac\eta2}\setminus Z$ and all $j \in \N$, there exists a triangle $T=T(y,\phi(j)) \in \mathbf{T}^{x_0,\phi(j)}_{b,int}$ such that $(T \cap B )^\xi_y \neq \emptyset$.

\medskip

Suppose by contradiction that such is not the case, and define
$$\gamma^*_1 := L_\eta M_\eta > 0, \quad \gamma^*_2 := \frac{ L_\eta |[u](x_0)\cdot \xi | }{1 + 2 L_\eta} > 0 \quad\text{and}\quad\gamma^* = \gamma^*(\eta) := \frac{ \gamma^*_1 \wedge \gamma^*_2}{4} > 0,$$
where we recall that the constants $M_\eta$ and $L_\eta$ are defined in \eqref{eq:M} and \eqref{eq:L(eta)}, respectively. Thanks to Lemma \ref{lemma afterSections}, there exists $j^* = j^*(\gamma^*) \in \N$ such that for all $y \in J_{\bar u} \cap B \setminus Z$ and all $j \geq j^*$,
$$
\ds \int_{B^\xi_y}  ( 1 - ( \chi_j ) ^\xi_y  )  |( (v_j)^\xi_y ) ' |^2  \, dt \leq {\gamma^*}^2 \, \text{ and } \,
 \ds \int_{B^\xi_y} M_\eta \wedge | (v_j - \bar{u})^\xi_y | \, dt \leq \gamma^*. 	
$$
We then consider the extraction $ \phi : j \in \N \longmapsto j + j^* \in \N$ which only depends on $\eta$. By assumption, there exists $y = y(\phi) \in J_{\bar u} \cap B_{1-\frac\eta2}\setminus Z$ and $ j=j(\phi) \in \N$ such that $\left( T \cap B \right)^\xi_y = \emptyset$ for all $T \in \mathbf{T}^{x_0,j + j^*}_{b,int}$. Remembering \eqref{eq:chiTb}, we deduce that $(\chi_{j+j^*})_y^\xi \equiv 0$ on $(a(y),b(y))$. Moreover, since $\phi(j) = j + j^* \geq j^*$, we have 
$$ 
\int_{a(y)}^{b(y)}  |( (v_{j + j^*})^\xi_y ) '|^2  \, dt \leq {\gamma^*}^2 \, \text{ and } \, \int_{a(y)}^{b(y)} M_\eta \wedge | (v_{j+j^*} - \bar{u})^\xi_y| \, dt \leq \gamma^*. 	
$$

\begin{figure}[hbtp]
\begin{tikzpicture}[x=0.9cm,y=0.9cm] 
\fill[color=lightgray] (-1.1986128611404252,1.9036569398662626) -- (0.8783898988376524,1.4498897110066784) -- (-1.2564524394395842,-1.1939742234801571) -- cycle;
\draw (-1.1986128611404252,1.9036569398662626) -- (0.8783898988376524,1.4498897110066784) ;    
\draw (0.8783898988376524,1.4498897110066784) -- (-1.2564524394395842,-1.1939742234801571)  ;
\draw (-1.2564524394395842,-1.1939742234801571) -- (-1.1986128611404252,1.9036569398662626) ;

\draw  (0.,0.) circle (3.9969988741554587);
\draw  (0.,0.) circle (3.603700199849338);
\draw [style=very thick,color=ffqqqq] (0.,3.1966288562712433)-- (0.,-3.195491430725323);
\draw [->,style=very thick] (0.,0.) -- (3.9969988741554587,0.);
\draw [->,style=very thick] (0.,0.) -- (3.870508994373706,0.9975771270795284);
\draw [domain=-5.:5.] plot(\x,{(--3.08947320813629--0.9975771270795284*\x)/3.870508994373706});
\draw  (0.,3.1966288562712433)-- (0.,4.623094859594526);
\draw  (0.,-3.195491430725323)-- (0.,-5.);
\draw (-0.45,1.2) node[anchor=north west] {$y$};
\draw (-1.2,1.8) node[anchor=north west] {$\mathbf T$};
\draw (-0.5,-4.8) node[anchor=north west] {$\Pi_\nu$};
\draw (0.6,-1.9) node[anchor=north west] {\large $u^+(x_0)$};
\draw (-2.,-1.9) node[anchor=north west] {\large $u^-(x_0)$};
\draw [color=qqqqff] (-2.6,0.35) node[anchor=north west] {{\small $y+t^-\xi$}};
\draw [color=qqqqff] (1.6,1.4) node[anchor=north west] {\small $y+t^+\xi$};
\draw [color=qqqqff] (-3.6,-0.05) node[anchor=north west] {\small $y+a(y)\xi$};
\draw [color=qqqqff] (1.5,2.1) node[anchor=north west] {\small$y+b(y)\xi$};
\draw [color=qqqqff] (0.03,0.85) node[anchor=north west] {$0$};
\draw (4,1.4) node[anchor=north west] {\Large $\xi$};
\draw (4.1,0.2) node[anchor=north west] {\Large $\nu$};
\draw  (-1.6637965878220102,3.196628856271243)-- (1.6637965878220107,3.196628856271243);
\draw  (-1.6659800858819374,-3.195491430725323)-- (1.6659800858819374,-3.195491430725323);
\begin{scriptsize}
\draw [fill=uuuuuu] (-3.601354720251178,-0.12999734350639214) circle (1.0pt);
\draw [fill=uuuuuu] (3.2155275035308755,1.626972035540455) circle (1.0pt);
\draw [fill=uuuuuu] (-2.660930013393115,0.11238581040586926) circle (1pt);
\draw [fill=uuuuuu] (1.6697115756915342,1.228556577901848) circle (1pt);
\draw [fill=uuuuuu] (0.,0.7982085076219292) circle (1pt);
\end{scriptsize}
\end{tikzpicture}
\caption{}
\label{fig:B}
\end{figure}
\noindent
By continuity of $\left( v_{j+j^*} \right)^\xi_y$ on the compact $[a(y),b(y)]$, $(v_{j+j^*})_y^\xi $ being in $ H^1(B_y^\xi)$, there exist two points $t^\pm \in [a(y),b(y)] \cap \R^\pm$ such that 
$$ \underset{ [a(y),b(y)] \cap \R^\pm}{\min} \left( M_\eta \wedge | \left( v_{j+j^*} \right)^\xi_y - u^\pm (x_0) \cdot \xi | \right) = M_\eta \wedge | \left( v_{j+j^*} \right)^\xi_y (t^\pm) - u^\pm (x_0) \cdot \xi |.$$
Hence,
\begin{eqnarray*}
\frac{\gamma^*}{L_\eta}	& \geq & \frac{1}{L_\eta}\int_{a(y)}^0 M_\eta \wedge | \left( v_{j+j^*} \right)^\xi_y - u^-(x_0)\cdot \xi| \, dt + \frac{1}{L_\eta}\int_0^{b(y)} M_\eta \wedge | \left( v_{j+j^*} \right)^\xi_y - u^+(x_0)\cdot \xi| \, dt \\
			& \geq &   M_\eta \wedge | \left( v_{j+j^*} \right)^\xi_y(t^-) - u^-(x_0)\cdot \xi| \, + \, M_\eta \wedge | \left( v_{j+j^*} \right)^\xi_y(t^+) - u^+(x_0)\cdot \xi|  \\
			& \geq &  M_\eta \wedge \left( | \left( v_{j+j^*} \right)^\xi_y(t^-) - u^-(x_0)\cdot \xi| \, + \, | \left( v_{j+j^*} \right)^\xi_y(t^+) - u^+(x_0)\cdot \xi| \right) \\
			& \geq & M_\eta \wedge \left( | [u](x_0)\cdot \xi | \, - \, \left\lvert \int_{t^-}^{t^+} \left( (v_{j + j^*})^\xi_y \right) ' (t) \, dt \right\rvert \right) \\
			& \geq & M_\eta \wedge \left( | [u](x_0) \cdot \xi | - 2 \gamma^* \right),
\end{eqnarray*}
which is impossible thanks to of our choice of $\gamma^*$.

\medskip 

We are now in position to complete the proof of Lemma \ref{lemma1}. For all $j \in \N$, let
\begin{eqnarray*}
Z_j & := &\Big\{ y \in J_{\bar u}\cap B_{1-\frac\eta2}:\; \text{ there exists } T \in \mathbf{T}^{x_0,j}_{b,int} \text{ such that }\\
&&\qquad\qquad ( T \cap B )^\xi_y \text{ is contained in an edge or a vertex of } T \Big\},
\end{eqnarray*}
and 
$$Z' :=Z \cup \bigcup_{j \in \N} Z_j.$$
We notice that $\bigcup_j Z_j$ is $\HH^1$-negligible (each $Z_j$ being finite), hence $\HH^1(Z')\leq \eta$. Moreover, for all $ y \in J_{\bar u} \cap B_{1-\frac\eta2}\setminus Z'$ and all $ j \in \N$, there exists a triangle $ T \in \mathbf{T}^{x_0, \phi(j)}_{b,int}$ such that $ ( T \cap B )^\xi_y$ is non-empty, and it is neither reduced to a vertex of $T$ nor contained in an edge of it. It thus implies that  $( \mathring{T} \cap B )^\xi_y \neq \emptyset$.
\end{proof}

Let us consider the further subsequence introduced in Lemma \ref{lemma1}. As a consequence, for all $j \in \N$, the family of triangles 
\begin{equation}\label{eq:Fj}
\mathscr{F}_j := \left\{ T \in \mathbf{T}^{x_0, j}_{b,int}: \; \text{there exists } y \in J_{\bar u}\cap B_{1-\frac\eta2} \text{ such that } ( \mathring{T} \cap B )^\xi_y \neq \emptyset \right\}
\end{equation}
is nonempty. Thanks to Lemma \ref{lemma1}, it is possible to obtain a bad lower bound. Indeed, from that result, we infer that $J_{\bar u} \cap B_{1-\frac\eta2} \setminus Z' \subset \bigcup_{T \in \mathscr F_j} \Phi(p_\xi(T))$ with $\Phi$ the projection onto $\Pi_\nu$ in the direction $\xi$ defined in \eqref{eq:Phi}. Using next that $\LL^2(T) \geq \HH^1(p_\xi(T)) (\e_{k_j}/\varrho_j) \sin\theta_0 /2$ and that the Lipschitz constant of $\Phi$ is bounded by $\sqrt{1+4\eta^2}$, we deduce from \eqref{eq:lambda_k greater than} and our choice of $x_0$ that 
\begin{eqnarray*}
2\frac{d \lambda}{d \HH^1 \res J_u}(x_0)& \geq & \liminf_{j \to \infty} \frac{(1-\delta)\kappa\varrho_j}{\e_{k_j}} \int_{B} \chi_j \, dy\geq \liminf_{j \to \infty} \sum_{T \in \mathscr F_j} \frac{(1-\delta)\kappa\varrho_j\LL^2(T)}{\e_{k_j}}\\
& \geq &  \frac{(1-\delta)\kappa\sin\theta_0}{2\sqrt{1+4\eta^2}}\liminf_{j \to \infty} \HH^1\left( \bigcup_{T \in \mathscr F_j} \Phi(p_\xi(T))\right)\\
& \geq &  \frac{(1-\delta)\kappa\sin\theta_0}{2\sqrt{1+4\eta^2}} \HH^1\left(J_{\bar u} \cap B_{1-\frac\eta2} \setminus Z'\right)\geq \frac{(1-\delta)\kappa\sin\theta_0}{\sqrt{1+4\eta^2}} (1-\eta).
\end{eqnarray*}
Letting $\eta \to 0$ and $\delta \to 0$ leads to
$$\frac{d \lambda}{d \HH^1 \res J_u}(x_0) \geq  \frac{\kappa\sin\theta_0}{2}$$
which corresponds to a too low lower bound because of the factor $1/2$ in the right-hand side of the previous inequality. In order to improve the previous argument, we need to establish that many lines $B_y^\xi$ parallel to $\xi$ and passing through the jump set at some point $y \in J_{\bar u} \cap B$ must actually intersect at least two triangles of the collection $\mathbf{T}^{x_0,j}_{b,int}$, where the longitudinal variation of $v_j$ in the direction $\xi$ is ``large''. This idea is precisely formulated in the following result which is an improvement of Lemma \ref{lemma1}.

\begin{lem}\label{lemma2}
For all $\eta>0$, there exist $ Z'' \subset J_{\bar u} \cap B$ containing $Z'$ with $\HH^1(Z'') \leq 3 \eta$, and a (not relabeled) subsequence such that for all $j \in \N$ and for all $y \in J_{\bar u} \cap B_{1-\frac\eta2}\setminus Z''$,
$$ \# \left\{ T \in \mathbf{T}^{x_0, j}_{b,int}: \; ( \mathring{T} \cap B )^\xi_y \neq \emptyset \right\} \geq 2.$$
\end{lem}

The proof of Lemma \ref{lemma2} consists in constructing both $Z''$ and the subsequence inductively by means of the following technical result, Lemma \ref{lemma3}. It stipulates that the set of all points $y \in J_{\bar u} \cap B$ such that $B_y^\xi$ intersects exactly one triangle $T$ in the collection $\mathbf{T}^{x_0,j}_{b,int}$, has arbitrarily small $\HH^1$ measure. To establish this property, we first show that if such situation arises, then the function $(v_j)_y^\xi$ is uniformly close (with respect to $y$) to the step function $\bar u_y^\xi$ taking the values $u^\pm(x_0)\cdot\xi$. Thus, up to a small error which is uniform in $y$, the function $(v_j)_y^\xi$ must pass from the value $u^-(x_0)\cdot \xi$ to $u^+(x_0)\cdot \xi$ in an affine way inside the only triangle $T \in \mathbf{T}^{x_0,j}_{b,int}$ which is crossed by $B_y^\xi$. However, due to the shape of a triangle, this can happen for at most two different values of $y$, say $z_1$ and $z_2$. Then, if $y \in J_{\bar u} \cap B$ is far away from these two values $z_1$ and $z_2$, the variation of $(v_j)_y^\xi$ is not sufficient to connect the values $u^\pm(x_0)\cdot \xi$ in an affine way. It thus becomes necessary for $B_y^\xi$ to intersect an additional triangle $T' \in \mathbf{T}^{x_0,j}_{b,int}$, where the variation of $(v_j)_y^\xi$ is substantial, in order to recover the full jump.

\begin{lem}\label{lemma3}
For all $\eta>0$, there exist constants $C_* = C_*(\eta) > 0$, $\gamma_* = \gamma_*(\eta) > 0$ and a subset $Z_* = Z_*(\eta) \subset J_{\bar u} \cap B$ containing $Z'$ and satisfying $\HH^1(Z_*) \leq 2 \eta$ such that the following property holds: for all $0<\gamma<\gamma_*$, there exists $ j(\gamma) \in \N $ such that for all $ j \geq j(\gamma)$, the set 
\begin{equation}\label{eq:Zj}
Y_j := \left\{ y \in J_{\bar u}\cap B_{1-\frac\eta2}\setminus Z': \; \text{there exists a unique } T \in \mathbf{T}^{x_0, j}_{b,int} \text{ such that } ( \mathring{T} \cap B )^\xi_y \neq \emptyset \right\}
\end{equation}
satisfies 
$$ \HH^1(Y_j \setminus Z_*) \leq C_* \gamma. $$
\end{lem}

\begin{proof}[Proof of Lemma \ref{lemma3}] The proof is divided into three steps.

\medskip

{\bf Step 1.} In this first step, we show that for $j$ large enough and for many points $y \in Y_j$, the set $(B \cap T)_y^\xi$ (where $T$ is the only triangle in $\mathbf{T}^{x_0,j}_{b,int}$ which crosses $B_y^\xi$) is close to $(J_{\bar u})_y^\xi$, uniformly  with respect to $y$. 

For all $j \in \N$ and all $y \in Y_j$, let $T_j(y)\in \mathbf{T}^{x_0, j}_{b,int}$ be the unique triangle such that $( \mathring{T}_j(y) \cap B)^\xi_y \neq \emptyset$. We define the end points of the section in the direction $\xi$ passing through $y$ inside $T_j(y)$ (see the Figure \ref{fig:C}) by
\begin{equation}\label{eq:aj and bj}
\begin{cases}
a_j(y) : =  \min \left\{ t \in [-2,2]:\; y + t \xi \in T_j(y) \right\} ,\\
b_j(y)  :=  \max \left\{ t \in [-2,2]:\; y + t \xi \in T_j(y) \right\},
\end{cases}
\end{equation}
so that $\left( T_j(y) \right)^\xi_y = [a_j(y),b_j(y)]$. Note that $T_j(y) \subset B$ (since $T_j(y) \cap \bar{B_{1-\frac\eta4}} \neq \emptyset$), hence $-2 \leq a_j(y) \leq b(y)$ and $2 \geq b_j(y) \geq a(y)$. Let us show that
\begin{equation}\label{eq:step1}
f_j(y) := \left( |a_j(y)| + |b_j(y)| \right) \mathds{1} _{Y_j}(y)\to  0 \quad \text{ for all }y \in J_{\bar u}\cap B_{1-\frac\eta2} \setminus Z'.
\end{equation}

\begin{figure}[hbtp]
\begin{tikzpicture}[x=3.0cm,y=3.0cm]
\fill[color=lightgray] (-0.07689447381578916,0.699360556254932) -- (0.17370711987197132,0.9511151291725417) -- (0.3767872723294621,0.33286896916692543) -- cycle;
\draw (0.,0.) circle (3.cm);
\draw  (0.,0.) circle (2.7463222859832954cm);
\draw [->,style=very thick] (0.,0.) -- (1.,0.);
\draw [->,style=very thick] (0.,0.) -- (0.9805806756909201,0.19611613513818404);
\draw [domain=-1.5:1.5] plot(\x,{(--0.48931805318080507--0.19611613513818404*\x)/0.9805806756909201});
\draw  (-0.07689447381578916,0.699360556254932)-- (0.17370711987197132,0.9511151291725417);
\draw (0.17370711987197132,0.9511151291725417)-- (0.3767872723294621,0.33286896916692543);
\draw  (0.3767872723294621,0.33286896916692543)-- (-0.07689447381578916,0.699360556254932);
\draw  (0.,1.1174245926780848)-- (0.,0.8356572845602078);
\draw  (0.,-1.123094505326141)-- (0.,-0.841560907951798);
\draw [style=very thick,color=ffqqqq] (0.,0.8356572845602078)-- (0.,-0.841560907951798);

\draw [color=qqqqff, dashed] (0.13716334022107884,0.5264411283486472)-- (0.1,0.35);
\draw [color=qqqqff, dashed] (0.3023510135206918,0.5594786630085697)-- (0.52,0.29);

\draw (-0.15,0.6310854888873771) node[anchor=north west] {$y$};
\draw (-0.035,0.78) node[anchor=north west] {\small $T_j(y)$};
\draw [color=qqqqff] (-0.9,0.35000784961062975) node[anchor=north west] {\small $y+a(y)\xi$};
\draw [color=qqqqff] (-0.2,0.4) node[anchor=north west] {\small $y+a_j(y)\xi$};
\draw [color=qqqqff] (0.48,0.35) node[anchor=north west] {\small $y+b_j(y)\xi$};
\draw [color=qqqqff] (0.64,0.66) node[anchor=north west] {\small $y+b(y)\xi$};
\draw (1.035948614194321,0.3294411930781848) node[anchor=north west] {\Large $\xi$};
\draw (1.035948614194321,0.07921353860010483) node[anchor=north west] {\Large $\nu$};
\draw (-0.1,-1.1) node[anchor=north west] {$\Pi_\nu$};
\begin{scriptsize}
\draw [fill=uuuuuu] (-0.854630792800675,0.3280823017442964) circle (1pt);
\draw [fill=uuuuuu] (0.6627044619143553,0.6315493526873025) circle (1pt);
\draw [fill=uuuuuu] (0.3023510135206918,0.5594786630085697) circle (1pt);
\draw [fill=uuuuuu] (0.13716334022107884,0.5264411283486472) circle (1pt);
\draw [fill=uuuuuu] (0.,0.49900846030443147) circle (1pt);
\end{scriptsize}
\end{tikzpicture}
\caption{}
\label{fig:C}
\end{figure}

Let $y \in J_{\bar u}\cap B_{1-\frac\eta2} \setminus Z'$ and set $\ell:=\limsup_j f_j(y) \in [0,4]$. Assume by contradiction that $\ell>0$ and extract a subsequence depending on $y$ (not relabeled) such that $f_j(y) \to \ell$. Then, there exists $j_0 \in \N$ such that $y \in Y_j$ for all $j \geq j_0$. Moreover, according to Lemma \ref{lemma afterSections} and setting $I_j(y) := \left( a(y),b(y) \right) \setminus \left( a_j(y),b_j(y) \right) \subset B^\xi_y$, we have 
\begin{equation}\label{eq:convwj}
\begin{cases}
\ds \lvert b_j(y) - a_j(y) \rvert \leq \int_{B^\xi_y} ( \chi_j )^\xi_y\, dt \to 0, \\
\ds \int_{I_j(y)} \lvert ( (v_j)^\xi_y )'  \rvert^2 \, dt \leq \int_{B^\xi_y}  \big( 1 - ( \chi_j ) ^\xi_y \big)  |( (v_j)^\xi_y ) ' |^2  \, dt \to 0, \\
\ds  \int_{a(y)}^{b(y)} M_\eta \wedge | (v_j - \bar{u})^\xi_y | \, dt \leq \int_{B^\xi_y} M_\eta \wedge | (v_j - \bar{u})^\xi_y | \, dt\to 0. 
\end{cases}
\end{equation}
Up to another subsequence (still not relabeled), the first condition in \eqref{eq:convwj} ensures that $a_j(y) \to m$ and $b_j(y) \to m$ for some $m \in [a(y),b(y)]$. Thus, for all $\tau >0$, there exists $j_1=j_1(\tau) \geq j_0$ such that for all $j \geq j_1$,
$$ I_\tau := \left( a(y), m - \tau \right) \cup \left( m + \tau, b(y) \right) \subset I_j(y), $$
with the convention that $ ( x,y ) = \emptyset$ if $y<x$. We set 
$$I^-_\tau := \left( a(y), m - \tau \right), \quad I^+_\tau := \left( m + \tau, b(y) \right),$$
so that $ {(v_j)^\xi_y}_{|I^\pm_\tau} \in H^1(I^\pm_\tau)$ and the truncated function $w_j := \left( M_\eta \wedge (v_j)^\xi_y \right) \vee (-M_\eta) \in H^1(I^\pm_\tau)$ satisfies $w_j' = ( (v_j)^\xi_y )' \mathds{1}_{\left\{ \lvert (v_j)^\xi_y \rvert \leq M_\eta \right\} }$. According to the second condition in \eqref{eq:convwj}, the sequence $\{w_j\}_{j \in \N}$ is bounded in $H^1(I_\tau^\pm)$ and $w_j' \to 0$ in $L^2(I^\pm_\tau)$. As a consequence, up to a subsequence, there exist constants $c^\pm \in \R$ such that $w_j \to c^\pm$ in $H^1(I_\tau^\pm)$ and $\LL^1$-a.e. in $I_\tau^\pm$. Yet, as $(v_j)^\xi_y$ converges in measure to ${\bar{u}}^\xi_y$ in $I^\pm_\tau$, up to another subsequence (still not relabeled), we have that $(v_j)^\xi_y$ pointwise converges to ${\bar{u}}^\xi_y$ $\LL^1$-a.e. in $I^\pm_\tau$. Hence $c^\pm =  ( M_\eta \wedge u^\pm (x_0) \cdot \xi ) \vee (-M_\eta) = u^\pm (x_0) \cdot \xi$ by our choice \eqref{eq:M} of $M_\eta$. Thus, for all $\tau > 0$,
$$u^-(x_0) \cdot \xi \, \mathds{1}_{( a(y), m - \tau ) } + u^+(x_0) \cdot \xi \, \mathds{1}_{( m + \tau , b(y) ) } = {{\bar{u}}^\xi_y} _{|I_\tau} \quad \LL^1\text{-a.e. in } I_\tau. $$
Taking the limit as $\tau \to 0^+$, we obtain that
$$u^-(x_0) \cdot \xi \, \mathds{1}_{( a(y), m  ) } + u^+(x_0) \cdot \xi \, \mathds{1}_{( m , b(y) ) } = {{\bar{u}}^\xi_y}  \quad \LL^1\text{-a.e. in } ( a(y), b(y) ), $$
leading to $m=0$ since $[u](x_0) \cdot \xi \neq 0$ by our choice \eqref{eq:xi} of $\xi$. As a consequence $f_j(y)= ( |a_j(y)| + |b_j(y)| ) \mathds{1} _{Y_j}(y)\to  0$ which is against $\ell>0$.

\medskip

Using \eqref{eq:step1}, Lemma \ref{lemma afterSections} and owing to Egoroff's Theorem, we can find a set $Z_* \subset J_{\bar u}\cap B$ containing $Z'$ with  $\HH^1(Z_*) \leq 2 \eta$ such that for all $\gamma > 0$, there exists $j_0(\gamma) \in \N$ satisfying 
\begin{equation}\label{eq:betterAfterSections}
\begin{cases}
\ds \int_{B^\xi_y}  ( 1 - ( \chi_j ) ^\xi_y  )  |( (v_j)^\xi_y ) ' |^2  \, dt \leq \gamma^2, \\
\ds \int_{B^\xi_y} M_\eta \wedge | (v_j - \bar{u})^\xi_y | \, dt \leq \gamma, \\
\ds ( |a_j(y)| + |b_j(y)| ) \, \mathds{1}_{Y_j}(y) \leq \gamma
\end{cases}  \text{ for all } y \in J_{\bar u}\cap B_{1-\frac\eta2} \setminus Z_* \text{ and all } j \geq j_0(\gamma).
\end{equation}

\medskip

 \textbf{Step 2.} In this step, we show that for many points $y \in Y_j$, the variation of $(v_j)_y^\xi$ inside the only triangle $T$ in $\mathbf{T}^{x_0,j}_{b,int}$ which is crossed by $B_y^\xi$, is uniformly close with respect to $y$ to the jump of $\bar u_y^\xi$. More precisely, let 
 \begin{equation}\label{eq;choiceCgamma}
 C_\eta := 8 \left( 1  + \frac{1}{L_\eta} \right) > 0, \quad   \gamma_*=\gamma_*(\eta) := \frac12 \min \left(  1,\frac{M_\eta}{C_\eta}, L_\eta , \frac{ \left\lvert [u](x_0) \cdot \xi \right\rvert }{4 C_\eta} \right) > 0.
 \end{equation}
Let us show that for all $0 < \gamma < \gamma_*$, there exists $j_1(\gamma) \in \N$ such that
\begin{equation}\label{eq:C*}
\left\lvert  (v_j)^\xi_y(b_j(y)) - (v_j)_y^\xi(a_j(y)) - [u](x_0) \cdot \xi \right\rvert \leq C_\eta\gamma \quad \text{ for all } j \geq j_1(\gamma) \text{ and all }y \in Y_j \setminus Z_*.
\end{equation}

\medskip
 
Fix $0 < \gamma < \gamma_*$ and, by \eqref{eq:betterAfterSections}, let $j_0(\gamma) \in \N$ be such that
$$
\begin{cases}
\ds \int_{ (a(y) , b(y)) \setminus (a_j(y),b_j(y)) } |( (v_j)^\xi_y ) ' |^2  \, dt \leq \gamma^2, \\
\ds \int_{B^\xi_y} M_\eta \wedge | (v_j - \bar{u})^\xi_y | \, dt \leq \gamma, \\
\ds ( |a_j(y)| + |b_j(y)| )  \leq \gamma
\end{cases}
\qquad \text{for all } j \geq j_0(\gamma)\text{ and all } y \in Y_j \setminus Z_*.
$$
In particular, recalling \eqref{eq:AB(Y)}, \eqref{eq:L(eta)} and by the choice \eqref{eq;choiceCgamma} of $\gamma_*$, we get that $2 \geq |b(y)| \geq L_\eta > \gamma_* > \gamma \geq |b_j(y)|$ and $ 2 \geq |a(y)| \geq L_\eta > \gamma_* >\gamma \geq |a_j(y)|,$ hence 
$$ a(y) < a_j(y) \leq b_j(y) < b(y).$$
Writing
\begin{eqnarray*}
M_\eta \wedge \left\lvert  (v_j)^\xi_y(b_j(y)) - (v_j)_y^\xi(a_j(y))- [u](x_0) \cdot \xi \right\rvert &  \leq & \,  M_\eta \wedge \left\lvert (v_j)^\xi_y \left(  0\vee b_j(y)\right) - u^+(x_0) \cdot \xi \right\rvert   \\
  & & + M_\eta \wedge \left\lvert  (v_j)^\xi_y ( 0\vee b_j(y)) - (v_j)_y^\xi(b_j(y)) \right\rvert \\
 & & + M_\eta \wedge \left\lvert (v_j)^\xi_y \left(  0 \wedge a_j(y)  \right) - u^-(x_0) \cdot \xi \right\rvert \\
 & & + M_\eta \wedge \left\lvert  (v_j)^\xi_y (a_j(y)) - (v_j)_y^\xi( 0 \wedge a_j(y) ) \right\rvert \\
&  =: &  \, J_1 + J_2 + J_3 + J_4,
\end{eqnarray*}
it remains to control each of the last four terms. 

\medskip

Let us first estimate the terms $J_2$ and $J_4$. If $b_j(y) \geq 0$, $J_2 = 0$. Otherwise, by the Cauchy-Schwarz inequality,
$$ 
J_2 = M_\eta\wedge \left\lvert \int_{b_j(y)}^0 ( (v_j)^\xi_y ) ' \, dt \right\rvert \leq \sqrt{\left\lvert b_j(y) \right\rvert } \left( \int_{b_j(y)}^0 \lvert \left( (v_j)^\xi_y \right) ' \rvert ^2 \, dt \right)^\frac12   \leq \gamma^{3/2} \leq \gamma.
$$
Similarly, we have that $J_4 \leq \gamma$.

\medskip

Let us now estimate the term $J_1$. We consider the function 
$$z_j :=M_\eta \wedge |  (v_j)^\xi_y - u^+(x_0)\cdot\xi| \in H^1(B^\xi_y)$$
with 
$$z_j ' = ( (v_j)^\xi_y ) ' \mathds{1}_{\lbrace 0  \leq  (v_j)^\xi_y-u^+(x_0)\cdot\xi  \leq M_\eta \rbrace} - ( (v_j)^\xi_y ) ' \mathds{1}_{\lbrace 0 \leq u^+(x_0)\cdot\xi- (v_j)^\xi_y  \leq M_\eta \rbrace},$$
and the nonempty open interval $I^+ := ( 0 \vee b_j(y)  , b(y) ) $. By the Sobolev embedding and \eqref{eq;choiceCgamma}, we have  that for all $t \in I^+$,
\begin{eqnarray*}
| z_j(t)| & \leq  & \quad \sqrt{ b(y) - 0 \vee b_j(y) } \, \norme{z_j '}_{L^2(I^+)}   + \frac{1}{b(y) - 0 \vee b_j(y)} \, \norme{z_j }_{L^1(I^+)} \\
			& \leq &\sqrt 2 \norme{( (v_j)^\xi_y ) '}_{L^2 ( I^+) } + \frac{2}{L_\eta} \norme{M_\eta \wedge |  (v_j)^\xi_y - u^+(x_0)\cdot\xi|  }_{L^1(I^+)} \\
			& \leq & \left( \sqrt 2 + \frac{2}{L_\eta} \right) \gamma.
\end{eqnarray*}
By continuity of $(v_j)^\xi_y$ in $B^\xi_y$, the above inequality remains true up to the end point $0 \vee  b_j(y)$ of $I^+$, so that $J_1 \leq ( \sqrt 2 + \frac{2}{L_\eta} ) \gamma$. A similar argument shows that $J_3 \leq ( \sqrt 2 + \frac{2}{L_\eta} ) \gamma$, and thus $J_1 + J_2 + J_3 + J_4 \leq 8 ( 1 + \frac{1}{L_\eta} ) \gamma = C_\eta \gamma$, which shows that
$$M_\eta \wedge \left\lvert  (v_j)^\xi_y(b_j(y)) - (v_j)_y^\xi(a_j(y)) - [u](x_0) \cdot \xi \right\rvert  \le C_\eta\gamma.$$
Eventually, as $C_\eta\gamma < M_\eta$ for all $0 < \gamma < \gamma^*$ by \eqref{eq;choiceCgamma}, we conclude the validity of \eqref{eq:C*}. 

\medskip

\textbf{Step 3.} We now show that it is possible to include $Y_j \setminus Z_*$ inside a finite union of arbitrarily small segments contained in $J_{\bar u}\cap B$ (see Figure \ref{fig:F}). 

\medskip

Let $0< \gamma < \gamma_*$ and $j_1(\gamma) \in \N$ be given by \eqref{eq:C*}. For all $j \geq j_1(\gamma)$, we define
$$\widehat{\mathbf T}_j:=\{T \in \mathbf{T}^{x_0, j}_{b,int} : \text{ there exists $y \in Y_j \setminus Z_*$ such that $(\mathring T \cap B)_y^\xi\neq \emptyset$}\},$$
and, for all $T \in \widehat{\mathbf T}_j$, we introduce both following quantities :
\begin{equation}\label{eq:Lref Lmax}
 \begin{cases}
 L^{\rm ref}(T) := \frac{ \left\lvert [u](x_0) \cdot \xi \right\rvert - 4 C_\eta \gamma }{\left\lvert e(v_j)_{|T} : (\xi \otimes \xi) \right\rvert} \text{ the reference length of } T, \\
 L^{\rm max}(T) := \underset{ z \in p_\xi (T) }{\max} \, \LL^1(T^\xi_z) \text{ the maximal section's length of } T \text{ along the direction } \xi.
\end{cases}
\end{equation}
Note that because $T \in \mathbf{T}^{x_0, j}_b$, see \eqref{eq:new-triang}, then $ |e(v_j)_{|T} \xi \cdot \xi |^2 \geq (1-\delta)\kappa\varrho_j^2 /(\alpha A\e_{k_j})> 0,$ so that $L^{\rm ref}(T)$ is well defined, and positive by \eqref{eq;choiceCgamma} since $\gamma < \gamma_*$. The quantity $L^{\rm ref}(T)$ stands for the required length of the section $T^\xi_y$ in order for the (affine) function $(v_j)_y^\xi$ to pass exactly from the values $u^-(x_0)\cdot\xi$ to $u^+(x_0)\cdot\xi$ across $T$, up to the error $4 C_\eta \gamma $. Note that $\LL^1(T_y^\xi)=L^{\rm ref}(T)$ for at most two values of $y$, say $z_1$ and $z_2$, only depending on $j$ and $T$.
If $y \in Y_j \setminus Z_*$ is such that $(\mathring T \cap B)_y^\xi \neq \emptyset$, we know from Step 2 that the variation of $(v_j)_y^\xi$ across $T$ is close to $[u](x_0) \cdot \xi$, up to a small error of order $O(\gamma)$ which is uniform with respect to $y$. Therefore, we will show that if $y$ is far away from $z_1$ and $z_2$, then the variation of $(v_j)_y^\xi$ across $T$ is not sufficient to recover the full jump $[u](x_0)\cdot\xi$.

Let $x_1$, $x_2$ and $x_3 \in T$ be the three vertices of $T$ and $X_i := p_\xi(x_i) \in B^\xi$. We easily see that there exists $i_0 \in \{1,2,3\}$ such that $X_{i_0} = \underset{z \in p_\xi(T) }{\rm arg\, max} \, \LL^1(T^\xi_z)$. Up to a permutation of $\{x_1,x_2,x_3\}$, there is no loss of generality to assume that $i_0=3$ and $X_3 \cdot \xi^\perp \leq X_1 \cdot \xi^\perp,$ with $\xi^\perp \in \mathbb{S}^1$ being one of the two orthogonal vectors to $\xi$.

\medskip
Let $h_T \geq \frac{ \e_{k_j} \sin(\theta_0)}{\varrho_j} > 0$ be the smallest height of $T$. We claim that, for all $ z$, $z' \in p_\xi(T)$ be such that $\mathring{T}^\xi_{z} \neq \emptyset $, $\mathring{T}^\xi_{z'} \neq \emptyset $ and either $ z \cdot \xi^\perp , z' \cdot \xi^\perp \geq X_3 \cdot \xi^\perp $ or $ z \cdot \xi^\perp , z' \cdot \xi^\perp \leq X_3 \cdot \xi^\perp$, then,
\begin{equation}\label{eq:Horizontal tube's width}
\left\lvert z - z' \right\rvert \leq \frac{2 \LL^2(T)}{h_T} \, \frac{ \lvert \LL^1(T^\xi_{z}) - \LL^1(T^\xi_{z'}) \rvert}{ \max (  \LL^1(T^\xi_{z}), \LL^1(T^\xi_{z'}) ) } . 
\end{equation}
Indeed, consider for instance the case where 
$$X_1 \cdot \xi^\perp \geq z \cdot \xi^\perp > z' \cdot \xi^\perp \geq X_3 \cdot \xi^\perp$$ (see Figure \ref{fig:D}). Let 
$$L := \LL^1(T^\xi_{z}), \quad L' := \LL^1(T^\xi_{z'}) ,  \quad d := \left\lvert X_1 - z \right\rvert, \quad d':= \left\lvert X_1 - z' \right\rvert.$$

\begin{figure}[hbtp]
\begin{tikzpicture}[x=3.0cm,y=3.0cm]
\clip(-1.5,-0.5) rectangle (2.,1.5);
\fill[color=lightgray] (-0.25620804235572403,0.8964843406142776) -- (-0.41800145386041876,0.03775262377974413) -- (0.548637428767752,-0.06254758727748241) -- cycle;
\fill[color=lightgray] (0.21928450048786532,1.0454538174535781) -- (0.6581641652249057,0.36611225306443396) -- (1.3636992487793027,0.5813851017549257) -- cycle;

\draw [->, style= thick] (-1.2384796489614813,0.796342222294513) -- (-1.,0.8);
\draw [->,style= thick] (-1.2384796489614813,0.796342222294513) -- (-1.0103546855113323,0.8659477059678724);
\draw [->,style= thick] (-1.2384796489614813,0.796342222294513) -- (-1.242137426666968,1.0348218712559945);
\draw [->,style= thick] (-1.2384796489614813,0.796342222294513) -- (-1.3080851326348408,1.0244671857446621);

\draw [domain=-1.5:2.] plot(\x,{(--0.13322786327940744--0.2281249634501491*\x)/-0.0696054836733595});

\draw (-0.25620804235572403,0.8964843406142776)-- (-0.41800145386041876,0.03775262377974413);
\draw  (-0.41800145386041876,0.03775262377974413)-- (0.548637428767752,-0.06254758727748241);
\draw  (0.548637428767752,-0.06254758727748241)-- (-0.25620804235572403,0.8964843406142776);
\draw  (0.21928450048786532,1.0454538174535781)-- (0.6581641652249057,0.36611225306443396);
\draw  (1.3636992487793027,0.5813851017549257)-- (0.21928450048786532,1.0454538174535781);

\draw [<->] (-1.2842382819147675,0.41703420607839753) -- (-1.3305123026077985,0.5686926490945678);
\draw [<->] (-1.0004389562748202,0.21683303845987598) -- (-1.1267667015230471,0.6308595022759934);

\draw [dashed]  (-1.3305123026077985,0.5686926490945678)-- (1.4614382908488106,1.4205725508403377);
\draw [dashed] (-1.2842382819147675,0.41703420607839753)-- (1.6293541242491147,1.3060294549582139);
\draw [dashed] (-1.2041845573595715,0.15466618527845044)-- (1.7355083618723737,1.0516252221135574);
\draw [dashed] (-1.1045915881077635,-0.1717397424079074)-- (1.8438791049520855,0.7278975685784635);

\draw [style=very thick,color=ffqqqq] (-0.28934724204590195,0.7205953375961096)-- (-0.14544544652637403,0.764502650180909);
\draw [style=very thick,color=ffqqqq] (-0.3466778202892718,0.41630862144254716)-- (0.046173051385492964,0.5361752704969298);
\draw [style=very thick,color=qqqqff] (-0.41800145386041876,0.03775262377974413)-- (0.284561149836255,0.25211850948347203);
\draw [style=very thick,color=ffqqqq] (0.3108457530584493,0.9037262106481941)-- (0.45803798175612165,0.9486375001296721);
\draw [style=very thick,color=ffqqqq] (0.46561694261561254,0.6641559810748338)-- (0.8616165618800201,0.7849833759223453);
\draw [style=very thick,color=qqqqff] (0.6581641652249057,0.36611225306443396)-- (1.3636992487793027,0.5813851017549257);

\draw (-1.45,1.22) node[anchor=north west] {$\xi^\perp$};
\draw (-1.29,1.25) node[anchor=north west] {$\nu^\perp$};
\draw (-1.03,0.99) node[anchor=north west] {$\xi$};
\draw (-1,0.85) node[anchor=north west] {$\nu$};
\draw (-0.32,1.05) node[anchor=north west] {$x_1$};
\draw (0.1516245610619467,1.196710128318959) node[anchor=north west] {$x_1$};
\draw (1.3929887389589737,0.5928032309636494) node[anchor=north west] {$x_3$};
\draw (-0.44389474049676225,0.025242582060279923) node[anchor=north west] {$x_3$};
\draw (-0.8,0.12) node[anchor=north west] {$X_3$};
\draw (-0.83,0.29) node[anchor=north west] {$z'$};
\draw (-0.9,0.55) node[anchor=north west] {$z$};
\draw (-0.83,0.935) node[anchor=north west] {$X_1$};
\draw [color=ffqqqq](-0.2537759024404608,0.7437799553024768) node[anchor=north west] {$L$};
\draw [color=ffqqqq](-0.1978585971297839,0.4809686203422958) node[anchor=north west] {$L'$};
\draw [color=qqqqff](-0.15,0.15) node[anchor=north west] {$L^{\rm max}$};
\draw (0.3,0.1) node[anchor=north west] {$T$};
\draw (1.05,0.67) node[anchor=north west] {$T$};
\draw (-1.5,0.5816197699015141) node[anchor=north west] {$d$};
\draw (-1.25,0.4) node[anchor=north west] {$d'$};
\draw [color=ffqqqq](0.380885512835722,0.9366946586243118) node[anchor=north west] {$L$};
\draw [color=ffqqqq](0.6409009825303695,0.7186171679126723) node[anchor=north west] {$L'$};
\draw [color=qqqqff](0.9148957785526863,0.4502141024214236) node[anchor=north west] {$L^{\rm max}$};
\draw (-0.65,-0.35) node[anchor=north west] {$\Pi_\xi$};

\begin{scriptsize}
\draw [fill=uuuuuu] (-0.806332344035182,0.7286304261642903) circle (1pt);
\draw [fill=uuuuuu] (-0.7600583233421508,0.5769719831481199) circle (1pt);
\draw [fill=uuuuuu] (-0.680004598786955,0.3146039623481729) circle (1pt);
\draw [fill=uuuuuu] (-0.5804116295351469,-0.011801965338185005) circle (1pt);
\draw [fill=uuuuuu] (-0.25620804235572403,0.8964843406142776) circle (1pt);
\draw [fill=uuuuuu] (-0.41800145386041876,0.03775262377974413) circle (1pt);
\draw [fill=uuuuuu] (0.21928450048786532,1.0454538174535781) circle (1pt);
\draw [fill=uuuuuu] (1.3636992487793027,0.5813851017549257) circle (1pt);
\end{scriptsize}
\end{tikzpicture}
\caption{}
\label{fig:D}
\end{figure}
\noindent
Then, $L'>L>0$, $d'>d >0$ and using Thal\`es' Theorem, we have that 
$$ \frac{d}{d'}= \frac{d' - \left\lvert z - z' \right\rvert}{d'} = \frac{L}{L'}.$$
Since $d' \leq |X_1-X_3|=\HH^1(p_\xi([x_1,x_3])) \leq \lvert x_1 - x_3 \rvert \leq \frac{2 \LL^2(T)}{h_T}$, we obtain that 
 $$ \lvert z - z' \rvert =d' \, \frac{L' - L}{L'} \leq \frac{2 \LL^2(T)}{h_T} \,  \frac{\left\lvert L - L' \right\rvert}{L'},$$
so that \eqref{eq:Horizontal tube's width} holds in that case. The proof of the other case $X_2 \cdot \xi^\perp \leq z \cdot \xi^\perp,z' \cdot \xi^\perp \leq X_3 \cdot \xi^\perp$ is similar and we omit it.

\medskip

For all $j \geq j_1(\gamma)$ and all $T \in \widehat{\mathbf T}_j$, we have $L^{\rm max}(T) > L^{\rm ref}(T)$. 
Indeed, if such would not be the case, denoting by $y \in Y_j \setminus Z_*$ a point such that $(\mathring T \cap B)_y^\xi\neq \emptyset$, then $ \LL^1(T^\xi_{p_\xi(y)}) = \LL^1(T^\xi_y) = b_j(y) - a_j(y) \leq L^{\rm max}(T) \leq L^{\rm ref}(T) $, entailing that 
$$\lvert (v_j)^\xi_y (b_j(y)) -  (v_j)_y^\xi(a_j(y)) \rvert = \left\lvert e(v_j)_{|T} : \left( \xi \otimes \xi \right) \right\rvert \left( b_j(y) - a_j(y) \right) \leq \left\lvert [u](x_0) \cdot \xi \right\rvert - 4 C_\eta \gamma,$$ 
by definition \eqref{eq:Lref Lmax} of $L^{\rm ref}(T)$. Therefore, we would obtain that 
$$4 C_\eta \gamma \leq \left\lvert [u](x_0) \cdot \xi \right\rvert - \left\lvert (v_j)^\xi_y (b_j(y)) -  (v_j)_y^\xi(a_j(y))  \right\rvert \leq \left\lvert  (v_j)^\xi_y (b_j(y)) -  (v_j)_y^\xi(a_j(y)) - [u](x_0) \cdot \xi\right\rvert,$$
which is against \eqref{eq:C*}. Applying the Intermediate Value Theorem to the strictly monotone and continuous functions $y \in [X_1,X_3] \mapsto \LL^1(T_y^\xi) \in [0,L^{\rm max}(T)]$ and $y \in [X_2,X_3] \mapsto \LL^1(T_y^\xi)\in[0,L^{\rm max}(T)]$, there are at least one and at most two points $z^1_{\rm ref}$, $z^2_{\rm ref} \in p_\xi(T)$ (according to whether $T$ has an edge along the direction $\xi$ or not, see Figure \ref{fig:E}), only depending on $j$ and $T$, such that 
$$\LL^1(T^\xi_{z^1_{\rm ref}}) =\LL^1(T^\xi_{z^2_{\rm ref}}) = L^{\rm ref}(T).$$ (We set $z^1_{\rm ref}= z^2_{\rm ref}$ in the case where $T$ has an edge along the direction $\xi$). Without loss of generality, we can assume that $z^1_{\rm ref}\cdot \xi^\perp  \geq z^2_{\rm ref} \cdot \xi^\perp$.

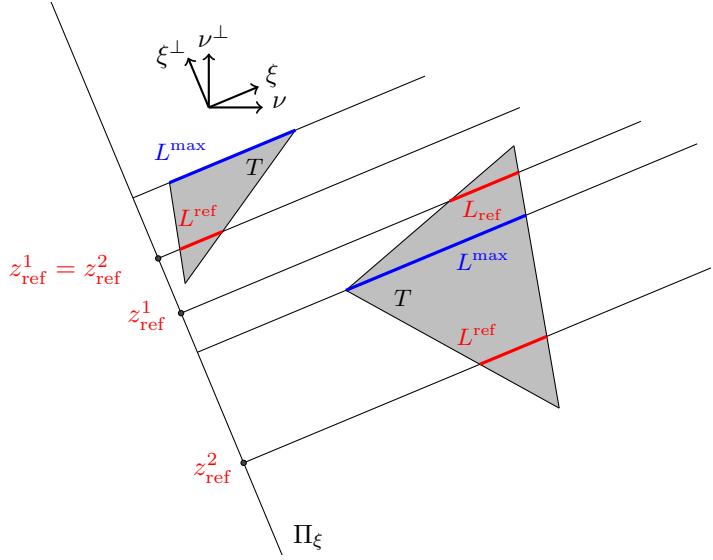
\begin{figure}[hbtp]
\begin{tikzpicture}[x=3.5cm,y=3.5cm]
\clip(-1.5,-0.9) rectangle (1.5,1.2);
\fill[color=lightgray] (-0.5545546205471268,0.5140503067130944) -- (-0.07636924132024087,0.7136780221796405) -- (-0.49509130381666416,0.1306011356644474) -- cycle;
\fill[color=lightgray] (0.11576920736101036,0.10490570114286517) -- (0.7537243377559556,0.6544687141207582) -- (0.9257010680572698,-0.34299632162686394) -- cycle;

\draw [->,style=thick] (-0.4052915917015789,0.8) -- (-0.2,0.8);
\draw [->,style=thick] (-0.4052915917015789,0.8) -- (-0.21584568941679844,0.8790878481871469);
\draw [->,style=thick] (-0.4052915917015789,0.8) -- (-0.4052915917015789,1.0052915917015792);
\draw [->,style=thick] (-0.4052915917015789,0.8) -- (-0.4843794398887259,0.9894459022847806);

\draw [domain=-2:2] plot(\x,{(--0.09529672431153875--0.18944590228478053*\x)/-0.07908784818714698});

\draw  (-0.6934262754045092,0.45607565419013574)-- (0.4160786680105971,0.9192599061028384);
\draw  (-0.5974570112698216,0.22619224918738087)-- (0.7761449435704861,0.79962889555345);
\draw  (-0.510583317571016,0.01809624373484487)-- (1.2362494486427327,0.7473452956091422);
\draw  (-0.4484850255094578,-0.1306531183935537)-- (1.4498104231126803,0.6618269534598709);
\draw  (-0.2730116789362557,-0.550979467263578)-- (1.544239515071211,0.20766711635862323);

\draw (-0.5545546205471268,0.5140503067130944)-- (-0.07636924132024087,0.7136780221796405);
\draw (-0.07636924132024087,0.7136780221796405)-- (-0.49509130381666416,0.1306011356644474);
\draw (-0.49509130381666416,0.1306011356644474)-- (-0.5545546205471268,0.5140503067130944);
\draw (0.11576920736101036,0.10490570114286517)-- (0.7537243377559556,0.6544687141207582);
\draw (0.7537243377559556,0.6544687141207582)-- (0.9257010680572698,-0.34299632162686394);
\draw (0.9257010680572698,-0.34299632162686394)-- (0.11576920736101036,0.10490570114286517);

\draw [style= very thick,color=ffqqqq] (0.5091998461483372,0.44382440709625437)-- (0.7711852745905473,0.5531952804801267);
\draw [style= very thick,color=ffqqqq] (0.624363338682626,-0.17635294695122986)-- (0.8786654316152407,-0.07018963026309614);
\draw [style= very thick,color=qqqqff] (-0.5545546205471268,0.5140503067130945)-- (-0.0763692413202408,0.7136780221796405);
\draw [style= very thick,color=ffqqqq] (-0.5152378638741342,0.2605162190852993)-- (-0.353225534409452,0.32815139283473);
\draw [style= very thick,color=qqqqff] (0.11576920736101037,0.1049057011428652)-- (0.799279260135693,0.3902501643182807);

\draw (-0.64,1.09) node[anchor=north west] {$\xi^\perp$};
\draw (-0.48186232456271727,1.15) node[anchor=north west] {$\nu^\perp$};
\draw (-0.22449854634956176,1.0001058629920712) node[anchor=north west] {$\xi$};
\draw (-0.20044585679693042,0.8774371462736513) node[anchor=north west] {$\nu$};
\draw [color=ffqqqq](-1.2,0.2761199074578674) node[anchor=north west] {$z_{\rm ref}^1 = z_{\rm ref}^2$};
\draw [color=qqqqff](-0.65,0.7) node[anchor=north west] {\small $L^{\rm max}$};
\draw [color=ffqqqq](-0.56,0.465) node[anchor=north west] {\small $L^{\rm ref}$};
\draw [color=ffqqqq](0.52,0.48) node[anchor=north west] {\small $L_{\rm ref}$};
\draw [color=qqqqff](0.49708214022937885,0.28333571432365684) node[anchor=north west] {\small $L^{\rm max}$};
\draw [color=ffqqqq](0.5,0.01394559133418566) node[anchor=north west] {\small $L^{\rm ref}$};
\draw [color=ffqqqq](-0.75,0.1) node[anchor=north west] {$z^1_{\rm ref}$};
\draw [color=ffqqqq](-0.5,-0.48) node[anchor=north west] {$z^2_{\rm ref}$};
\draw (-0.12,-0.75) node[anchor=north west] {$\Pi_\xi$};
\draw (-0.3,0.64) node[anchor=north west] {\small $T$};
\draw (0.2637710515688547,0.14142484596313185) node[anchor=north west] {\small $T$};
\begin{scriptsize}
\draw [fill=uuuuuu] (-0.5974570112698216,0.22619224918738087) circle (1pt);
\draw [fill=uuuuuu] (-0.510583317571016,0.01809624373484487) circle (1pt);
\draw [fill=uuuuuu] (-0.2730116789362557,-0.550979467263578) circle (1pt);
\end{scriptsize}
\end{tikzpicture}
\caption{Two possible configurations of $T$. }
\label{fig:E}
\end{figure}
\noindent
Let us introduce the following segments (orthogonal to $\xi$) associated to $T$ (see Figure \ref{fig:F}),
\begin{equation}\label{eq:Horizontal Tubes}
\mathfrak{T}_i(T) := \left\{ z \in \Pi_\xi: \; \left\lvert z - z^i_{\rm ref} \right\rvert \leq C'_\eta \, \frac{  \varrho_j \LL^2(T) }{\e_{k_j} } \, \gamma \right\} \quad \text{ for } i \in \{ 1,2 \},
\end{equation}
where
$$C'_\eta:= \frac{20 C_\eta }{\sin\theta_0 \left\lvert [u](x_0) \cdot \xi \right\rvert }$$
is a constant only depending on $\eta$.

For every $j \geq j_1(\gamma)$ and every $y \in Y_j \setminus Z_*$, let $T \in  \mathbf{T}^{x_0, j}_{b,int}$ be such that $(\mathring T \cap B)_y^\xi\neq \emptyset$. Note that $T \in \widehat{\mathbf T}_j$. If $p_\xi(y) \cdot \xi^\perp \geq X_3 \cdot \xi^\perp$ and $z^1_{\rm ref} \cdot \xi^\perp \geq X_3 \cdot \xi^\perp$ (the other cases being treated similarly), applying \eqref{eq:Horizontal tube's width} above, with $z = p_\xi(y)$ and $z' = z^1_{\rm ref}$, we get that
$$
\begin{aligned}
\left\lvert p_\xi(y) - z^1_{\rm ref} \right\rvert  & \leq \frac{2 \LL^2(T)}{h_T} \, \frac{  \left\lvert \left( b_j(y) - a_j(y) \right) - L^{\rm ref}(T) \right\rvert  }{\max \left( b_j(y) - a_j(y), L^{\rm ref}(T) \right) } \\ 
& \leq \frac{2 \LL^2(T)}{h_T} \, \frac{  \left\lvert   | (v_j)^\xi_y (b_j(y)) - (v_j)_y^\xi(a_j(y))| - \left\lvert [u](x_0) \cdot \xi \right\rvert + 4 C_\eta \gamma  \right\rvert  }{ \left\lvert e(v_j)_{|T} : \left( \xi \otimes \xi \right) \right\rvert \,  L^{\rm ref}(T)  } \\
& \leq \frac{2 \LL^2(T)}{h_T} \, \frac{5 C_\eta \gamma}{\left\lvert [u](x_0) \cdot \xi \right\rvert - 4 C_\eta \gamma} \\
& \leq C_\eta' \, \frac{  \varrho_j \LL^2(T) }{\e_{k_j} } \, \gamma,
\end{aligned}
$$
where we also used \eqref{eq:C*} and \eqref{eq;choiceCgamma}.

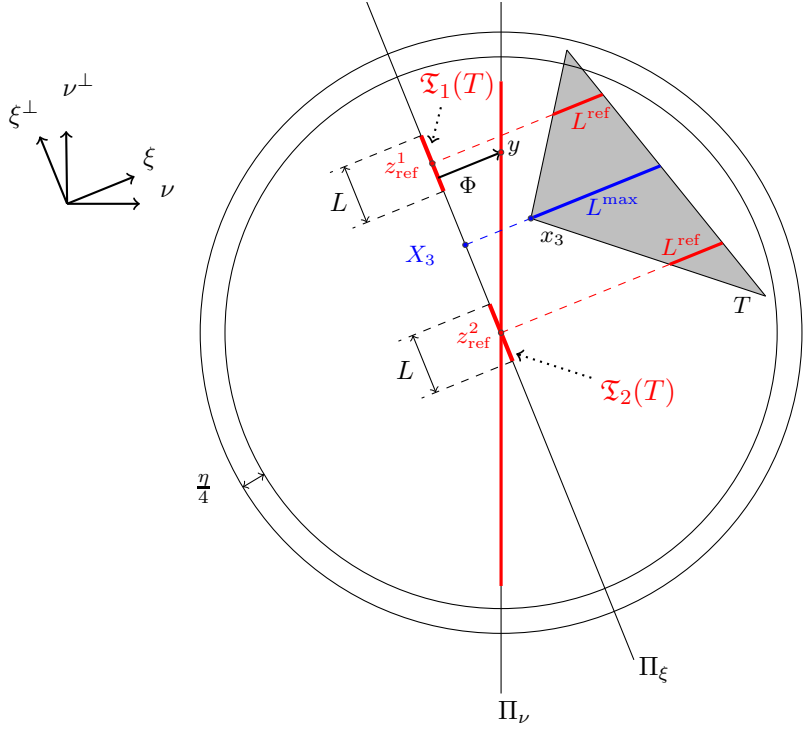
\begin{figure}[hbtp]
\begin{tikzpicture}[x=4.cm,y=4.cm]
\clip(-2,-1.3) rectangle (2.,1.1);
\fill[color=lightgray] (0.21878285399439446,0.9404773286429337) -- (0.09904806014744462,0.38042748645558916) -- (0.8792554265049887,0.12164583523798858) -- cycle;
\draw (0.,0.) circle (1);
\draw  (0.,0.) circle (2.7527556080451836/3);

\draw [->,style=thick] (-1.4422538785401653,0.4288370738523108) -- (-1.2000177737259419,0.4288370738523108);
\draw [->,style=thick] (-1.4422538785401653,0.4288370738523108) -- (-1.21779479267978,0.5199219202884091);
\draw [->,style=thick] (-1.4422538785401653,0.4288370738523108) -- (-1.4454834062146524,0.6710516494388497);
\draw [->,style=thick] (-1.4422538785401653,0.4288370738523108) -- (-1.5333387249762636,0.6532961597126955);

\draw (0.44153477372299726,-1.088067835245956)-- (-0.5010790495237435,1.2348019434692232);
\draw (0.,0.8362098787841941)-- (0.,1.2782028522424405);
\draw (0.,-0.8423695410328859)-- (0.,-1.2);

\draw [style=very thick,color=ffqqqq] (0.,-0.8423695410328859)-- (0.,0.8362098787841941);
\draw (0.21878285399439446,0.9404773286429337)-- (0.09904806014744462,0.38042748645558916);
\draw (0.09904806014744462,0.38042748645558916)-- (0.8792554265049887,0.12164583523798858);
\draw (0.8792554265049887,0.12164583523798858)-- (0.21878285399439446,0.9404773286429337);

\draw [dashed,color=qqqqff] (-0.11854503601020781,0.29212883873944023)-- (0.09904806014744462,0.38042748645558916);
\draw [style=very thick,color=qqqqff] (0.09904806014744462,0.38042748645558916)-- (0.5295954773477779,0.5551423803919565);
\draw [style=very thick,color=ffqqqq] (0.33787610061365586,0.7928295609044347)-- (0.1729043913416452,0.7258845194607201);
\draw [dashed,color=ffqqqq] (0.1729043913416452,0.7258845194607201)-- (-0.22846742088120148,0.5630090014572458);
\draw [style=very thick,color=ffqqqq] (0.7363539137251748,0.29881028383050606)-- (0.5603947357018815,0.22740655941525653);
\draw [dashed,color=ffqqqq] (0.5603947357018815,0.22740655941525653)-- (0.,0.);
\draw [style=ultra thick,color=ffqqqq] (-0.26600170342595364,0.6555041977282419)-- (-0.19093313833644948,0.47051380518624986);
\draw [style=ultra thick,color=ffqqqq] (-0.038462948898429945,0.09478369549970211)-- (0.03846294889842994,-0.09478369549970211);
\draw [dashed] (-0.26600170342595364,0.6555041977282419)-- (-0.5686370846343108,0.5326956372378937);
\draw [dashed] (-0.19093313833644948,0.47051380518624986)-- (-0.4935685195448068,0.3477052446959018);
\draw [dashed] (-0.038462948898429945,0.09478369549970211)-- (-0.3410983301067872,-0.028024864990645944);
\draw [dashed] (0.03846294889842994,-0.09478369549970211)-- (-0.2641724323099275,-0.2175922559900502);

\draw [<->] (-0.4461305965039151,0.3669554163646694) -- (-0.5211991615934193,0.5519458089066613);
\draw [<->] (-0.2167345092690358,-0.19834208432128256) -- (-0.29366040706589563,-0.008774693321878302);

\draw (-0.6,0.4972981603740395) node[anchor=north west] { $L$};
\draw (-0.38,-0.06269198966233744) node[anchor=north west] { $L$};
\draw [color=ffqqqq](-0.42,0.6292703613380264) node[anchor=north west] {\small $z^1_{\rm ref}$};
\draw [color=ffqqqq](-0.3,0.9) node[anchor=north west] { \large $\mathfrak{T}_1(T)$};
\draw [->,style=dotted, style=thick] (-0.2,0.75) -- (-0.23,0.63);

\draw [color=ffqqqq](0.3,-0.12) node[anchor=north west] {\large $\mathfrak{T}_2(T)$};
\draw [color=ffqqqq](-0.18,0.063) node[anchor=north west] {\small $z^2_{\rm ref}$};
\draw [->,style=dotted, style=thick] (0.3,-0.15) -- (0.05,-0.07);

\draw [color=qqqqff](-0.35,0.31895734826054367) node[anchor=north west] {\small $X_3$};
\draw (-0.01,0.67) node[anchor=north west] {\small $y$};

\draw [->,style=thick] (-0.20905320108205616,0.515166816952209) -- (0.,0.6);

\draw (0.09886339416920005,0.3653259594100526) node[anchor=north west] {\small $x_3$};
\draw [color=ffqqqq](0.2,0.78) node[anchor=north west] {\small $L^{\rm ref}$};
\draw [color=ffqqqq](0.5,0.36) node[anchor=north west] {\small $L^{\rm ref}$};
\draw [color=qqqqff](0.25,0.48) node[anchor=north west] {\small $L^{\rm max}$};
\draw (-0.040242439279326754,-1.193372738461901) node[anchor=north west] {$\Pi_\nu$};
\draw (0.4270104884580325,-1.0471332725288345) node[anchor=north west] {$\Pi_\xi$};
\draw (0.7408903177777854,0.15131698487385756) node[anchor=north west] {\small $T$};
\draw (-1.6667106457544094,0.8040443572092524) node[anchor=north west] {$\xi^\perp$};
\draw (-1.4883698336409135,0.8932147632660002) node[anchor=north west] {$\nu^\perp$};
\draw (-1.2208586154706695,0.6506712587916459) node[anchor=north west] {$\xi$};
\draw (-1.163789555594351,0.518699057827659) node[anchor=north west] {$\nu$};
\draw (-0.17,0.55) node[anchor=north west] {\small $\Phi$};
\draw (-1.05,-0.44) node[anchor=north west] {$\frac{\eta}{4}$};

\draw [<->] (-0.7874469224832926,-0.4710519593973247) -- (-0.858173083199147,-0.5133604574492174);

\begin{scriptsize}
\draw [color=uuuuuu,fill=qqqqff] (0.09904806014744462,0.38042748645558916) circle (1pt);
\draw [color=uuuuuu,fill=qqqqff] (-0.11854503601020781,0.29212883873944023) circle (1pt);
\draw [color=uuuuuu,fill=ffqqqq] (-0.22846742088120148,0.5630090014572458) circle (1pt);
\draw [color=uuuuuu,fill=ffqqqq] (0.,0.) circle (1pt);
\draw [color=uuuuuu,fill=ffqqqq] (0.,0.6) circle (1pt);
\end{scriptsize}
\end{tikzpicture}
\caption{The length of $\mathfrak T_i(T)$ is given by $L = \frac{2C'_\eta\varrho_j \LL^2(T)}{\e_{k_j}} \, \gamma$.}
\label{fig:F}
\end{figure}

We have just shown that for all $j \geq j_1(\gamma)$ and all $y \in Y_j \setminus Z_*$, there exists $T \in \widehat{\mathbf T}_j$ such that $p_\xi(y) \in \mathfrak{T}_1(T) \cup \mathfrak{T}_2(T)$. Since $y \in \Pi_\nu$, then $y=\Phi(p_\xi(y)) \in \Phi \left(   \mathfrak{T}_1(T) \cup \mathfrak{T}_2(T)  \right)$, with $\Phi$ introduced in \eqref{eq:Phi}. Recalling that the Lipschitz constant of $\Phi$ is less than $\sqrt{1 + 4 \eta^2} \leq 2$ for $\eta$ small enough, we deduce that
$$\HH^1( \Phi \left( \mathfrak{T}_1(T) \cup \mathfrak{T}_2(T) \right) ) \leq 2 \HH^1( \mathfrak{T}_1(T) \cup \mathfrak{T}_2(T) ) \leq 8 C'_\eta \, \frac{  \varrho_j \LL^2(T) }{\e_{k_j} } \, \gamma.$$ Together with the fact that each triangle in $\widehat{\mathbf T}_j \subset \mathbf{T}^{x_0, j}_{b,int}$ is contained in $B$, we obtain that for all $j \geq j_1(\gamma)$,
\begin{multline*}
\HH^1 ( Y_j \setminus Z_* ) \leq \sum_{ T \in \widehat{\mathbf T}_j } \HH^1( \Phi \left( \mathfrak{T}_1(T) \cup \mathfrak{T}_2(T) \right) ) \\
 \leq 8 C'_\eta \, \gamma \,  \frac{\varrho_j}{ \e_{k_j}}\sum_{ T \in \widehat{\mathbf T}_j } \LL^2(T)   \leq  \frac{8 C'_\eta\gamma}{\kappa (1-\delta)}  \frac{(1-\delta)\kappa\varrho_j}{\e_{k_j}} \int_{B} \chi_j \, dx  \leq \frac{ 8 C'_\eta \gamma}{\kappa (1-\delta)} \frac{\lambda_{k_j}(B_{\varrho_j}(x_0) ) }{\varrho_j}.
\end{multline*}
Possibly taking a larger $j_1(\gamma) \in \N$, we finally get that  for all $j \geq j_1(\gamma)$,
$$ 
\HH^1 ( Y_j \setminus Z_* ) \leq \frac{ 8 C'_\eta}{\kappa(1-\delta)}  \, \left(2 \frac{d\lambda}{d\HH^1\res J_u}(x_0) + 1\right) \gamma =: C_* \gamma,
$$
for some constant $C_*>0$ only depending on $\eta$.
\end{proof}

We are now in position to prove Lemma \ref{lemma2}.

\begin{proof}[Proof of Lemma \ref{lemma2}]
Let $j_0 = j_0(\eta) \in \N$ such that $\frac{\eta}{2^{j_0} C_*} < \gamma_*$, where $C_*$ and $\gamma_*$ are given by Lemma \ref{lemma3}. For all $j \geq j_0$, as $0 < \gamma_j := \frac{\eta}{2^{j} C_*} < \gamma_*$, Lemma \ref{lemma3} ensures the existence of an integer $i(\gamma_j) \geq j_0$ such that for all $i \geq i(\gamma_j)$, 
$$\HH^1(Y_i \setminus Z_*) \leq C_* \gamma_j = \frac{\eta}{2^j}.$$
Thereby, we define the following extraction
\begin{equation}\label{eq:RecursivePhi}
\begin{cases}
& \phi(j_0) := i(\gamma_{j_0}) \geq j_0, \\
& \phi(j+1) := \max \left( \phi(j) + 1, i(\gamma_{j+1}) \right) \quad \text{ for all } j \geq j_0.
\end{cases}
\end{equation}
Since $\phi(j) \geq i(\gamma_j)$, then  $\HH^1(Y_{\phi(j)} \setminus Z_*) \leq \frac{\eta}{2^j}$. Hence, we set 
\begin{equation}\label{eq:Z''}
 Z'' :=  Z_* \, \cup \, \bigcup_{j = j_0}^{+ \infty} Y_{\phi(j)}, 
\end{equation}
which satisfies 
$$ \HH^1(Z'') \leq  2\eta + \underset{j \geq j_0}{\sum} \frac{\eta}{2^j} \leq 3 \eta.$$
Moreover, for all $j \geq j_0$ and all $y \in J_{\bar u} \cap B_{1-\frac\eta2} \setminus Z''$, Lemma \ref{lemma1} ensures that 
$$ \# \left\{ T \in \mathbf{T}^{x_0, \phi(j)}_{b,int}: \; ( \mathring{T} \cap B )^\xi_y \neq \emptyset \right\} \geq 1,$$
and since $y \notin Y_{\phi(j)}$ for all $j \geq j_0$, it actually follows that 
$$\# \left\{ T \in \mathbf{T}^{x_0, \phi(j)}_{b,int}: \; ( \mathring{T} \cap B )^\xi_y \neq \emptyset \right\} \geq 2,$$
concluding the proof of Lemma \ref{lemma2}.
\end{proof}

Let us consider the further subsequence introduced in Lemma \ref{lemma2}. In order to derive a lower bound for the surface energy without the factor $1/2$, we now construct two disjoint subfamilies $\mathscr F_j^1$ and $\mathscr F_j^2$ from $ \mathscr{F}_j $ (see \eqref{eq:Fj}) with the property that both sets
$$\bigcup_{T \in \mathscr F_j^1} T, \qquad \bigcup_{T \in \mathscr F_j^2} T$$
project onto $B^\nu=J_{\bar u} \cap B$, thanks to the mapping $\Phi \circ p_\xi$, into two sets of almost full $\HH^1$ measure in $J_{\bar u} \cap B$. This is the object of the following technical result.

\begin{lem}\label{lemma F_1 F_2}
Let $K \subset J_{\bar u} \cap B_{1-\frac\eta2} \setminus Z''$ be a compact set. For all $j \in \N$, there exist two disjoint subfamilies $\mathscr F_j^1$ and $\mathscr F_j^2$ of $\mathscr{F}_j$ such that
$$K \subset \Phi \left(  \bigcup_{T \in \mathscr F_j^1}  p_\xi( \mathring{T}) \right) \cap  \Phi \left(  \bigcup_{T \in \mathscr F_j^2}  p_\xi( \mathring{T}) \right) .$$
\end{lem}

\begin{proof} 
For the sake of clarity, we omit to write the explicit dependance on $j$ for the different objects considered herafter (triangles, intervals, and so forth).

For all $y \in J_{\bar u} \cap B_{1-\frac\eta2} \setminus Z''$, we consider a pair of distinct triangles of $\mathscr{F}_j$ satisfying
\begin{eqnarray}\label{eq:T1 T2}
\left\{ T^1(y),T^2(y) \right\} & \in & {\rm arg\; min } \Big\{ \HH^1 \left( \Phi ( p_\xi( \mathring{T^1} )  \cap  p_\xi( \mathring{T^2} ) ) \right) :\nonumber\\ 
&&  T^1, \,T^2 \in \mathbf{T}^{x_0, j}_{b,int} ,\,  \mathring{T^1} \cap \mathring{T^2} =\emptyset, \, ( \mathring{T^1} \cap B ) ^\xi_y \neq \emptyset, \, ( \mathring{T^2} \cap B ) ^\xi_y \neq \emptyset  \Big\}.
\end{eqnarray}
Note that Lemma \ref{lemma2} ensures that the set
 $$\left\{ \lbrace T^1, T^2 \rbrace \subset \mathbf{T}^{x_0, j}_{b,int} :\, ( \mathring{T^i} \cap B ) ^\xi_y \neq \emptyset \text{ for all } i \in \{ 1,2 \}  \right\}$$
 is nonempty and finite, hence the minimum in \eqref{eq:T1 T2} is achieved and we have at our disposal such a pair of distinct triangles $\left\{ T^1(y),T^2(y) \right\}$. Then, we introduce the following open segment in $B^\nu= J_{\bar u } \cap B$
\begin{equation}\label{eq:I(j,y)}
I(y) := \Phi \left( p_\xi( \mathring{T^1}(y) ) \cap p_\xi( \mathring{T^2}(y) ) \right) \subset B^\nu = J_{\bar u } \cap B.
\end{equation}
Since $y \in I(y)$, it follows that 
$$K \subset J_{\bar u} \cap B_{1-\frac\eta2} \setminus Z'' \subset \bigcup_{y \in J_{\bar u} \cap B_{1-\frac\eta2} \setminus Z''} I(y).$$
Furthermore, $I(y)$ is optimal in the sense that any other triangle $T \in \mathbf{T}^{x_0, j}_{b,int}$ satisfying $( \mathring{T} \cap B ) ^\xi_y \neq \emptyset$ is such that
\begin{equation}\label{eq:Optimality I(j,y)}
 I(y) \subset \Phi ( p_\xi (\mathring{T}) ) . 
\end{equation}

\begin{figure}[hbtp]
\begin{tikzpicture}[x=4.cm,y=4.cm]
\fill[color=lightgray] (0.07306512448896575,0.569172626755291) -- (0.4293154217009433,0.6095837250233429) -- (0.6245117390498757,-0.964380230742652) -- cycle;
\fill[color=lightgray] (1.2875595471875192,-0.29513571411774076) -- (0.3425615028791955,0.9720911715285031) -- (0.8351998276169774,0.9720911715285031) -- cycle;
\fill[color=qqqqff,fill opacity=0.10000000149011612] (1.1977072741221373,0.8605504187576847) -- (0.6864788239225523,0.0085030017583764) -- (1.5788048460891009,-0.33541431928498083) -- cycle;
\fill[color=ffqqqq,fill opacity=0.10000000149011612] (0.271299355275617,-0.4066764668885593) -- (0.5873314881262692,-0.10613610525607603) -- (0.10708658036302307,0.8977306696812908) -- cycle;

\draw  (0.,1.)-- (0.,-1.);

\draw [->,style=thick] (-0.6,0.4) -- (-0.40595443396530506,0.4);
\draw [->,style=thick] (-0.6,0.4) -- (-0.4234563427807303,0.4805358230439732);
\draw [->,style=thick] (-0.6,0.4) -- (-0.6,0.594045566034695);
\draw [->,style=thick] (-0.6,0.4) -- (-0.6805358230439732,0.5765436572192698);

\draw  (-0.001355818164161618,0.5352232231761306)-- (1.0030123163356806,0.9933966609826284);
\draw  (0.,-0.8838230204081721)-- (1.7749465189441744,-0.07412653505292466);
\draw  (0.,0.09215856633649022)-- (1.6981856895341896,0.8668382367495103);
\draw [color=qqqqff] (0.,0.3431252600708319)-- (1.3209516791443512,0.945718025831192);
\draw [color=ffqqqq] (0.,-0.6297579724301965)-- (1.7677020968683588,0.17663374682048838);

\draw (-0.11007988525894939,0.15666066831759476) node[anchor=north west] {$y$};
\draw [color=qqqqff](-0.10162982823085709,0.41297906483639496) node[anchor=north west] {$z$};
\draw [color=ffqqqq](-0.11007988525894939,-0.5559608077181903) node[anchor=north west] {$z$};
\draw (0.74,0.658030718651072) node[anchor=north west] { ${T_{i_a}}$};
\draw (0.25,0.15947735399362553) node[anchor=north west] { ${ T_{i_b}}$};
\draw (-0.08472971417467245,0.6101470621585489) node[anchor=north west] {$b$};
\draw (-0.11571325661101094,-0.8179125755890521) node[anchor=north west] {$a$};

\draw [color=qqqqff] (1.1977072741221373,0.8605504187576847)-- (0.6864788239225523,0.0085030017583764);
\draw [color=qqqqff] (0.6864788239225523,0.0085030017583764)-- (1.5788048460891009,-0.33541431928498083);
\draw [color=qqqqff] (1.5788048460891009,-0.33541431928498083)-- (1.1977072741221373,0.8605504187576847);
\draw [color=qqqqff](1.3,-0.04332401468058998) node[anchor=north west] { $ T$};

\draw [color=ffqqqq] (0.271299355275617,-0.4066764668885593)-- (0.5873314881262692,-0.10613610525607603);
\draw [color=ffqqqq] (0.5873314881262692,-0.10613610525607603)-- (0.10708658036302307,0.8977306696812908);
\draw [color=ffqqqq] (0.10708658036302307,0.8977306696812908)-- (0.271299355275617,-0.4066764668885593);
\draw [color=ffqqqq](0.24,-0.20387509821434394) node[anchor=north west] { $ T$};

\draw (-0.031212686330087818,-0.9869137161508983) node[anchor=north west] {$\Pi_\nu$};
\draw (-0.7438341623658727,0.7143644321716874) node[anchor=north west] {$\xi^\perp$};
\draw (-0.631166735324642,0.7368979175799336) node[anchor=north west] {$\nu^\perp$};
\draw (-0.39456513853805725,0.46931277835701035) node[anchor=north west] {$\nu$};
\draw (-0.3917484528620265,0.5876135767503028) node[anchor=north west] {$\xi$};
\begin{scriptsize}
\draw [fill=uuuuuu] (0.,0.09215856633649022) circle (1pt);
\draw [color =qqqqff ,fill=qqqqff] (0.,0.3431252600708319) circle (1pt);
\draw [color=ffqqqq ,fill=ffqqqq] (0.,-0.6297579724301965) circle (1pt);
\end{scriptsize}
\end{tikzpicture}

\caption{}
\label{fig:G}
\end{figure}

\noindent
Indeed, setting $J := \Phi ( p_\xi (\mathring{T}) ) $, there exist points $a$, $b$, $\tilde{a}$ and $\tilde{b}$ in $B^\nu= J_{\bar u } \cap B$ such that $I(y) = (a,b)$, $J = (\tilde{a},\tilde{b})$ with $a \cdot \nu^\perp < b \cdot \nu^\perp$ and $\tilde{a} \cdot \nu^\perp < \tilde{b} \cdot \nu^\perp$. By construction, there exist $i_a$ (resp. $ i_b $) $\in \{ 1,2 \}$ such that $a$ (resp. $b$) is the image by $\Phi \circ p_\xi$ of a vertex of $T^{i_a}(y)$ (resp. $T^{i_b}(y)$), see Figure \ref{fig:G}. Assume by contradiction that there exists a point $z \in I(y) \setminus J$. If $y \cdot \nu^\perp < z \cdot \nu^\perp$, then $J \subset (\tilde{a},z)$ since $J$ is a segment containing $y$. In particular, $\Phi ( p_\xi(\mathring{T}^{i_a}(y)) ) \cap J  \subset (a,z)$. Together with \eqref{eq:T1 T2} and recalling that $J := \Phi ( p_\xi (\mathring{T}) ) $, it ensures that 
$$\HH^1((a,z)) \geq \HH^1 \left(   \Phi ( p_\xi(\mathring{T}^{i_a}(y)) ) \cap  \Phi ( p_\xi (\mathring{T}) )  \right) \geq \HH^1(I(y)) = \HH^1((a,z)) + \HH^1((z,b)) > \HH^1((a,z)),$$
which is impossible. A similar argument shows that the other situation $y \cdot \nu^\perp > z \cdot \nu^\perp$ is also impossible. This shows the validity of \eqref{eq:Optimality I(j,y)}.

\medskip

By compactness of $K$, there exist an integer $N = N(j,K) \geq 1$ and points $y_1, \ldots, y_N \in J_{\bar u} \cap B_{1-\frac\eta2} \setminus Z''$ such that 
\begin{equation}\label{eq:K's cover}
K \subset \bigcup_{i=1}^{N} I(y_i).
\end{equation}
Up to relabeling the points $y_i$, we can assume that $y_1 \cdot \nu ^\perp < \cdots < y_N \cdot \nu^\perp$ (See Figure \ref{fig:H1}). Let us now construct two disjoint subfamilies $\mathscr{F}_j^1$ and $\mathscr{F}_j^2$ of $\mathscr{F}_j$ by induction in $N$ iterations.

\medskip 

\textbf{Iteration 1.} Set $\mathscr F^1(1) := \lbrace T^1(y_1) \rbrace$ and $\mathscr F^2(1) := \lbrace T^2(y_1) \rbrace$. Clearly $\mathscr F^1(1) \cap \mathscr F^2(1)=\emptyset$ and, for all $k \in \{ 1,2 \}$, there is $T \in \mathscr F^k(1)$ such that $ ( B \cap \mathring{T} ) ^\xi_{y_1} \neq \emptyset $ and $I(y_1) \subset \Phi ( p_\xi ( \mathring{T} ) )$ .

\medskip

\textbf{Iteration 2.}  We distinguish two cases:
\begin{enumerate}
\item[i)] If $ \lbrace T^1(y_2), T^2(y_2) \rbrace \cap  \left( \mathscr F^1(1) \cup \mathscr F^2(1) \right) = \emptyset$, then we set $\mathscr F^1(2) := \mathscr F^1(1) \cup \lbrace T^1(y_2) \rbrace$ and $\mathscr F^2(2) := \mathscr F^2(1) \cup \lbrace T^2(y_2) \rbrace$. We have that $\mathscr F^1(2) \cap \mathscr F^2(2) = \emptyset$ and, for all $i,k \in \{ 1,2 \}$, there exists $T \in \mathscr F^k(2)$ such that $ ( B \cap \mathring{T} ) ^\xi_{y_i} \neq \emptyset $ and $I(y_i) \subset \Phi ( p_\xi ( \mathring{T} ) )$.
\item[ii)] Otherwise, there exist $i,k \in \{ 1,2 \}$ such that $T^i(y_2) \in \mathscr F^k(1)$, i.e. $T^i(y_2)=T^k(y_1)$, and $T^{3-i}(y_2) \notin \mathscr F^k(1)$. In that case, we set $\mathscr F^k(2) := \mathscr F^k(1) $ and $\mathscr F^{3-k}(2) := \mathscr F^{3-k}(1) \cup \lbrace T^{3-i}(y_2) \rbrace $. Note that, it might be the case that $T^{3-i}(y_2)  \in  \mathscr F^{3-k}(1)$. We have that $\mathscr F^1(2) \cap \mathscr F^2(2) = \emptyset$ and, for all $i,k \in \{ 1,2 \}$, there exists $ T \in \mathscr F^k(2) $ such that $ ( B \cap \mathring{T} ) ^\xi_{y_i} \neq \emptyset $ and $I(y_i) \subset \Phi (  p_\xi ( \mathring{T} ) ) $.
\end{enumerate}

\medskip
\textbf{Iteration $n+1$ for some $n \in \{ 1,\ldots,N-1 \}$.} Assume that we have constructed two disjoint subfamilies $\mathscr F^1(n)$ and $\mathscr F^2(n)$ of $\mathscr{F}_j$ with the following properties: for all $k \in \{ 1,2 \}$ and all $i \in \{ 1,\ldots,n \}$, there exists $T \in \mathscr F^k(n)$ such that $ ( B \cap \mathring{T} ) ^\xi_{y_i} \neq \emptyset $ and $I(y_i) \subset \Phi ( p_\xi ( \mathring{T} ) ) $. Let us now construct $\mathscr F^1(n+1)$ and $\mathscr F^2(n+1)$:
\begin{enumerate}
\item[i)] If $ \lbrace T^1(y_{n+1}), T^2(y_{n+1}) \rbrace \cap  \left( \mathscr F^1(n) \cup \mathscr F^2(n) \right) = \emptyset$, then we set $\mathscr F^1(n+1) := \mathscr F^1(n) \cup \lbrace T^1(y_{n+1}) \rbrace$ and $\mathscr F^2(n+1) := \mathscr F^2(n) \cup \lbrace T^2(y_{n+1}) \rbrace$. In that case, we have that $\mathscr F^1(n+1) \cap \mathscr F^2(n+1) = \emptyset$ and that for all $k \in \{ 1,2 \}$ and all $i \in \{ 1,\ldots,n+1 \}$, there exists $ T \in \mathscr F^k(n+1) $ such that $ ( B \cap \mathring{T} ) ^\xi_{y_i} \neq \emptyset $ and $I(y_i) \subset  \Phi ( p_\xi ( \mathring{T} ) ) $. For $i \in \{ 1,\ldots,n \}$, it follows from the previous iteration $n$ and because $\mathscr F^k(n) \subset \mathscr F^k(n+1)$, while for $i = n+1$, it is a consequence of the fact that $T^k(y_{n+1}) \in \mathscr F^k(n+1)$.

\item[ii)]  Otherwise, there exist $p,q \in \{ 1,2 \}$ such that $T^p(y_{n+1}) \in \mathscr F^q(n)$. Let us further distinguish two subcases:
		\begin{enumerate}
		\item If $T^{3-p}(y_{n+1}) \notin \mathscr F^q(n)$, then we set $\mathscr F^q(n+1) := \mathscr F^q(n)$ and $\mathscr F^{3-q}(n+1) := \mathscr F^{3-q}(n) \cup \lbrace T^{3-p}(y_{n+1}) \rbrace$. Then $\mathscr F^1(n+1) \cap \mathscr F^2(n+1) = \emptyset$ and, for all $k \in \{ 1,2 \}$ and all $i \in \{ 1,\ldots,n+1 \}$, there exists $ T \in \mathscr F^k(n+1) $ such that $ ( B \cap \mathring{T} ) ^\xi_{y_i} \neq \emptyset $ and $I(y_i) \subset \Phi ( p_\xi ( \mathring{T} ) ) $. Indeed, for $i \in \{ 1,\ldots,n \}$ this is a consequence of the previous iteration $n$ and of the fact that $\mathscr F^k(n) \subset \mathscr F^k(n+1)$, while, for $i = n+1$, it results from $T^p(y_{n+1}) \in \mathscr F^q(n+1)$ and $T^{3-p}(y_{n+1}) \in \mathscr F^{3-q}(n+1)$.
		
		\item If both $T^1:=T^1(y_{n+1}) \in \mathscr F^q(n)$ and $T^2:=T^2(y_{n+1})  \in \mathscr F^q(n)$, we introduce the indexes 
		$$i_1 := {\rm arg\, min} \left\{ i \in \{ 1,\ldots,n + 1 \}: \; ( B \cap \mathring{T^1} )^\xi_{y_i} \neq \emptyset \right\}$$
		and 
		$$i_2 := {\rm arg\, min} \left\{ i \in \{ 1,\ldots,n + 1 \}: \; ( B \cap \mathring{T^2} )^\xi_{y_i} \neq \emptyset \right\}.$$
		Up to interchanging $T_1$ and $T_2$, there is no loss of generality to assume that $i_1 \geq i_2$ (see Figure \ref{fig:H1}). We set $\mathscr F^q(n+1) := \mathscr F^q(n) \setminus \{ T^1 \}$ and $\mathscr F^{3-q}(n+1) := \mathscr F^{3-q}(n) \cup \{ T^1 \}$. Once more, we have $\mathscr F^1(n+1) \cap \mathscr F^2(n+1) = \emptyset$ and, for all $k \in \{ 1,2 \}$ and all $i \in \{ 1,\ldots,n+1 \}$, there exists $T \in \mathscr F^k(n+1)$ such that $( B \cap \mathring{T} ) ^\xi_{y_i} \neq \emptyset$ and $I(y_i) \subset \Phi ( p_\xi (\mathring{T}) ) $. This is immediate for $i = n+1$ because $T^2 = T^2(y_{n+1}) \in \mathscr F^q (n+1)$ and $T^1 = T^1(y_{n+1}) \in \mathscr F^{3-q}(n+1)$. For $i \in \{ 1, \ldots,n \}$, there are two possibilities:
				\begin{itemize}
		\item for $k=3-q$, it follows from $\mathscr F^{3-q}(n) \subset \mathscr F^{3-q}(n+1)$.
		
		\item for $k = q$ and $i \in \{ 1,\ldots,i_1 - 1 \}$, by the previous iteration $n$ there exists $T \in \mathscr F^q(n)$ such that $( B \cap \mathring{T} ) ^\xi_{y_i} \neq \emptyset$ and $I(y_i) \subset \Phi ( p_\xi (\mathring{T}) ) $. As $ ( B \cap \mathring{T^1} ) ^\xi_{y_i} = \emptyset$, by definition of $i_1$, this implies that $T \neq T^1$, so that actually $T \in \mathscr F^q(n+1)$ satisfies the above requirements. Assuming next that $i \in \{ i_1, \ldots,n \}$, we deduce that $n + 1 > i \geq  i_2$. By definition of $i_2$, we have $ ( B \cap \mathring{T^2} ) ^\xi_{y_{i_2}} \neq \emptyset$ while $ ( B \cap \mathring{T^2} ) ^\xi_{y_{n+1}} \neq \emptyset$ so that the convexity of $\mathring{T^2}$ together with the ordering of the points $y_i$ lead to $ ( B \cap \mathring{T^2} )^\xi_{y_i} \neq \emptyset$. Hence, owing to \eqref{eq:Optimality I(j,y)}, we infer that $I(y_i) \subset \Phi ( p_\xi(\mathring{T^2}) ) $ and $T^2 \in \mathscr F^q(n+1)$ satisfies the above requirements.
		\end{itemize}
				\end{enumerate}
\end{enumerate}

We proceed this construction up to the $N^{\text{th}}$ iteration, and finally define
\begin{equation}\label{eq:F1 F2}
\mathscr F_j^k := \mathscr F^k(N) \text{ for } k \in \{ 1,2 \},
\end{equation}
which define two disjoint subfamilies of $\mathscr{F}_j$ satisfying in particular, thanks to \eqref{eq:K's cover},
\begin{eqnarray*}
K & \subset &  \bigcup_{i=1}^{N} I(y_i) \subset \left(\bigcup_{T \in \mathscr F_j^1} \Phi(p_\xi( \mathring{T}) ) \right) \cap \left(\bigcup_{T \in \mathscr F_j^2} \Phi(p_\xi( \mathring{T}) )\right)\\
& = & \Phi \left(  \bigcup_{T \in \mathscr F_j^1}  p_\xi( \mathring{T}) \right) \cap  \Phi \left(  \bigcup_{T \in \mathscr F_j^2}  p_\xi( \mathring{T}) \right).
\end{eqnarray*}
The proof of Lemma \ref{lemma F_1 F_2} is now complete.
\end{proof}

		\begin{figure}[hbtp]
		\begin{tikzpicture}[x=4.5cm,y=4.5cm]
		\clip(-1.,-0.7) rectangle (1.3,1.05);
		\fill[color=lightgray] (-0.22821299524667416,0.3434825407804412) -- (0.2650878907808107,0.80095710547632) -- (0.3202053082140492,0.3214355738071458) -- cycle;
		\fill[color=lightgray] (0.4276842722088643,0.9277271655727684) -- (0.35878750041731616,0.015533907052672626) -- (1.0560228309477835,0.803712976347982) -- cycle;
		\draw [->,style=thick] (-0.8,0.4) -- (-0.6,0.4);
		\draw [->,style=thick] (-0.8,0.4) -- (-0.6202623213555118,0.4877175402954889);
		\draw [->,style=thick] (-0.8,0.4) -- (-0.8,0.6);
		\draw [->,style=thick] (-0.8,0.4) -- (-0.8877175402954889,0.5797376786444882);
		\draw  (0.2573854274200193,0.9982217888171208)-- (-0.5434150231799666,0.6074064276535364);
		\draw (-0.5479538233334597,0.3325816048098055)-- (0.803774133181341,0.9922666101372176);
		\draw  (-0.5473321258268862,-0.031111410813596085)-- (1.2056611997720112,0.8244034919914525);
		\draw  (-0.44376428445646454,-0.15143109846945946)-- (1.2371899271834532,0.6689264910265251);
		\draw  (-0.29513453844827764,-0.3440347731436621)-- (1.0405438166228842,0.3078175308159775);
		\draw [style=thick,color=ffqqqq] (0.10214923339314715,0.6498519818654145)-- (0.2728791167266823,0.7331734397472381);
		\draw [style=thick,color=ffqqqq] (0.18610116521878667,0.3268266951335886)-- (0.3124954858211417,0.38851102862544107);
		\draw [style=thick,color=ffqqqq] (0.41835206428342087,0.8041687326398954)-- (0.6013115639728692,0.8934586211456622);
		\draw [style=thick,color=ffqqqq] (0.3898077335815111,0.42624179414661345)-- (0.9745512947603623,0.7116147180491582);
		\draw [style=thick,color=ffqqqq] (0.3764087053494639,0.24883866035430807)-- (0.708575329299268,0.41094623535400865);
		\draw (-0.94,0.7194870369806564) node[anchor=north west] {$\xi^\perp$};
		\draw (-0.8293808812098884,0.7451230450891394) node[anchor=north west] {$\nu^\perp$};
		\draw (-0.5939957158501804,0.5819848116715205) node[anchor=north west] {$\xi$};
		\draw (-0.6009873544252213,0.47244914066254773) node[anchor=north west] {$\nu$};
		\draw (-0.14,0.9548722023403639) node[anchor=north west] {$y_N$};
		\draw (-0.21,0.7101648522139353) node[anchor=north west] {$y_{n+1}$};
		\draw (-0.14,0.3349469153534117) node[anchor=north west] {$y_{i_1}$};
		\draw (-0.14,0.1554948585940308) node[anchor=north west] {$y_{i_2}$};
		\draw (-0.21,-0.11717904583256092) node[anchor=north west] {$y_{i_2 -1}$};
		\draw (-0.14,-0.32692820308378534) node[anchor=north west] {$y_1$};
		\draw (-0.1,-0.573966099401894) node[anchor=north west] {$\Pi_\nu$};
		\draw (0.15177906548750814,0.5680015345214389) node[anchor=north west] {$T_1$};
		\draw (0.5992439342901213,0.7707590531976224) node[anchor=north west] {$T_2$};
		\draw (-0.09758937702228147,0.85) node[anchor=north west] {$\vdots$};
		\draw (-0.09059773844724064,-0.18) node[anchor=north west] {$\vdots$};
		\draw  (-0.17488074962146666,-0.4926792555709062)-- (0.9000274514743517,0.03190915657873158);
		\draw (0.,1.)-- (0.,-0.576529700212742);
		\begin{scriptsize}
		\draw [fill=uuuuuu] (0.,0.8726097481395301) circle (1pt);
		\draw [fill=uuuuuu] (0.,0.6) circle (1pt);
		\draw [fill=uuuuuu] (0.,0.23600357678562633) circle (1pt);
		\draw [fill=uuuuuu] (0.,0.06513958274258713) circle (1pt);
		\draw [fill=uuuuuu] (0.,-0.2) circle (1pt);
		\draw [fill=uuuuuu] (0.,-0.4073320466967544) circle (1pt);
		\end{scriptsize}
		\end{tikzpicture}
		
		\caption{}
		\label{fig:H1}
		\end{figure}		
		
We are now in position to complete the proof of Proposition \ref{lower bound}.

\begin{proof}[Proof of Proposition \ref{lower bound}]
Since all triangles $T$ in $\mathscr{F}_j $ are contained in $B$, we get from \eqref{seq:2h(x_0)} and \eqref{eq:lambda_k greater than},
$$\frac{d\lambda}{d\HH^1\res J_u}(x_0)=\lim_{j \to \infty}\frac{\lambda_{k_j}(B_{\varrho_j}(x_0))}{2\varrho_j} \geq \liminf_{j \to \infty}\frac{(1-\delta) \kappa\varrho_j}{2 \e_{k_j}} \int_{B} \chi_j \, dx
 \geq \liminf_{j \to \infty} \frac{(1 - \delta) \kappa\varrho_j}{2\e_{k_j}} \sum_{T \in \mathscr{F}_j} \LL^2(T).$$
Using next the inequality $\LL^2(T) \geq \HH^1(p_\xi(T)) (\e_{k_j}/\varrho_j)\sin\theta_0/2$, we deduce that
 $$2\frac{d\lambda}{d\HH^1\res J_u}(x_0) \geq  \frac{(1 - \delta) \kappa\sin\theta_0}{2} \liminf_{j \to \infty}\sum_{T \in \mathscr{F}_j} \HH^1(p_\xi(T)).$$
Let $K \subset J_{\bar u} \cap B_{1-\frac\eta2} \setminus Z''$ be a compact set and $\mathscr F_j^1$ and $\mathscr F_j^2$ be two disjoint subfamilies of $\mathscr{F}_j$ given by Lemma \ref{lemma F_1 F_2}. Thus
  $$2\frac{d\lambda}{d\HH^1\res J_u}(x_0)  \geq  \frac{(1 - \delta) \kappa\sin\theta_0}{2}\liminf_{j \to \infty}\left\{\sum_{T \in \mathscr{F}^1_j} \HH^1(p_\xi(T))+ \sum_{T \in \mathscr{F}^2_j} \HH^1(p_\xi(T))\right\},$$
and remembering that $\Phi$ has a Lipschitz constant bounded by $\sqrt{1+4\eta^2}$,
\begin{eqnarray*}
2\frac{d\lambda}{d\HH^1\res J_u}(x_0)  & \geq & \frac{(1 - \delta) \kappa\sin\theta_0}{2\sqrt{1+4\eta^2}} \liminf_{j \to \infty}\left\{\sum_{T \in \mathscr{F}^1_j} \HH^1(\Phi(p_\xi(T)))+\sum_{T \in \mathscr{F}^2_j} \HH^1(\Phi(p_\xi(T)))\right\}\\
& \geq & \frac{(1 - \delta) \kappa\sin\theta_0}{2\sqrt{1+4\eta^2}}\liminf_{j \to \infty} \left\{ \HH^1 \left( \bigcup_{T \in \mathscr F_j^1} \Phi \left(  p_\xi(T) \right) \right) + \HH^1 \left( \bigcup_{T \in \mathscr F_j^2} \Phi \left( p_\xi(T) \right) \right) \right\}\\
& \geq & \frac{(1 - \delta) \kappa\sin\theta_0}{\sqrt{1+4\eta^2}}\HH^1(K).
\end{eqnarray*}
By inner regularity of the Radon measure $\HH^1 \res (J_{\bar u} \cap B_{1-\frac\eta2} \setminus Z'')$, passing to the supremum with respect to all compact sets $K \subset J_{\bar u} \cap B_{1-\frac\eta2} \setminus Z''$, we get that
 $$2\frac{d\lambda}{d\HH^1 \res J_u}(x_0)  \geq  \frac{(1 - \delta) \kappa\sin\theta_0}{\sqrt{1+4\eta^2}}\HH^1(J_{\bar u} \cap B_{1-\frac\eta2} \setminus Z'').$$
 Remembering that $J_{\bar u} \cap B=B^\nu$, we have
 $$2=\HH^1(J_{\bar u} \cap B)=\HH^1(J_{\bar u} \cap B_{1-\frac\eta2})+\eta \leq \HH^1(J_{\bar u} \cap B_{1-\frac\eta2} \setminus Z'')+4\eta$$
 because $\HH^1(Z'')\leq 3\eta$. Hence
$$2\frac{d\lambda}{d\HH^1 \res J_u}(x_0) \geq  \frac{(1 - \delta) \kappa\sin\theta_0}{\sqrt{1+4\eta^2}} (2-4\eta).$$
Finally passing to the limit as $\eta \to 0$ and $\delta \to 0$, we deduce that
$$\frac{d\lambda}{d\HH^1 \res J_u}(x_0) \geq \kappa\sin\theta_0,$$
which corresponds to the desired lower bound with the correct multiplicative constant.
\end{proof}

\subsection{The upper bound}

The proof of the $\Gamma$-$\limsup$ inequality relies on suitable approximation results in $GSBD$ (see \cite{C,I2,CCdensity,CT}) which allow us to reduce to the case where the jump set of $u$ is a finite union of pairwise disjoint closed line segments, and $u$ is smooth outside the jump set. Then, an explicit mesh construction introduced in \cite{CDM}, adapted to this simple geometrical situation, provides the desired upper bound.

\begin{prop}\label{upper bound}
For all $u \in L^0(\O;\R^2) $, 
$$ \F''(u) \leq \F(u) .$$
\end{prop}

\begin{proof}
We can assume that $\F(u)<\infty$, and thus that $u \in GSBD^2(\O)$. Using the density result for $GSBD$ functions (see \cite[Theorem 1.1]{CCdensity}) as well as the lower semicontinuity of $\F''$ with respect to the convergence in measure (see \cite[Proposition 6.8]{DM}), we can further assume without loss of generality that $u \in SBV^2(\O;\R^2) \cap L^\infty(\O;\R^2)$. 

Writing $u = (u_1,u_2)$, we can apply \cite[Lemma 4.2]{CDM} to both components $u_1$ and $u_2 \in SBV^2(\O) \cap L^\infty(\O)$ of $u$. For $\O' := (a,b) \times (c,d) \subset \R^2 $ with $\O \subset\subset \O'$, we can find an extension $v \in SBV^2(\O';\R^2) \cap L^\infty(\O';\R^2)$ such that
\begin{equation}\label{eq:extension}  
{v}_{|\O} = u , \quad \norme{v}_{L^\infty(\O';\R^2)} \leq \sqrt{2} \norme{u}_{L^\infty(\O;\R^2)} \quad \text{and} \quad \HH^1(\partial \O \cap J_{v} ) = 0.
\end{equation}
Next owing to the density result in $SBV$ (see \cite[Theorem 3.1]{CT}), there exists a sequence $\{v_k\}_{k \in \N}$ in $SBV^2(\O';\R^2) \cap L^\infty(\O';\R^2)$ as well as $N_k$ disjoint closed segments $L^k_1,\ldots,L_{N_k}^k  \subset \bar{\O'}$ with the following properties:
$$\bar{J_{v_k}} = \bigcup_{i=1}^{N_k} L^k_i  , \quad \HH^1( \bar{J_{v_k}} \setminus J_{v_k}) = 0, \quad v_k \in W^{2,\infty}(\O' \setminus \bar{J_{v_k}};\R^2)$$
and 
\begin{equation}\label{eq:densitySBV}
\begin{cases}
v_k \to v \quad \text{ strongly in } L^1(\O';\R^2), \\
\nabla v_k \to  \nabla v \quad \text{ strongly in } L^2(\O'; \mathbb M^{2 \times 2}), \\
\limsup_k \HH^1(\bar{A} \cap J_{v_k} ) \leq \HH^1(\bar{A} \cap J_v) \text{ for all open subset } A \subset \subset \O'.
\end{cases}
\end{equation}
Using \eqref{eq:densitySBV} and the lower semicontinuity of $\F ''$ in $L^0(\O;\R^2)$ with respect to the convergence in measure, we obtain that 
$$ \F ''(u) \leq \liminf_{k \to \infty} \F ''({v_k}_{|\O}).$$
The proof is complete once we know that $\liminf_k \F''({v_k}_{|\O}) \leq \F(u)$. This follows from Lemma \ref{LemmaUpperBound} below, applied to each function $v_k$. Indeed, using that result, we get that
$$
\liminf_{k \to \infty} \F''({v_k}_{|\O}) \leq \liminf_{k \to \infty} \left\{ \int_\O \mathbf A e(v_k): e(v_k) \, dx +  \kappa \sin \theta_0 \HH^1(J_{v_k} \cap \bar \O) \right\}.
$$
Recalling the convergences \eqref{eq:densitySBV}, we conclude that
$$
\F''(u) \leq  \int_\O \mathbf A e(v): e(v) \, dx +  \kappa \sin \theta_0 \HH^1(J_{v} \cap \bar \O) = \int_\O \mathbf A e(u): e(u) \, dx + \kappa \sin \theta_0 \HH^1(J_u)=\F(u),
$$
where we used $\HH^1(J_{v} \cap \partial \O) = 0$ and that $v = u$ in $\O$.
\end{proof}

We are back to establishing the following result.

\begin{lem}\label{LemmaUpperBound}
Let $v \in SBV^2(\O';\R^2) \cap L^\infty(\O';\R^2)$ be such that
$$ \bar{J_v} =  \bigcup_{i=1}^{N} L_i  , \quad \HH^1( \bar{J_{v}} \setminus J_{v}) = 0, \quad v \in W^{2,\infty}(\O' \setminus \bar{J_{v}};\R^2),$$
for some pairwise disjoint closed segments $L_1,\ldots, L_N \subset \bar{\O'} $. Then, 
$$ \F''(v_{|\O}) \leq  \int_\O \mathbf A e(v) : e(v) \, dx +  \kappa \sin\theta_0 \HH^1(J_v \cap \bar{\O}).$$
\end{lem}

\begin{proof} Since $\O\subset \subset \O'$, then $d:={\rm dist}(\O,\R^2 \setminus \O')>0$. For all $\delta \in (0,d)$, let us consider the open sets
$$\O_\delta:=\{x \in \O' : \; {\rm dist}(x,\R^2\setminus \O')>\delta\}$$
which satisfy $\O \subset\subset \O_\delta \subset\subset \O'$. We introduce a cut-off function $\phi_\delta \in C^\infty_c(\R^2;[0,1])$ which is supported in $\O'$ and such that $\phi_\delta=1$ in $\O_{\delta}$, $\phi_\delta=0$ in $\R^2 \setminus \O_{\frac\delta2}$. We next introduce the function $ \bar{v} := \phi_\delta v \in SBV^2(\R^2;\R^2) \cap L^\infty(\R^2;\R^2)$. We remark that 
\begin{equation}\label{eq:Sv}
 \bar{v} \in W^{2,\infty} \left(\R^2 \setminus  \bigcup_{i=1}^N \bar{L_i} ;\R^2 \right), \quad J_{\bar{v}} \subset J_v, \quad \text{and} \quad  J_v \setminus J_{\bar{v}} \subset J_v \setminus \O_\delta.
\end{equation}
Since $J_v \subset \bar{J_v} $ and $\HH^1(\bar{J_v}) < \infty$, the disjoint closed segments $L_i \subset \bar{\O'}$  satisfy 
$$\HH^1(J_v) = \HH^1(\bar{J_v}) = \sum_{i=1}^{N} \HH^1(L_i).$$
Then according to \cite[Appendix A]{CDM}, since $\theta_0 $ is smaller than or equal to $\Theta_0 := 45^\circ - \arctan (1/2) $, for all $\e >0$ there exists an admissible triangulation $\mathbf{T}_\e \in  \mathcal T_{\e}(\R^2,\omega,\theta_0)$ such that, setting $\mathbf T'_\e := \lbrace T \in \mathbf T_\e: \; T \cap  \bigcup_i L_i  \neq \emptyset \rbrace$,
\begin{itemize}
\item The vertices of $\mathbf T_\e$ are never situated on any $L_i$ :
$$ \text{ for all } i \in \{ 1,\ldots,N \}, \quad L_i \cap {\rm Vertices}(\mathbf T_\e) = \emptyset,$$
\item Using \cite[Formula (4.9)]{CDM},
\begin{equation}\label{eq:CV_LL2(Deps)}
\sum_{ T \in \mathbf T'_\e }\frac{\LL^2(T)}{\e}\to\sin \theta_0 \HH^1(J_v).
\end{equation}
\end{itemize}

We define the set $D_\e := {\bigcup}_{T \in \mathbf T'_\e} T$ and $\chi_\e := \mathds{1}_{D_\e} \in L^\infty(\R^2;\lbrace 0,1 \rbrace)$, while $\bar v_\e$ is the Lagrange interpolation of the values of $\bar v$ at the vertices of the triangulation $\mathbf T_\e$. Note that, if  $x_1$, $x_2$ and $x_3$ are the vertices of $T \in \mathbf T_\e$, the values $\bar v (x_i)$ are well defined since, by construction of the triangulation $\mathbf T_\e$, the points $x_1$, $x_2$ and $x_3$ do not belong to $\bigcup_i L_i$. 
\bigskip
In particular, $\bar v_\e \in V_{\e} (\O',\omega,\theta_0)$ and
\begin{equation}\label{eq:App_v_eps}
\begin{cases}
\chi_\e\to 0 & \text{ strongly in } L^1(\O'), \\
\bar v_\e\to \bar v & \text{ strongly in } L^2(\O';\R^2), \\
e(\bar v_\e) \mathds{1}_{\O' \setminus D_\e}\to e(\bar v)&  \text{ strongly in } L^2(\O';\mathbb M^{2 \times 2}_{\rm sym}).
\end{cases}
\end{equation}
Indeed, the first convergence is a consequence of \eqref{eq:CV_LL2(Deps)} since 
$$\norme{\chi_\e}_{L^1(\O')} \leq \LL^2(D_\e) =\sum_{ T \in \mathbf T'_\e }\LL^2(T) \to 0.$$
Next, noticing that every $T \in \mathbf T_\e \setminus \mathbf T'_\e$ is contained in $\R^2 \setminus \bigcup_i  L_i$ and $\bar v \in W^{2,\infty} ( \R^2 \setminus \bigcup_i  L_i; \R^2 )$, we infer that for all $\e >0$ and $T \in \mathbf T_\e \setminus \mathbf T'_\e$, 
\begin{equation}\label{eq:LagrangeEstimate}
\norme{ \bar v_\e - \bar v }_{H^1(T;\R^2)} \leq C \e  \norme{D^2 \bar v}_{L^2(T)},
\end{equation}
for some constant $C = C(\theta_0) >0$ depending only on $\theta_0$ (see e.g. \cite[Theorem 3.1.5]{Ciarlet}). On the one hand, since $\norme{\bar v_\e }_{L^\infty(T;\R^2)} \leq \norme{\bar v}_{L^\infty(T;\R^2)}$ for all $T \in \mathbf T_\e$, we get that 
$$\norme{ \bar v_\e - \bar v }^2_{L^2(\O';\R^2)} \leq 2 \norme{\bar v}^2_{L^\infty(\R^2;\R^2)} \sum_{T \in \mathbf T'_\e} \LL^2(T) +\sum_ { T \in \mathbf T_\e \setminus \mathbf T'_\e, \, T \cap \O' \neq \emptyset } \int_T \lvert \bar v_\e - \bar v \rvert^2 \, dx.$$
Then, using \eqref{eq:LagrangeEstimate} yields
$$
\begin{aligned}
\norme{ \bar v_\e - \bar v }^2_{L^2(\O';\R^2)} & \leq 2 \norme{\bar v}^2_{L^\infty(\R^2;\R^2)}  \LL^2(D_\e) + C ^2\e^2 \underset{ T \in \mathbf T_\e \setminus \mathbf T'_\e, \, T \cap \O' \neq \emptyset }{\sum}  \norme{D^2 \bar v}^2_{L^2(T)} \\
& \leq 2 \norme{\bar v}^2_{L^\infty(\R^2;\R^2)}  \LL^2(D_\e) + C^2\e^2   \norme{D^2 \bar v}^2_{L^2 ( \R^2 \setminus \bigcup_i  L_i) }  \to 0,
\end{aligned}
$$
leading to the second convergence in \eqref{eq:App_v_eps}. Note in particular that $\bar v_\e$ converges in measure to $\bar v$ in $\O'$. On the other hand, writing
$$\norme{ e( \bar v_\e) \mathds{1}_{\O' \setminus D_\e} - e(\bar v) }_{L^2(\O';\mathbb M^{2 \times 2}_{\rm sym})}^2 = \int_{\O' \cap D_\e} \lvert e(\bar v) \rvert ^2 \, dx + \int_{\O' \setminus D_\e} \lvert e( \bar v_\e) - e (\bar v) \rvert ^2 \, dx,$$
and using that $\LL^2(D_\e) \to 0$, that $\nabla \bar v \in L^2(\R^2;\mathbb M^{2 \times 2})$ (because $\bar v \in SBV^2(\R^2;\R^2)$) as well as \eqref{eq:LagrangeEstimate}, we get that
$$
 \norme{ e( \bar v_\e) \mathds{1}_{\O' \setminus D_\e} - e(\bar v) }_{L^2(\O';\mathbb M^{2 \times 2}_{\rm sym})}^2  \leq \int_{\O' \cap D_\e} \lvert e(\bar v) \rvert ^2 \, dx + C^2\e^2\norme{D^2 \bar v}_{L^2 ( \R^2 \setminus \bigcup_i  L_i) }^2  \to 0,
$$
which implies the third convergence in \eqref{eq:App_v_eps}.

\medskip

We next show  that 
\begin{equation}\label{eq:BadUpperBound}
\F''(v_{|\O}) \leq  \int_\O \mathbf A e(v):e(v) \, dx +  \kappa \sin \theta_0 \left( \HH^1(J_v \cap \bar \O) +  \HH^1(J_v  \setminus \O_\delta)  \right).
\end{equation}
Indeed, as $f \leq \kappa$ thanks to the growth properties \eqref{eq:f}, we get
$$
 \int_{\O'} \frac1\e f \left(  \e  \mathbf A e (\bar v_\e) : e (\bar v_\e) \right) \, dx \leq \sum_{T \in \mathbf T_\e \setminus \mathbf T'_\e, \, T \cap \O' \neq \emptyset }  \LL^2(T \cap \O')  \frac1\e f \left( \e  \mathbf A {e (\bar v_\e)}_{|T} : {e(\bar v_\e)}_{|T} \right) +  \frac{\kappa}{\e} \sum_{T \in \mathbf T'_\e}  \LL^2(T).
$$
On the one hand, \eqref{eq:CV_LL2(Deps)} implies that 
\begin{equation}\label{eq:bornesupsur}
\frac\kappa\e \sum_{T \in \mathbf T'_\e}  \LL^2(T) \to  \kappa \sin \theta_0 \HH^1(J_v).
\end{equation}
On the other hand, since every triangle $T \in \mathbf T_\e \setminus \mathbf T'_\e$ is contained in $\R^2 \setminus \bigcup_{i=1}^{N} L_i$, then
$$ \left\lvert \nabla {\bar v_\e}_{|T} \frac{x_i-x_j}{|x_i - x_j|} \right\rvert =  \frac{ \lvert\bar v(x_i) - \bar v(x_j)\rvert}{|x_i - x_j|}  \leq \norme{\nabla \bar v}_{L^\infty(\R^2 \setminus \bigcup_i L_i;\mathbb M^{2 \times 2})},$$ where $x_1$, $x_2$ and $x_3$ are the vertices of $T$. Hence, applying \cite[Remark 3.5]{CDM}, it results that 
\begin{equation}\label{eq:UniformBound_e(v_eps)}
\norme{ e (\bar v_\e) }_{L^\infty(\R^2 \setminus D_\e; \mathbb M^{2 \times 2}_{\rm sym})} \leq \frac{\sqrt{5}}{\sin \theta_0} \norme{ \nabla \bar v}_{L^\infty(\R^2 \setminus \bigcup_i L_i; \mathbb M^{2 \times 2})} =: K < \infty.
\end{equation}
Therefore, setting 
$$
\delta_\e := \sup_{0 < t < \e \beta K^2} \, \frac{f(t)}{t},
$$
we deduce, using $f(0)=0$ and the property \eqref{eq:A} of $\mathbf A$, that
$$\frac1\e f \left( \e \mathbf A e(\bar v_\e) : e(\bar v_\e) (1 - \chi_\e) \right) \leq  \delta_\e   \mathbf A e(\bar v_\e) : e(\bar v_\e) (1 - \chi_\e) \quad  \text{ in }\O' $$
From the properties \eqref{eq:f} of $f$, we infer that $\delta_\e \to 1$ as $\e\to 0$. Hence, using the third convergence in \eqref{eq:App_v_eps}, it ensures that 
\begin{eqnarray}\label{eq:bornesupvol}
\sum_{T \in \mathbf T_\e \setminus \mathbf T'_\e, \, T \cap \O' \neq \emptyset }  \LL^2(T \cap \O')  \frac1\e f \left( \e  \mathbf A {e (\bar v_\e)}_{|T} : {e(\bar v_\e)}_{|T} \right) & = & \int_{\O'}\frac1\e f \left( \e \mathbf A e(\bar v_\e) : e(\bar v_\e) (1 - \chi_\e) \right) \, dx \nonumber\\
& \leq & \delta_\e \int_{\O'}  \mathbf A e(\bar v_\e) : e(\bar v_\e) \mathds{1}_{\O' \setminus D_\e} \, dx \nonumber\\
& \to & \int_{\O'} \mathbf A e(\bar v) : e(\bar v) \, dx.
\end{eqnarray}
Gathering \eqref{eq:bornesupsur} and \eqref{eq:bornesupvol}, we obtain that 
\begin{equation}\label{eq:AlmostUpperBound}
\limsup_{\e \to 0^+} \int_{\O'} \frac1\e f \left( \e \mathbf A e(\bar v_\e) : e(\bar v_\e) \right) \, dx \leq  \int_{\O'} \mathbf A e (\bar v) : e( \bar v) \, dx +  \kappa \sin \theta_0 \HH^1(J_v).
\end{equation}
Besides, after decomposing the above integral over $\O' \setminus \bar \O$ and $\bar \O$, we can apply the lower bound estimate of Propostion \ref{lower bound} to the open bounded set with Lipschitz boundary $\O' \setminus \bar \O$ (for which $\mathbf T_\e$ is also an admissible triangulation, ${\bar v_\e}_{|\O' \setminus \bar \O} \in V_{\e}(\O' \setminus \bar \O)$ and $\bar v_\e$ converges in measure to $\bar v$ in $\O' \setminus \bar \O$), which leads to
\begin{multline*}
\limsup_{\e \to 0^+} \int_{\O'}\frac1\e f ( \e \mathbf A e(\bar v_\e) : e(\bar v_\e) ) \, dx \\
 \geq \limsup_{\e \to 0} \int_\O \frac1\e f ( \e \mathbf A e(\bar v_\e) : e(\bar v_\e) ) \, dx + \liminf_{\e \to 0} \int_{\O' \setminus \bar \O}\frac1\e f ( \e \mathbf A e(\bar v_\e) : e(\bar v_\e) ) \, dx \\
 \geq \F ''({\bar v}_{|\O}) +  \int_{\O' \setminus \bar \O} \mathbf A e(\bar v) : e(\bar v) \, dx +  \kappa \sin \theta_0 \HH^1 (J_{\bar v} \cap \O' \setminus \bar \O ).
\end{multline*}
Gathering \eqref{eq:AlmostUpperBound} and \eqref{eq:Sv}, as by construction $\bar v _{|\O} = v _{|\O}$, we deduce that
\begin{eqnarray*}
\F ''(v_{|\O}) 
&  \leq &  \int_{\O} \mathbf A e( v) :e( v) \, dx +  \kappa \sin \theta_0 \HH^1 ( J_v \setminus J_{\bar v}) + \kappa \sin \theta_0 \HH^1(J_v \cap  \bar \O  ) \\
& \leq  & \int_{\O} \mathbf A e( v) :e( v) \, dx +  \kappa \sin \theta_0  \HH^1 ( J_v  \setminus \O_\delta)  +  \kappa \sin \theta_0 \HH^1 ( J_v \cap  \bar \O  ),
\end{eqnarray*}
which settles \eqref{eq:BadUpperBound}. Passing to the limit as $\delta \searrow 0^+$ thanks to the monotone convergence Theorem, we obtain that $\HH^1 ( J_v \setminus \O_\delta ) \to \HH^1 ( J_v \setminus \O' )=0$, hence
$$ \F ''(v_{|\O}) \leq  \int_{\O} \mathbf A e( v) :e( v) \, dx  +  \kappa \sin \theta_0 \HH^1 ( J_v \cap  \bar \O  ),$$
which completes the proof of Lemma \ref{LemmaUpperBound}.
\end{proof}

\section{Convergence of minimizers}\label{sec:4}

In order to investigate the approximation of minimizers for the Griffith energy, it is natural to impose boundary conditions to avoid trivial minimizers such as rigid displacements. This section is devoted to an approximation of the Griffith functional under a Dirichlet boundary condition by means of brittle damage energies.
 
\subsection{Griffith energy with Dirichlet boundary condition}

In order to formulate a Dirichlet boundary condition, we need to consider a larger bounded Lipschitz open set $\O'$ such that $\overline \O \subset \O'$. Let $w \in W^{2,\infty}(\R^2; \R^2)$ be a prescribed boundary displacement. Given an admissible triangulation $\mathbf T_\e \in \mathcal T_\e(\O')$ of $\O'$, we define $w_{\mathbf T_\e}$ as the piecewise affine Lagrange interpolation of $w$ on $\mathbf T_\e$. Note that by standard finite element estimates (see \cite[Theorem 3.1.5]{Ciarlet}), 
\begin{equation}\label{eq:w_eps}
w_{\mathbf T_\e} \in V_\e(\O'), \quad
w_{\mathbf T_\e} \to w \text{ strongly in } H^1(\O';\R^2) \quad \text{and}\quad 
\sup_{\e > 0} \F_\e(w_{\mathbf T_\e}) < + \infty.
\end{equation}
We define $V^{\rm Dir}_\e(\O')$ to be the set of all continuous functions $u : \O' \to \R^2$ for which there exists a triangulation $\mathbf T_\e \in \mathcal T_\e(\O')$ so that $u$ is affine on each triangle $T \in \mathbf T_\e$ and $u=w_{\mathbf T_\e}$ on each triangle $T \in \mathbf T_\e$ such that $T \cap ( \O'\setminus \bar \O) \neq \emptyset$. We consider the following discrete functionals
$$\G_\e : u \in L^0(\O';\R^2) \mapsto \begin{cases} 
\ds  \frac{1}{\e} \int_\O  f \big(\e \mathbf A e(u):e(u) \big)  \, dx   & \text{ if } u \in V^{\rm Dir}_\e(\O'), \\
+ \infty & \text{ otherwise.}
\end{cases}$$
The Griffith energy with Dirichlet boundary condition $w$ is defined, for $u \in L^0(\O';\R^2)$, by 
$$\G (u) := 
\begin{cases}
\begin{array}{l}
\ds \int_\O \mathbf A e(u):e(u)\, dx \\
\quad + \kappa\sin\theta_0\left[\HH^1(J_u  \cap \O) + \HH^1(\partial\O \cap \{u \neq w \})\right]
\end{array} & \text{ if } 
\begin{cases}
u \in GSBD^2 (\O'),\\
u=w \; \LL^2\text{-a.e. in }\O' \setminus \overline \O,
\end{cases}\\
 + \infty & \text{ otherwise.}
\end{cases}
$$
Note that the additional boundary term accounts for possible jumps at the boundary, where the boundary condition fails to be satisfied. In the previous expression and in the sequel, we still denote by $u$ the trace of $u_{|\O} \in GSBD^2(\O)$ on $\partial \O$ (see \cite[Theorem 5.5]{DM2}).

\medskip

We will first prove the following result, generalizing Theorem \ref{thm:GSBD} to the case of Dirichlet boundary conditions.

\begin{thm}[{\bf $\Gamma$-convergence under Dirichlet boundary conditions}] \label{prop:Gamma_CV_Dirichlet}
The family $\{\G_\e\}_{\e>0}$ $\Gamma$-converges, with respect to the $L^0(\O';\R^2)$-topology, to the Griffith functional $\G$.
\end{thm}
Next, we will show a compactness result for sequences of dispacements $u_\e $ with uniformly bounded energy, with respect to the $L^0(\O';\R^2)$-topology of convergence in measure, under the simplifying assumption that $f : [0,+ \infty) \to [0,+\infty)$ reduces to 
$$f(t) = \kappa \wedge t$$
for $t \in \R^+$. Considering eventually a sequence of minimizers of $\G_\e$ (see Lemma \ref{rem:existence_Minimizers}), we will show that, up to a subsequence and up to subtracting a sequence of piecewise rigid body motions, it converges in measure in $\O'$ to a minimizer of $\G$ and the minimal value of $\G_\e$ converges to the minimal value of $\mathcal G$. In other words, we obtain the fundamental theorem of $\Gamma$-convergence in our specific context.

\begin{cor}[{\bf Convergence of minimizers}] \label{thm:CV_minimizers}
Assume further that $\O$ and $\O'$ are connected. For each $\e>0$ small, let $u_\e \in V_\e^{\rm Dir}(\O')$ be a minimizer of $\G_\e$. Then, there exist a subsequence (not relabeled), a sequence of piecewise rigid body  motions $\{r_\e\}_{\e>0}$ and a function $u \in GSBD^2(\O')$ with $u=w$ $\LL^2$-a.e. in $\O' \setminus \overline \O$, such that $u_\e - r_\e \to u$ in measure in $\O $,  $\G_\e (u_\e) \to \G(u)$ and $u$ is a minimizer of $\G$.
\end{cor}
 
\begin{rem}
Let us clarify that the improved lower bound \eqref{eq:CaccPart} in Proposition \ref{prop:Comp} is crucial only for the compactness and convergence of minimizing sequences, to ensure that after the removal of piecewise rigid body motions, minimizers of the approximating functionals converge to a minimizer of the Griffith functional and their energies converge as well. Instead, the $\Gamma$-convergence result under Dirichlet boundary conditions directly follows from Theorem \ref{thm:GSBD} and Proposition \ref{prop:comp}. Indeed, the proof of the lower bound in Theorem \ref{prop:Gamma_CV_Dirichlet} is the consequence of the lower bound in Theorem \ref{thm:GSBD} applied in $\O'$ together with the identification of the volume terms in $\O' \setminus \bar \O$. 
\end{rem}

\subsection{$\Gamma$-limit under Dirichlet boundary conditions}

Let us introduce the $\Gamma$-lower and upper limits $\G'$ and $\G''$ defined, for all $u \in L^0(\O';\R^2)$, by
$$\G'(u):=\inf\left\{\liminf_{\e \to 0}\G_\e(u_\e) : \; u_\e \to u \text{ in measure  in } \O' \right\},$$
and
$$\G''(u):=\inf\left\{\limsup_{\e \to 0}\G_\e(u_\e) : \; u_\e \to u \text{ in measure in } \O' \right\}.$$
\begin{proof}[Proof of Theorem \ref{prop:Gamma_CV_Dirichlet}]

\medskip

\textbf{Lower bound.} Let $u \in L^0(\O';\R^2)$. Without loss of generality, we can assume that $\G'(u) < + \infty$. For any $\zeta > 0$, there exists a sequence $\{u_\e\}_{\e >0}$ such that $u_\e \to u $ in measure in $\O'$ and 
$$ \underset{\e \to 0}{\liminf} \, \G_{\e}(u_\e) \leq \G' (u) + \zeta  < + \infty .$$
Let us extract a subsequence $\{u_k\} _{k \in \N} := \{u_{\e_k}\}_{k\in \N}$ such that $u_k \to u$ $\LL^2$-a.e. in $\O'$ and
$$\underset{k \to\infty}{\lim} \, \G_{\e_k} (u_k) = \underset{\e \to 0}{\liminf} \, \G_\e (u_\e) < + \infty .$$
This implies that, for $k$ large enough, $u_k \in V^{\rm Dir}_{\e_k} (\O')$ and $\sup_k \G_{\e_k} (u_k)< + \infty$. Especially, we get that $u_k \in V_{\e_k}(\O')$. Hence, according to Proposition \ref{prop:comp} and Theorem \ref{thm:GSBD} applied in $\O'$, we infer that 
$$u \in GSBD^2(\O')$$
and 
$$\liminf_{k \to \infty}  \int_{\O'}  \frac{1}{\e_k} f \big( \e_k \mathbf A e(u_k):e(u_k) \big)  \, dx \geq \int_{\O'} \mathbf A e(u):e(u) \, dx + \kappa \sin \theta_0 \HH^1(J_u).$$
Setting $w_k := w_{\mathbf T_{\e_k}}$ and using \eqref{eq:w_eps} together with the convergence in measure of $u_k = w_k$ to $u$ in $\O' \setminus \bar \O$, we get that $u = w $ $\LL^2$-a.e. in $\O' \setminus \bar \O$. 
Setting now $\delta_k := \sup \, \left\{ f(t)/t \, : \, 0 < t < \e_k \beta K^2 \right\} $ where
$$
K = \frac{\sqrt{5}}{\sin \theta_0} \norme{ \nabla w}_{L^\infty(\R^2;\R^2)} < \infty,
$$
one can check that $\delta_k  $ converges to $ 1$ as $k \to \infty$ according to \eqref{eq:f} and 
$$ \int_{\O'} \frac{1}{\e_k} f \big( \e_k \mathbf A e(u_k):e(u_k) \big)  \, dx \leq \G_{\e_k}(u_k) + \int_{\O' \setminus \bar \O}  \delta_k \, \mathbf A e(w_k) : e(w_k)   \, dx. $$
Hence, the Dominated Convergence Theorem ensures that 
$$
\zeta + \G'(u) +\int_{\O' \setminus \bar \O}  \mathbf A e(w) : e(w) \, dx  \geq \int_{\O'} \mathbf A e(u):e(u) \, dx + \kappa \sin \theta_0 \HH^1(J_u).
$$
Recalling that $J_u \cap \partial \O=\{u \neq w\} \cap \partial \O$ and $J_u \cap (\O' \setminus \overline\O)=J_w   \cap (\O' \setminus \overline\O)=\emptyset$, it entails that $\zeta+\G'(u) \geq \G(u)$, and the conclusion follows letting $\zeta \searrow 0$.

\medskip

\textbf{Upper bound.} Let $u \in L^0(\O';\R^2)$. We can assume that $\G(u)< + \infty$ so that $u \in GSBD^2(\O')$ and $u=w$ $\LL^2$-a.e. in $\O' \setminus \overline \O$. According to the density results for $GSBD$ functions (see \cite[Theorem 1.1]{CCdensity} and \cite[Formula (5.11)]{CCdensity}), there exists a sequence of functions $u_n \in SBV^2(\O;\R^2)\cap L^\infty(\O;\R^2)$ such that 
\begin{equation}\label{eq:density GSBD with Boundary Datum}
\begin{cases}
u_n \to u \text{ in measure in } \O, \\
u_n = w \text{ in an open bounded neighborhood of } \partial \O, \\
\limsup_n \G(u_n) \leq \G(u).
\end{cases}
\end{equation}
Extending (continuously) $u_n$ by $w$ on $\O' \setminus \overline \O$, we get that $u_n \to u$ in measure in $\O'$. The proof is thus complete once we know that $\G''(u_n) \leq \G(u_n)$, for all $n \in \N$, as it would imply $\mathcal G''(u) \leq \mathcal G(u)$, using the lower semicontinuity of $\G ''$ in $L^0(\O';\R^2)$ with respect to the convergence in measure together with the last point of \eqref{eq:density GSBD with Boundary Datum}.

\medskip

Therefore, we can assume without loss of generality that $u \in SBV^2(\O';\R^2) \cap L^\infty(\O'; \R^2)$ and $u = w$ on $V \cup (\O' \setminus \overline \O)$ with $V$ an open bounded neighborhood of $\partial \O$. Next according to Proposition \ref{prop:Cortesani-Toader Dirichlet} (see Appendix) there exist a sequence $\{u_k\}_{k \in \N}$ in $SBV^2(\O;\R^2) \cap L^\infty(\O;\R^2)$ as well as $N_k$ disjoint closed segments $L^k_1,\ldots,L_{N_k}^k  \subset \O$ satisfying :
$$\bar{J_{u_k}} = \bigcup_{i=1}^{ N_k} L^k_i  , \quad \HH^1( \bar{J_{u_k}} \setminus J_{u_k}) = 0, \quad u_k \in W^{2,\infty}(\O \setminus \bar{J_{u_k}};\R^2),$$
\begin{equation}\label{eq1548}
\begin{cases}
u_k = w \text{ in an open neighborhood of } \partial \O, \\
u_k \to u \quad \text{ strongly in } L^1(\O;\R^2), \\
\nabla u_k \to  \nabla u \quad \text{ strongly in } L^2(\O; \mathbb M^{2 \times 2}), \\
\limsup_{k\to \infty} \HH^1(J_{ u_k} ) \leq \HH^1(J_u).
\end{cases}
\end{equation}
Thus, extending continuously $ u_k$ to $\O'$ by setting $ u_k=w$ on $\O' \setminus \O$ and using again the lower semicontinuity of $\G ''$ in $L^0(\O';\R^2)$ with respect to the convergence in measure together with the convergences in \eqref{eq1548}, we are back to establishing that $\G''( u_k) \leq \G( u_k)$ as the conclusion follows letting $k \to +\infty$. Arguing almost word for word as in the proof of Lemma \ref{LemmaUpperBound}, we show the following Lemma \ref{lem:v=w near the boundary}, which leads to the desired inequality.  
\end{proof}

\begin{lem}\label{lem:v=w near the boundary}
Let $v \in SBV^2(\O';\R^2) \cap L^\infty(\O';\R^2)$ and $W$ be a bounded open neighborhood of $\partial\O$ be such that $ v = w$ on $W \cup (\O' \setminus \O)$ and
$$ \O \setminus W \supset \bar{J_v} =  \bigcup_{i=1}^{N} L_i  , \quad \HH^1( \bar{J_{v}} \setminus J_{v}) = 0, \quad v \in W^{2,\infty}(\O' \setminus \bar{J_{v}};\R^2),$$
for some pairwise disjoint closed segments $L_1,\ldots, L_N \subset \O \setminus W$. Then, 
$$ \G''(v) \leq  \int_\O \mathbf A e(v) : e(v) \, dx +  \kappa \sin\theta_0 \HH^1(J_v ) = \G(v).$$
\end{lem}

We do not detail the proof of Lemma \ref{lem:v=w near the boundary}. We only stress that, following Lemma \ref{LemmaUpperBound}, for $\e>0$ small enough, if $\mathbf T_\e$ is the admissible triangulation given by \cite[Appendix A]{CDM} and $v_\e$ is the Lagrange interpolation of the values of $v$ at the vertices of $\mathbf T_\e$, each triangle $T \in \mathbf T_\e$ such that $T \cap (\O' \setminus \bar \O) \neq \emptyset$ is contained in $W \cup \R^2\setminus \O$, so that $v_\e = w_{\mathbf T_\e}$ on $T$. In particular, it ensures that $ v_\e \in V^{\rm Dir}_{\e} (\O')$.

\subsection{Compactness for sequences with uniformly bounded energy and convergence of minimizers}

In this paragraph, the density $f$ reduces to $f(t) = \kappa \wedge t$ for $t \in \R$, so that the energy $\mathcal G_\e$ corresponds to 
$$
\mathcal G_\e(u)=\int_\O \frac{\kappa}{\e} \wedge \mathbf A e(u) : e(u) \, dx  \quad \text{ for } u \in V^{\rm Dir}_\e(\O').
$$
The reason of this simplifying assumption on $f$ comes from the difficulty to obtain compactness for sequences with uniformly bounded energies and from the difficulty to prove the existence of minimizers, as it will be detailed below. This is however a meaningful case since it corresponds to a brittle damage type energy from the mechanical point of view.`

\medskip

The following result shows a compactness and lower bound estimate for any sequence with uniformly bounded energy.

\begin{prop}\label{prop:Comp}
Let $\{\e_k\}_{k \in \N}$ satisfying $\e_k \to 0$ and let $\{ u_k \}_{k \in \N} \subset L^0(\O';\R^2)$ be such that $M := \sup_k \G_{\e_k}(u_k) < \infty$. Then there exist a subsequence (not relabeled), a Caccioppoli partition $\mathcal P = \big\{  P_j \big\}_{j\in \N}$ of $\O'$, a sequence of piecewise rigid motions $\{ r_k \}_{k\in \N}$ with
$$r_k := \sum_{j \in \N} r^j_k \, \mathds{1}_{P_j} ,$$
and a function $u \in GSBD^2(\O')$ such that $u=w$ $\LL^2$-a.e. in $\O' \setminus \overline \O$,
\begin{equation}\label{eq:diverging r_j-r_i}
\lvert r^i_k (x) - r^j_k(x) \rvert \to + \infty \quad \text{ for } \LL^2 \text{-a.e. } x \in \O', \quad \text{ for all } i \neq j,
\end{equation}
\begin{equation}\label{eq:CV measure}
u_k - r_k \to u \quad \text{in measure in } \O',
\end{equation}
and
\begin{equation}\label{eq:CaccPart}
\liminf_{k\to \infty}  \int_{\O} \frac{\kappa}{\e_k} \wedge \mathbf A e(u_k) : e(u_k)  \, dx  \geq \int_{\O} \mathbf A e(u):e(u)\, dx + \kappa\sin\theta_0\HH^1(J_u \cup \partial^*\mathcal P).
\end{equation}
\end{prop}

\begin{rem}
The lower bound inequality \eqref{eq:CaccPart} strongly relies on the simplifying assumption 
$$ f (t) = \kappa \wedge t \quad \text{ for } t \in \R.$$
Indeed, when working with a more general density $f :[0,+\infty) \to [0,+ \infty)$ satisfying \eqref{eq:f}, the main issue arises when one needs to fix some $\delta >0$ to use \eqref{eq:dependance in delta}, in order to exhibit an extraction, a Caccioppoli partition, rigid motions and a limit displacement which satisfy \eqref{eq:diverging r_j-r_i} and \eqref{eq:CV measure}. As all of them depend on $\delta >0$, it becomes difficult to derive the lower bound, even for the Lebesgue part \eqref{eq:Lebesgue AGAIN} below, since one simultaneously needs $\delta$ to be fixed (so that $\mathcal P$, $\{r_k \}_{k\in \N}$ and $u$ are well defined) and to converge to $0$ (in order to recover \eqref{eq:Lebesgue AGAIN} as in the proof of Proposition \ref{prop:leb}).
\end{rem}

\begin{proof}
By definition of $V^{\rm Dir}_{\e_k}(\O')$, there exists an admissible triangulation $\mathbf T_k \in \mathcal T_{\e_k} (\O')$ such that $u_k$ is affine on each triangle $T \in \mathbf T_k$ and $u_k = w_{\mathbf T_k}$ on each triangle $T \in \mathbf T_k$ intersecting  $\O' \setminus \bar \O$. We introduce the characteristic functions 
$$
\chi_k := \mathds{1}_{\{ \mathbf A e(u_k):e(u_k) \geq \frac{\kappa}{\e_k} \}} \in L^\infty( \O';\{ 0,1 \})
$$
which are constant on each triangle $T \in \mathbf T_k$. Since $u_k=w_{\mathbf T_k}$ on each triangle $T \in \mathbf T_k$ intersecting $\O' \setminus \bar \O$ and $w \in W^{2,\infty}(\R^2;\R^2)$, we deduce that, for $k$ large enough, $\chi_k=0$ in $\O' \setminus \bar \O$. Thus 
$$D_k := \{ \chi_k = 1 \} = \bigcup_{i=1}^{N_k}  T^k_i \subset \bar \O$$
for some triangles $T^k_i \in \mathbf T_k$, and $\LL^2(D_k) = \int_{\O} \chi_k \, dx \to 0$. 

\medskip

Let $v_k:= (1-\chi_k)u_k \in SBV^2(\O';\R^2)$ with $\nabla v_k= (1-\chi_k)\nabla u_k$ and $J_{v_k} \subset \bigcup_{i=1}^{N_k} \partial T^k_i  \subset \bar \O$. Arguing as in the proof of Proposition \ref{prop:comp}, we infer that
$$\sup_{k \in \N} \left\{ \int_{\O'} |e(v_k)|^2\, dx+ \HH^1(J_{v_k})\right\}<\infty.$$
In view of the $GSBD^2$-compactness Theorem (\cite[Theorem 1.1]{CCminimization}), there exist a subsequence (not relabeled), a Caccioppoli partition $ \mathcal P = \big\{  P_j \big\}_{j\in \N}$ of $\O'$, a sequence of piecewise rigid motions $\{ \tilde r_k \}_{k\in \N}$ with
$$\tilde r_k := \sum_{j \in \N} \tilde r^j_k \, \mathds{1}_{P_j} ,$$
and a function $\tilde u \in GSBD^2(\O')$ such that
$$\begin{cases}
\lvert \tilde r^i_k (x) - \tilde r^j_k(x) \rvert \to + \infty \quad \text{ for } \LL^2 \text{-a.e. } x \in \O', \text{ for all } i \neq j,\\
v_k - \tilde r_k \to \tilde u \quad \text{in measure in } \O',\\
e(v_k) \rightharpoonup e(\tilde u) \quad \text{weakly in } L^2(\O';\mathbb M^{2 \times 2}_{\rm sym}).
\end{cases}$$
Since $\LL^2(D_k) \to 0$, we deduce that $u_k - \tilde r_k \to \tilde u$ in measure in $\O'$. 

\medskip

For all $j \in \N$ such that $\LL^2(P_j \cap \O' \setminus \bar \O) >0$, the convergence in measure of $u_k - \tilde r^j_k$ to $\tilde u$ together with the convergence in measure of $u_k$ to $w$ in $P_j \cap \O' \setminus \bar \O =: V_j$ ensure that $\tilde r^j_k \to w - \tilde u$ in measure in $V_j$. Since the space of rigid body motions is a closed finite dimensional subspace of $L^0(\O';\R^2)$, we can find a rigid body motion $r^j$ such that $r^j_{|V_j}=w-\tilde u$ $\LL^2$-a.e. in $V_j$. Therefore, with 
$$ 
r:=\sum_{j \in \N, \, \LL^2( P_j \cap \O' \setminus \bar \O) >0} r^j \, \mathds{1}_{P_j} ,
$$
the piecewise rigid body motion $r_k := \tilde r_k-r $ and the function $u=\tilde u +r \in GSBD^2(\O')$ are such that
\begin{equation}\label{eq:0854}
\begin{cases}
 u_k - r_k  \to u \text{ in measure in } \O',\\
u = w \quad \LL^2 \text{-a.e. in } \O' \setminus \bar \O,\\
e(v_k) \rightharpoonup e(u) \quad \text{ weakly in } L^2(\O';\mathbb M^{2 \times 2}_{\rm sym}).
\end{cases}
\end{equation}

\medskip

We are now back to prove \eqref{eq:CaccPart}. As in the proof of Proposition \ref{lower bound}, we define the following Radon measures on $\O'$
$$ \lambda_k := \frac{\kappa}{\e_k} \wedge \mathbf A e(u_k):e(u_k)  \,  \LL^2 \res \O' \in \M(\O').$$ 
Using \eqref{eq:w_eps} and the energy bound assumption on $u_k$, we obtain that the sequence $\{\lambda_k\}_{k \in \N}$ is uniformly bounded in $\M(\O')$. Thus, up to a subsequence (not relabeled), we have $\lambda_k \overset{*}{\wto} \lambda$ weakly* in $\M(\O')$ for some nonnegative measure $\lambda \in \M(\O')$. Thanks to the lower semicontinuity of weak* convergence in $\M(\O')$ along open sets, we have that
\begin{equation}\label{eq:lambda AGAIN}
\liminf_{k\to \infty} \int_{\O'} \frac{\kappa}{\e_k} \wedge \mathbf A e(u_k) : e(u_k)  \, dx = \liminf_{k\to \infty} \lambda_k(\O') \geq \lambda(\O').
\end{equation}
Recalling that $ \mathcal{P}^{(1)} \cup \partial^* \mathcal P $ contains $\HH^1$-almost all of $\O'$, and using that the measures $ \LL^2 \res \O' $, $\HH^1 \res (\mathcal P^{(1)} \cap J_u)$ and $\HH^1 \res \partial^* \mathcal P $ are mutually singular, it is enough to show that
\begin{equation}\label{eq:Lebesgue AGAIN}
\frac{d\lambda}{d\LL^2 \res \O' } \geq \mathbf A e(u):e(u)\quad \LL^2\text{-a.e. in } \O'  ,
\end{equation}
\begin{equation}\label{eq:Hausdorff P1}
\frac{d\lambda}{d\HH^1\res  (\mathcal P^{(1)} \cap J_u)} \geq \kappa \sin\theta_0 \quad \HH^1\text{-a.e. in } \mathcal P^{(1)} \cap J_u,
\end{equation}
and
\begin{equation}\label{eq:Hausdorff essential boundary P}
\frac{d\lambda}{d\HH^1\res  \partial^* \mathcal P } \geq \kappa \sin\theta_0 \quad \HH^1\text{-a.e. in } \partial^* \mathcal P.
\end{equation}
Indeed, once \eqref{eq:Lebesgue AGAIN}, \eqref{eq:Hausdorff P1} and \eqref{eq:Hausdorff essential boundary P} are satisfied, it follows from the Radon-Nikod\'ym decomposition and the Besicovitch differentiation Theorems that
$$\lambda = \frac{d\lambda}{d\LL^2 \res \O'}  \LL^2 \res \O'  \, + \, \frac{d\lambda}{d\HH^1 \res (\mathcal P^{(1)} \cap J_u)}  \HH^1 \res (\mathcal P^{(1)} \cap J_u) \, + \, \frac{d\lambda}{d\HH^1 \res \partial^* \mathcal P}  \HH^1 \res \partial^* \mathcal P \, + \, \lambda^s,$$
for some nonnegative measure $\lambda^s$ which is singular with respect to  $\LL^2 \res \O' $, $\HH^1 \res (\mathcal P^{(1)} \cap J_u)$ and $\HH^1\res \partial^* \mathcal P$. Thus, after integration over $\O'$ and recalling \eqref{eq:lambda AGAIN}, we would get that
$$
\liminf_{k\to \infty} \int_{\O'} \frac{\kappa}{\e_k} \wedge \mathbf A e(u_k) : e(u_k)  \, dx \geq \int_{\O' } \mathbf A e(u):e(u)\, dx + \kappa \sin\theta_0 \HH^1 \big((J_u \cap \mathcal P^{(1)} ) \cup \partial^* \mathcal P   \big).
$$
On the one hand, the convergence in $H^1(\O';\R^2)$ of $w_{\mathbf T_k}$ to $w$ (see \eqref{eq:w_eps}) ensures that 
\begin{eqnarray}\label{eq:0939}
\liminf_{k\to \infty}  \int_{\O'} \frac{\kappa}{\e_k} \wedge \mathbf A e(u_k) : e(u_k)  \, dx & \leq & \limsup_{k\to \infty}  \int_{\O' \setminus \O} \mathbf A e(w_{\mathbf T_k}):e(w_{\mathbf T_k}) \, dx  + \liminf_{k\to \infty} \G_{\e_k}(u_k) \nonumber\\
& \leq & \int_{\O' \setminus \O} \mathbf A e(w):e(w) \, dx + \liminf_{k\to \infty} \G_{\e_k}(u_k).
\end{eqnarray}
On the other hand, using that $u=w$ in $\O' \setminus \overline \O$ and that $\mathcal P^{(1)} \cup \partial^* \mathcal P$ covers $\HH^1$ almost every $\O'$, we obtain that 
\begin{multline}\label{eq:0940}
\int_{\O' } \mathbf A e(u):e(u)\, dx + \kappa \sin\theta_0 \HH^1 \big((J_u \cap \mathcal P^{(1)} ) \cup \partial^* \mathcal P   \big)\\
= \int_{\O' \setminus \O} \mathbf A e(w):e(w) \, dx+\int_{\O} \mathbf A e(u):e(u)\, dx + \kappa \sin\theta_0 \HH^1 (J_u  \cup \partial^* \mathcal P).
\end{multline}
Gathering \eqref{eq:0939} and \eqref{eq:0940} leads to \eqref{eq:CaccPart}, which completes the proof of Proposition \ref{prop:Comp}.
\end{proof}

Using the last convergence in \eqref{eq:0854}, we easily get inequality \eqref{eq:Lebesgue AGAIN} arguing in an identical manner than in the proof of Proposition \ref{prop:leb}. We do not reproduce the argument. The rest of this section is devoted to the establishment of \eqref{eq:Hausdorff P1} and \eqref{eq:Hausdorff essential boundary P}. We start with the lower bound inequality for the jump part of the energy in the measure theoretic interior of $\mathcal P$.

\begin{prop}[{\bf Lower bound for the jump part in $ \mathcal P ^{(1)}$}]\label{prop Jump part in P1}
For $\HH^1$-a.e. $x_0 \in \mathcal P ^{(1)} \cap J_u$, 
$$\frac{d\lambda}{d\HH^1 \res (\mathcal P ^{(1)} \cap J_u)}(x_0) \geq \kappa \sin\theta_0.$$
\end{prop}

\begin{proof} 
The proof is very similar to that of Proposition \ref{prop Jump part}. We just sketch it, underlying the main differences.

Let $x_0 \in \mathcal P ^{(1)} \cap J_u$ be such that 
$$ \frac{d \lambda}{d \HH^1 \res (\mathcal P ^{(1)} \cap J_u)}(x_0)=\lim_{\varrho \searrow 0} \, \frac{\lambda\big( B_\varrho (x_0) \big)}{\HH^1\big( \mathcal P ^{(1)} \cap J_u \cap  B_\varrho (x_0) \big)}$$
exists and is finite, and
$$\lim_{\varrho \searrow 0} \, \frac{\HH^1(\mathcal P ^{(1)} \cap J_u \cap B_\varrho (x_0))}{2 \varrho}=1.$$
According to the Besicovitch differentiation Theorem and the  countably $(\HH^1,1)$-rectifiability of $\mathcal P ^{(1)} \cap J_u$, it follows that $\HH^1$-almost every point $x_0$ in $\mathcal P ^{(1)} \cap J_u$ fulfills these conditions.

\medskip

By definition of the jump set $J_u$, there exist $\nu :=\nu_u(x_0) \in \mathbb S^1$ and $u^\pm(x_0) \in \R^2$ with $u^+(x_0) \neq u^-(x_0)$ such that the function 
$$u_{x_0,\varrho} := u( x_0 + \varrho \, \cdot )$$
converges in measure in $B := B_1(0)$ to the jump function
$$\bar{u} : y \in B \mapsto 
\begin{cases}
u^+(x_0) & \text{ if } y \cdot \nu > 0, \\
u^-(x_0) & \text{ if } y \cdot \nu < 0,
\end{cases} $$
as $\varrho \searrow 0$. As before, we consider a sequence of radii $\{\varrho_j\}_{j \in \N}$ such that $\varrho_j \searrow 0$ and $\lambda(\partial B_{\varrho_j}(x_0))=0 = \HH^1(\mathcal P ^{(1)} \cap J_u \cap \partial B_{\varrho_j}(x_0))$ for all $j \in \N$. Arguing as in Proposition \ref{prop Jump part}, there exists an increasing sequence $\{k_j\}_{j \in \N}$ such that $k_j \nearrow  \infty$ as $j \to \infty$ and
$$ \begin{cases}
	\ds \big( u_{k_j} - r_{k_j} \big) (x_0 + \varrho_j \, \cdot) \to \bar{u} \quad\text{ in measure in } B, \\
	\ds  \frac{\lambda_{k_j}( B_{\varrho_j}(x_0)) }{2\varrho_j} \to \frac{d \lambda}{d \HH^1 \res (\mathcal P ^{(1)} \cap J_u)}(x_0),  \\
	\ds \e_{k_j} /\varrho_j \to 0, \quad \omega(\e_{k_j}) / \varrho_j \to 0. 
\end{cases}
$$
By definition of $\mathcal P ^{(1)}$, there exists $i_0 \in \N$ such that $x_0 \in {P_{i_0}}^{(1)}$. We thus infer that the function $v_j := \big( u_{k_j} - r^{i_0}_{k_j} \big) (x_0 + \varrho_j \, \cdot) \in H^1(B;\R^2)$ converges in measure to $\bar u$ in $B$. Indeed, for all $\eta > 0$, 
\begin{multline*}
\LL^2 \big( B \cap \{ \lvert v_j - \bar u \rvert > \eta \} \big) \\
 \leq  \LL^2 \big( B \cap  \{ \lvert \big( u_{k_j} -r_{k_j} \big) (x_0 + \varrho_j \, \cdot) - \bar u \rvert > \eta \} \big) + \LL^2 \left( B \setminus \left( \frac{ P_{i_0} - x_0}{\varrho_j} \right) \right)  \rightarrow  0,
 \end{multline*}
where we used that $x_0$ is a point of density $1$ for $P_{i_0}$. We are now back to an analogous situation than \eqref{eq:beforeSections}, since $v_j$ is continuous on $B$ and piecewise affine on each triangle $T \in \big( \mathbf T_{k_j} - x_0 \big) / \varrho_j $. Therefore, from here the conclusion of Proposition \ref{prop Jump part in P1} results from the proof of Proposition \ref{prop Jump part}.
\end{proof}

We next pass to the lower bound inequality for the energy on the reduced boundary of $\mathcal{P}$, which presents some non trivial  adaptations of the proof of Proposition \ref{prop Jump part}. 

\begin{prop}[{\bf Lower bound on the reduced boundary $\partial^* \mathcal{P}$}]\label{prop Jump part in essential bdry P}
For $\HH^1$-a.e. $x_0 \in \partial^* \mathcal{P}$, 
$$\frac{d\lambda}{d\HH^1 \res \partial^* \mathcal{P}}(x_0) \geq \kappa \sin \theta_0.$$
\end{prop}

\medskip

The rest of this subsection is devoted to prove Proposition \ref{prop Jump part in essential bdry P}, with essentially the same structure than the proof of  Proposition \ref{prop Jump part}. 

\medskip

{\bf Blow-up.} Let $x_0 \in \partial^* \mathcal P$ be such that
$$x_0 \in \partial^* P_{i_0} \cap \partial^* P_{j_0} \quad \text{ for some }i_0\neq j_0,$$
$$\nu:=\nu_{P_{i_0}}(x_0) = - \nu_{P_{j_0}}(x_0)  \quad \text{where } \nu_{P_{k}}(x_0) := \lim_{\varrho \searrow 0} \frac{ D \mathds{1}_{P_{k}} \big( B_\varrho(x_0) \big) }{\lvert D \mathds{1}_{P_{k}} \rvert \big( B_\varrho(x_0) \big)} \quad \text{for } k \in \{i_0,j_0\},$$ 
$$ \frac{d \lambda}{d \HH^1 \res \partial^* \mathcal{P}}(x_0)=\lim_{\varrho \searrow 0} \, \frac{\lambda\big( B_\varrho (x_0) \big)}{\HH^1\big(\partial^* \mathcal{P}  \cap  B_\varrho (x_0)}\big) \quad \text{ exists and is finite},$$
$$\lim_{\varrho \searrow 0} \, \frac{\HH^1(\partial^* \mathcal{P} \cap B_\varrho (x_0))}{2 \varrho}=1,$$
and there exist traces $u^\pm(x_0) \in \R^2$ such that the function 
$$u_{x_0,\varrho} := u( x_0 + \varrho \, \cdot )$$
converges in measure in $B := B_1(0)$ to
$$ y \in B \mapsto \bar{u}(y):=
\begin{cases}
u^+(x_0) & \text{ if } y \cdot \nu> 0, \\
u^-(x_0) & \text{ if } y \cdot \nu < 0,
\end{cases} \qquad \text{as }\varrho \searrow 0.$$
The previous properties turn out to be satisfied for $\HH^1$-a.e. $x_0 \in \partial^*\mathcal P$. This is a consequence of the  countably $(\HH^1,1)$-rectifiability of that set, the Besicovitch differentiation Theorem, the fact that $\mathcal P^{(1)} \cup \bigcup_{i \neq j} (\partial^* P_i \cap \partial^* P_j)$ covers $\HH^1$ almost all of $\O'$ (\cite[Theorem 4.17]{AFP}), and the existence of traces on $(\HH^1,1)$-rectifiable sets  (see \cite[Theorem 5.2]{DM2} in the case of $1$-dimensional $C^1$ submanifolds which may be extended to countably $(\HH^1,1)$-rectifiable sets arguing as in \cite[Proposition 4.1]{Bab}).

\medskip

To simplify notation, let us denote by $P^+:=P_{i_0}$ and $P^-:=P_{j_0}$. According to De Giorgi's Theorem (see \cite[Theorem 3.59]{AFP}) we infer that 
\begin{equation}\label{eq:DeGiorgi}
 \mathds{1}_{ \frac{P^\pm - x_0 }{ \varrho} } \to \mathds{1}_{ H^\pm}  \text{ strongly in } L^1(B) \text{ as } \varrho \searrow 0,
\end{equation}
where $H^\pm \subset \R^2$ denote the halfspaces orthogonal to $\nu$ and containing $\pm \nu$. With these notation, we have that 
$$ \bar u = u^+(x_0) \mathds{1}_{H^+ \cap B} + u^-(x_0) \mathds{1}_{H^- \cap B}.$$
Note also that contrary to Proposition  \ref{prop Jump part} where jump points were considered, it might be the case that $u^+(x_0)=u^-(x_0)$, i.e. that $\bar u$ is constant.

\medskip

{\bf Extraction of diagonal subsequences.}  As before, we consider a sequence of radii $\{\varrho_j\}_{j \in \N}$ such that $\varrho_j \searrow 0$ and $\lambda(\partial B_{\varrho_j}(x_0))=0  = \HH^1\big(\partial^* \mathcal{P}  \cap \partial B_{\varrho_j} (x_0) \big)$ for all $j \in \N$. By our choice of $x_0$, \eqref{eq:diverging r_j-r_i} and \eqref{eq:CV measure}, with $r^\pm_k := {r_k}_{|P^\pm}$, we have :   
$$  \begin{cases}
\ds \lim_{j \to \infty}\lim_{k \to \infty} \big( u_k - r_k \big) (x_0 + \varrho_j \, \cdot)  =\lim_{j \to \infty} u_{x_0, \varrho_j} =   \bar u \quad \text{ in measure in } B, \\
\ds  \lim_{j \to \infty}\lim_{k \to \infty} \arctan \lvert r^+_k - r^-_k \rvert (x_0 + \varrho_j \, \cdot) =  \frac\pi2 \quad \text{ in measure in } B, \\
\ds\lim_{j \to \infty}\lim_{k \to \infty}  \frac{\lambda_k (B_{\varrho_j}(x_0))}{2\varrho_j} =\lim_{j \to \infty}\frac{ \lambda (B_{\varrho_j}(x_0))}{2\varrho_j}  = \frac{d \lambda}{d \HH^1 \res \partial^* \mathcal{P}}(x_0), \\
\ds\lim_{j \to \infty}\lim_{k \to \infty} \frac{\e_k}{\varrho_j} = \lim_{j \to \infty}\lim_{k \to \infty} \frac{\omega(\e_k)}{\varrho_j}= 0. 
\end{cases}   $$
We can thus find an increasing sequence $\{ k_j \}_{j \in \N}$ such that $k_j \nearrow \infty$ as $j \to \infty$ and 
\begin{subequations}\label{eq:beforeSections partial P PART 1}
	\begin{empheq}[left=\empheqlbrace]{align}
	& \big( u_{k_j} - r_{k_j} \big) (x_0 + \varrho_j \, \cdot ) \to \bar u \quad\text{ in measure in } B, \\
	& \arctan \lvert r^+_{k_j} - r^-_{k_j}   \rvert ( x_0 + \varrho_j \, \cdot ) \to \frac\pi2 \quad \text{ in measure in } B,  \\
	& \frac{ \lambda_{k_j}( B_{\varrho_j}(x_0))}{2\varrho_j} \to \frac{d \lambda}{d \HH^1 \res \partial^* \mathcal{P}}(x_0),  \\
	& \frac{\e_{k_j}}{\varrho_j} \to 0, \quad \frac{\omega(\e_{k_j})}{\varrho_j}\to 0. \label{seq:eps' AGAIN}	
	\end{empheq}
\end{subequations}
Let  $ v_j :=  u_{k_j}(x_0 + \varrho_j \, \cdot)$, $r_j^\pm:=r^\pm_{k_j}(x_0+\varrho_j \cdot)$ and $r_j := r^+_j \mathds{1}_{H^+ \cap B} + r^-_j \mathds{1}_{H^- \cap B}$. By \eqref{eq:DeGiorgi} and \eqref{eq:beforeSections partial P PART 1}, we have for all $\eta>0$,
\begin{multline*}
 \LL^2 \big( \{ \lvert v_j - r^\pm_j - u^\pm(x_0) \rvert > \eta \} \cap B \cap H^\pm \big) \\
\leq \LL^2 \left( B \cap H^\pm \setminus ( P^\pm - x_0) / \varrho_j  \right)  + \LL^2 \left( \{ \lvert v_j - r_{k_j}(x_0 + \varrho_j \, \cdot) - \bar u \rvert > \eta \} \cap B  \right)    \to 0.
\end{multline*}
Thus, up to a subsequence
\begin{equation}\label{eq:beforeSections partial P CV measure}
\begin{cases}
v_j - r^\pm_j \to  u^\pm(x_0) \quad \text{ in measure in } B \cap H^\pm,\\
\lvert r^+_j - r^-_j \rvert \to + \infty \quad \LL^2 \text{-a.e. in } B.
\end{cases}
\end{equation}

\medskip

{\bf Selection of a slicing direction.} According to \cite[Lemma 2.8]{CCminimization}, there exist an $\HH^1$-negligible set $N \subset B^\nu$ and a countable dense subset $D$ of $\mathbb{S}^1$ such that for all $\xi \in D$ and all $ y \in B^\nu \setminus N$,
$$\lvert ( r^+_j - r^-_j )(y) \cdot \xi \rvert \to + \infty$$
as $j \to + \infty$.  Note that $\nabla r^\pm_j \xi \cdot \xi = 0$, so that the quantity $t \mapsto ( r^+_j - r^-_j )^\xi_y(t)  =( r^+_j - r^-_j )(y) \cdot \xi$ is independent of $t \in B^\xi_y$. Thus, for all $y \in B^\nu \setminus N$, we have
\begin{equation}\label{eq:diverging section r_j - r_i}
\arctan\lvert ( r^+_j - r^-_j )^\xi_y \rvert \to \frac{\pi}{2} \quad \text{uniformly in } t \in B^\xi_y.
\end{equation}
For any $\eta>0$, let $\xi \in \mathbb{S}^1 \cap D$ be such that 
\begin{equation}\label{eq:xi AGAIN}
\left\lvert \nu - \xi \right\rvert \leq \eta, \quad \nu \cdot \xi \geq \frac12, \quad \left\lvert \nu \cdot \xi^\perp \right\rvert \leq \eta.
\end{equation}

As in the proof of Proposition \ref{prop Jump part}, using a change of variables and the ellipticity property \eqref{eq:A} of $\mathbf A$ and \eqref{seq:eps' AGAIN}, we get
$$
2\frac{d \lambda}{d \HH^1 \res \partial^* \mathcal{P}}(x_0) \geq  \limsup_{j \to \infty} \varrho_j \int_{B} \frac{\kappa}{\e_{k_j}} \wedge \frac{\alpha}{\varrho_j^2} |e(v_j)\xi\cdot\xi|^2 \, dy
$$
so that, introducing the following characteristic functions,
$$  \chi_j := \mathds{1}_{  \left\{ \frac{\alpha }{\varrho_j^2}|e(v_j)\xi \cdot\xi|^2 \geq \frac{ \kappa}{\e_{ k_j}} \right\}} \in L^\infty(B; \lbrace0,1 \rbrace),$$ 
we obtain that
\begin{equation}\label{eq:lambda_k greater than AGAIN}
2\frac{d \lambda}{d \HH^1 \res \partial^* \mathcal P}(x_0) \geq \limsup_{j \to \infty}\left\{\frac{\alpha }{\varrho_j} \int_{B}  (1-\chi_j) |e(v_j)\xi\cdot\xi|^2 \, dy + \frac{\kappa\varrho_j}{\e_{k_j}} \int_{B} \chi_j \, dy\right\}.
\end{equation}
We define the translated and rescaled triangulations:
$$ \mathbf{T}^{x_0,j} := \frac{1}{\varrho_j} \left( \mathbf{T}^{k_j} - x_0 \right), \quad \mathbf{T}^{x_0,j}_b := \left\{ T \in \mathbf{T}^{x_0,j}: \, \frac{\alpha }{\varrho_j} |e(v_j)_{|T} \xi \cdot \xi |^2 \geq\frac{  \kappa\varrho_j}{\e_{k_j}} \right\},$$
and the family of triangles which intersect $\bar{B_{1-\frac\eta4}}$:
$$\mathbf{T}^{x_0,j}_{b,int} := \left\{ T \in \mathbf{T}^{x_0,j}_b:\;  T \cap \bar{B_{1-\frac\eta4}} \neq \emptyset \right\}.$$
Note that $v_j - r^\pm_j$ is affine and $\chi_j$ is constant on each $T \in \mathbf{T}^{x_0,j}$, and that \eqref{seq:eps' AGAIN} ensures that for $j \in \N$ large enough (depending on $\eta$), each $T \in \mathbf{T}^{x_0,j}_{b,int}$ is contained in $B$. As in \eqref{eq:AB(Y)}, for all $y \in (B_{1-\frac\eta4})^\nu$, we denote by $a(y)$ and $b(y)$ the end points of the section passing through $y$ in the direction $\xi$ (see the Figure \ref{fig:A}) in such a way that $(B_{1- \frac\eta4})^\xi_y = (a(y),b(y) )$. We also recall that $L_\eta$ is defined as in \eqref{eq:L(eta)} and satisfies $ 0 < L_\eta  \leq |a(y)|,\, |b(y)| \leq 2$.

\medskip

Using \eqref{eq:beforeSections partial P CV measure}, \eqref{eq:diverging section r_j - r_i}, \eqref{eq:lambda_k greater than AGAIN}, Fubini's and Egoroff's Theorem, we can adapt the proof of Lemma \ref{lemma afterSections} to show the following result.

\begin{lem}\label{lemma afterSections AGAIN}
For all $\eta > 0$, there exist a subset $Z \subset B^\nu$  containing $N$ with $\HH^1(Z) \leq \eta$, and a subsequence (not relabeled) such that the following property holds : for all $\gamma >0$, there exists $j_0 = j_0(\gamma) \in \N$ such that for all $y \in B^\nu \setminus Z$ and all $j \geq j_0$, 
$$
\int_{B^\xi_y}   (\chi_j) ^\xi_y  \, dt \leq \gamma , \quad \int_{B^\xi_y}  \big( 1 - ( \chi_j ) ^\xi_y  \big)  |( (v_j)^\xi_y ) ' |^2  \, dt \leq \gamma^2,
$$
and 
$$ \int_{B^\xi_y \cap \R^\pm} 1\wedge | (v_j - r^\pm_j)^\xi_y - u^\pm(x_0) \cdot \xi| \, dt \leq \frac{\gamma}{2}, \quad |(r_j^+-r_j^-)_y^\xi | \geq  |[u](x_0)\cdot\xi|+1.$$
\end{lem}

As in the proof of Proposition \ref{prop Jump part}, we next show that, for some subset $Z' \subset (B_{1-\frac\eta2})^\nu$ of arbitrarily small $\HH^1$ measure, and along a subsequence (only depending on $\eta$), all the sections in the direction $\xi$ passing through $(B_{1-\frac\eta2})^\nu\setminus Z'$ must cross at least one triangle $T \in \mathbf{T}^{x_0,j}_{b,int}$ contained in $B$.

\begin{lem}\label{lemma1 AGAIN}
For all $\eta>0$, there exist a subset $ Z' \subset B^\nu$ containing $Z$ with $\HH^1(Z') \leq \eta$, and a subsequence (not relabeled) such that the following property holds : for all $y \in (B_{1-\frac\eta2})^\nu \setminus Z'$ and all $j \in \N$, there exists a triangle $T=T(y,j) \in \mathbf{T}^{x_0,j}_{b,int}$ such that $( \mathring{T} \cap B )^\xi_y \neq \emptyset$.
\end{lem}

\begin{proof}
Let $Z$ be the exceptional set given by Lemma \ref{lemma afterSections AGAIN}. We first show that there exists an increasing mapping $\phi : \N \to \N$ such that : for all $y \in (B_{1-\frac\eta2})^\nu \setminus Z$ and all $j \in \N$, there exists a triangle $T=T(y,\phi(j)) \in \mathbf{T}^{x_0,\phi(j)}_{b,int}$ such that $(T \cap B )^\xi_y \neq \emptyset$. Suppose by contradiction that such is not the case, and define
$$\gamma^*_1 := L_\eta  > 0, \quad \gamma^*_2 := \frac{ L_\eta }{1 + 2 L_\eta} > 0 \quad\text{and}\quad\gamma^* = \gamma^*(\eta) := \frac{ \gamma^*_1 \wedge \gamma^*_2}{4} > 0.$$
Thanks to Lemma \ref{lemma afterSections AGAIN}, there exists $j^* = j^*(\gamma^*) \in \N$ such that for all $y \in B^\nu \setminus Z$ and all $j \geq j^*$, $\left\lvert (r^+_j - r^-_j)^\xi_y \right\rvert \geq |[u](x_0)\cdot\xi|+1 $ and
$$ 
\ds \int_{B^\xi_y}  ( 1 - ( \chi_j ) ^\xi_y  )  |( (v_j)^\xi_y ) ' |^2  \, dt \leq {\gamma^*}^2, \quad \int_{B^\xi_y \cap \R^\pm}  1\wedge | (v_j - r^\pm_j )^\xi_y  - u^\pm(x_0)\cdot \xi| \, dt \leq \frac{\gamma^*}{2}. 	
$$
As in Lemma \ref{lemma1}, we consider the extraction $ \phi : j \in \N \longmapsto j + j^* \in \N$. By assumption, there exist $y \in (B_{1-\frac\eta2})^\nu\setminus Z$ and $ j \in \N$ such that $\left( T \cap B \right)^\xi_y = \emptyset$ for all $T \in \mathbf{T}^{x_0,j + j^*}_{b,int}$, implying that $(\chi_{j+j^*})_y^\xi \equiv 0$ on $(a(y),b(y))$, $\left\lvert (r^+_{j+j^*} - r^-_{j+j^*})^\xi_y \right\rvert \geq  |[u](x_0)\cdot\xi|+1$ and
$$ 
\ds  \int_{a(y)}^{b(y)}  |( (v_{j + j^*})^\xi_y ) '|^2  \, dt \leq {\gamma^*}^2,  \quad  \int_{[a(y),b(y)] \cap \R^\pm} 1 \wedge | (v_{j+j^*} - r^\pm_{j+j^*} )^\xi_y  - u^\pm(x_0)\cdot \xi| \, dt \leq \frac{\gamma^*}{2}, 	
$$
since $\phi(j) = j + j^* \geq j^*$. By continuity of $( v_{j+j^*} - r^\pm_{j+j^*} )^\xi_y$ on the compact sets $[a(y),b(y)] \cap \R^\pm$, there exist two points 
$$t^\pm \in  \underset{ [a(y),b(y)] \cap \R^\pm}{\rm arg \min} \left( 1 \wedge | \left( v_{j+j^*} - r^\pm_{j+j^*}\right)^\xi_y - u^\pm (x_0) \cdot \xi | \right) .$$
Hence, recalling \eqref{eq:L(eta)}
\begin{eqnarray*}
\frac{\gamma^*}{L_\eta}	& \geq &    1 \wedge | \left( v_{j+j^*} - r^-_{j+j^*} \right)^\xi_y (t^-) - u^- (x_0) \cdot \xi | +  1 \wedge | \left( v_{j+j^*} - r^+_{j+j^*} \right)^\xi_y (t^+) - u^+ (x_0) \cdot \xi | \\
			& \geq & 1 \wedge \left( \left\lvert [u](x_0)\cdot \xi + (r^+_{j+j^*} - r^-_{j+j^*})^\xi_y  \right\rvert \, - \, \left\lvert \int_{t^-}^{t^+} \left( (v_{j + j^*})^\xi_y \right) ' (t) \, dt  \right\rvert \right) \geq   1 - 2 \gamma^* ,
\end{eqnarray*}
which is impossible thanks to of our choice of $\gamma^*$. We conclude the proof of Lemma \ref{lemma1 AGAIN} in the same way as for Lemma \ref{lemma1}.
\end{proof}
As a consequence of Lemma \ref{lemma1 AGAIN}, introducing the family of triangles 
$$\mathscr{F}_j := \left\{ T \in \mathbf{T}^{x_0, j}_{b,int}: \; \text{there exists } y \in (B_{1-\frac\eta2})^\nu \text{ such that } ( \mathring{T} \cap B )^\xi_y \neq \emptyset \right\}$$
for all $j \in \N$, it is possible to obtain a too low lower bound, roughly speaking because Lemma \ref{lemma1 AGAIN} does not exhibit enough triangles in $\mathbf{T}^{x_0,j}_{b,int}$, as explained after \eqref{eq:Fj}. Therefore, we need to establish that many lines $B_y^\xi$ parallel to $\xi$ and passing through $B^\nu$ must actually intersect at least two triangles of the collection $\mathbf{T}^{x_0,j}_{b,int}$. To this aim, we show that the set of points $y \in B^\nu$ such that $B_y^\xi$ intersects exactly one triangle $T$ in the collection $\mathbf{T}^{x_0,j}_{b,int}$, has arbitrarily small $\HH^1$ measure. 

\begin{lem}\label{lemma3 AGAIN}
For all $\eta>0$, there exist constants $C_* = C_*(\eta) > 0$, $\gamma_* = \gamma_*(\eta) > 0$ and a subset $Z_* = Z_*(\eta) \subset B^\nu$ containing $Z'$ and satisfying $\HH^1(Z_*) \leq 3 \eta$ such that the following property holds: for all $0<\gamma<\gamma_*$, there exists $ j(\gamma) \in \N $ such that for all $ j \geq j(\gamma)$, the set 
$$Y_j := \left\{ y \in (B_{1-\frac\eta2})^\nu\setminus Z': \; \text{there exists a unique } T \in \mathbf{T}^{x_0, j}_{b,int} \text{ such that } ( \mathring{T} \cap B )^\xi_y \neq \emptyset \right\}$$
satisfies 
$$ \HH^1(Y_j \setminus Z_*) \leq C_* \gamma. $$
\end{lem}

\begin{proof}[Proof of Lemma \ref{lemma3 AGAIN}] We follow the same three steps structuring the proof of Lemma \ref{lemma3}.

\medskip

{\bf Step 1.} We start by showing that for $j$ large enough and for many points $y \in Y_j$, the only triangle $T$ in $\mathbf{T}^{x_0,j}_{b,int}$ crossing $B_y^\xi$ is getting closer to the diameter $B^\nu$. 

\medskip

For all $j \in \N$ and all $y \in Y_j$, let $T_j(y)\in \mathbf{T}^{x_0, j}_{b,int}$ be the unique triangle such that $( \mathring{T}_j(y) \cap B)^\xi_y \neq \emptyset$. We keep the notation \eqref{eq:aj and bj} for the end points $a_j(y)$ and $ b_j(y)$ of the section in the direction $\xi$ passing through $y$ inside $T_j(y)$ (see the Figure \ref{fig:C}). Let us show that 
$$ \ds f_j(y) := \left( |a_j(y)| + |b_j(y)| \right) \mathds{1} _{Y_j}(y)\to  0 \quad \text{ for all }y \in (B_{1-\frac\eta2})^\nu \setminus Z'.$$

Let $y \in (B_{1-\frac\eta2})^\nu \setminus Z'$. Assume by contradiction that $\ell:= \limsup_j f_j(y)>0$ and extract a subsequence (depending on $y$, not relabeled) such that $f_j(y) \to \ell$. Then, there exists $j_0 \in \N$ such that $y \in Y_j$ for all $j \geq j_0$. Moreover, according to Lemma \ref{lemma afterSections AGAIN} and setting $I_j(y) := \left( a(y),b(y) \right) \setminus \left( a_j(y),b_j(y) \right) \subset B^\xi_y$, we infer that
\begin{equation}\label{eq:convwj AGAIN}
\ds \lvert b_j(y) - a_j(y) \rvert  \, + \,  \int_{I_j(y)} \lvert ( (v_j)^\xi_y )'  \rvert^2 \, dt \, + \, \int_{[a(y),b(y)] \cap \R^\pm} 1 \wedge | (v_j - r^\pm_j)^\xi_y - u^\pm(x_0) \cdot \xi | \, dt \to 0. 
\end{equation}
Up to another subsequence (still not relabeled), it ensures that $a_j(y) , b_j(y) \to m$ for some $m \in [a(y),b(y)]$. Thus, for all $\tau >0$, $ I_\tau := \left( a(y), m - \tau \right) \cup \left( m + \tau, b(y) \right) \subset I_j(y)$ for $j \in \N$ sufficiently large. In particular, we deduce that $m-\tau \leq 0$. Indeed, assuming that $m-\tau >0$, by continuity of $(v_j - r^+_j)^\xi_y$ on $\big( 0, m-\tau \big)$ and of $(v_j - r^-_j)^\xi_y$ on $\big( a(y) , 0 \big)$, there exist 
$$ t^+_j \in \underset{( 0, m-\tau )}{ \rm arg \min} \, 1 \wedge \lvert (v_j - r^+_j)^\xi_y - u^+(x_0) \cdot \xi \rvert \quad \text{and} \quad t^-_j \in \underset{(a(y),0 )}{ \rm arg \min} \, 1 \wedge \lvert (v_j - r^-_j)^\xi_y - u^-(x_0) \cdot \xi \rvert,$$
which satisfy 
\begin{multline*}
1 \wedge \lvert (v_j - r^+_j)^\xi_y(t^+_j) - u^+(x_0) \cdot \xi \rvert + 1 \wedge \lvert (v_j - r^-_j)^\xi_y(t^-_j) - u^-(x_0) \cdot \xi \rvert \\
\geq 1 \wedge \left( \left\lvert [u](x_0) \cdot \xi + (r^+_j - r^-_j)^\xi_y \right\rvert - 2 \sqrt{ \int_{I_j(y)} \lvert ( (v_j)^\xi_y )'  \rvert^2 \, dt}   \right) \to 1
\end{multline*}
according to \eqref{eq:convwj AGAIN} and \eqref{eq:diverging section r_j - r_i}. However, 
\begin{multline*}
1 \wedge \lvert (v_j - r^+_j)^\xi_y(t^+_j) - u^+(x_0) \cdot \xi \rvert + 1 \wedge \lvert (v_j - r^-_j)^\xi_y(t^-_j) - u^-(x_0) \cdot \xi \rvert \\
 \leq \frac{1}{L_\eta} \int_{a(y)}^{0} 1 \wedge | (v_j - r^-_j)^\xi_y - u^-(x_0) \cdot \xi | \, dt + \frac{1}{m-\tau} \int_{0}^{m-\tau} 1 \wedge | (v_j - r^+_j)^\xi_y - u^+(x_0) \cdot \xi | \, dt \to 0
\end{multline*}
according again to \eqref{eq:convwj AGAIN}, which leads to a contradiction. We similarly show that $m+ \tau \geq 0$, leading to $|m| \leq \tau$. Taking the limit as $\tau \to 0^+$, we obtain that $m=0$ which is against $\ell>0$.

\medskip

Therefore, owing to Lemma \ref{lemma afterSections AGAIN} and Egoroff's Theorem, we can find a set $Z_*^1 \subset B^\nu$ containing $Z'$ with  $\HH^1(Z_*^1) \leq 2 \eta$ such that for all $\gamma > 0$, there exists $j_0(\gamma) \in \N$ satisfying 
\begin{equation}\label{eq:betterAfterSections AGAIN}
\begin{cases}
\ds \int_{B^\xi_y}  ( 1 - ( \chi_j ) ^\xi_y  )  |( (v_j)^\xi_y ) ' |^2  \, dt \leq \gamma^2, \\
\ds \int_{B^\xi_y \cap \R^\pm} 1 \wedge | (v_j - r^\pm_j)^\xi_y - u^\pm(x_0) \cdot \xi | \, dt \leq \frac{\gamma}{2}, \\
\ds ( |a_j(y)| + |b_j(y)| ) \, \mathds{1}_{Y_j}(y) \leq \gamma
\end{cases}
\end{equation}
for all $y \in (B_{1-\frac\eta2})^\nu \setminus Z_*^1$ and all $j \geq j_0(\gamma)$.
\medskip

\textbf{Step 2.} Arguing in the same manner as for \eqref{eq:C*}, one can show that for many points $y \in Y_j$, the variation of $(v_j - r_j)_y^\xi$ inside the only triangle $T$ in $\mathbf{T}^{x_0,j}_{b,int}$ which is crossed by $B_y^\xi$, is close to that of $\bar u_y^\xi$. Precisely, setting the constants
$$C_\eta := 8 \left( 1  + \frac{1}{L_\eta} \right) > 0, \quad   \gamma_*=\gamma_*(\eta) := \frac12 \min \left(  1,\frac{1}{C_\eta}, L_\eta \right) > 0,$$
we get that for all $0 < \gamma < \gamma_*$, there exists $j_1(\gamma) \in \N$ such that for all $ j \geq j_1(\gamma) $ and all $y \in Y_j \setminus Z_*^1$,
\begin{equation}\label{eq:C* AGAIN}
\left\lvert  (v_j )^\xi_y(b_j(y)) - (v_j)_y^\xi(a_j(y)) - (r^+_j - r^-_j)^\xi_y - [u](x_0) \cdot \xi \right\rvert \leq C_\eta\gamma . 
\end{equation}

\medskip

\textbf{Step 3.} We now show that, after enlarging slightly the set $Z_*^1$ into a set $Z_* \subset B^\nu$ with $\HH^1(Z_*) \leq 3 \eta$, it is possible to include $Y_j \setminus Z_*$ inside a finite union of arbitrarily small segments contained in $B^\nu$ (see Figure \ref{fig:F}).

\medskip

By definition of rigid body motions, there exist skew symmetric matrices $M_j \in \mathbb M^{2 \times 2}_{\rm skew}$ and vectors $m_j \in \R^2$ such that 
$$ r^+_j - r^-_j = M_j \, {\rm id}_{\R^2} + m_j$$
for all $j \in \N$. Recalling \eqref{eq:diverging section r_j - r_i}, we get that for all $y \in B^\nu \setminus Z_*^1 \subset B^\nu \setminus N$,
$$ \left\lvert (r^+_j - r^-_j)^\xi_y \right\rvert = \left\lvert \big(M^T_j \xi \big) \cdot y + m_j \cdot \xi \right\rvert \to + \infty$$
as $j \to + \infty$. In particular, setting $\alpha_j := M^T_j \xi \in \R^2$ and $\beta_j := m_j \cdot \xi \in \R$, we get that $\mu_j := \lvert \alpha_j \rvert + \lvert \beta_j \rvert \to + \infty$ as $j \to +\infty$. Hence, up to a subsequence (depending only on $\xi$, not relabeled), there exist $\alpha \in \R^2$ and $\beta \in \R$ such that
$$ \frac{\alpha_j}{\mu_j} \to \alpha, \quad \frac{\beta_j}{\mu_j} \to \beta, \quad \text{and} \quad \lvert \alpha \rvert + \lvert \beta \rvert = 1.$$
In particular,
\begin{equation}\label{eq:CVU}
\frac{1}{\mu_j} \big( \alpha_j \cdot \rm id_{\R^2} + \beta_j \big) \to \alpha \cdot \rm id_{\R^2} + \beta \text{ uniformly on $B$}.
\end{equation}
Notice that the affine line $\Delta := \{ z \in \R^2 : \, \alpha \cdot z + \beta = 0 \}$ cannot coincide with $\Pi_\nu$.  Indeed, if such would be the case, it would entail that $\beta = 0$ and $\alpha = \pm \nu \in \mathbb{S}^1$. Yet, $M_j \in \mathbb M^{2 \times 2}_{\rm skew}$ being skew symmetric, we would obtain that 
$$0=\frac{M_j \xi \cdot \xi}{\mu_j}=\frac{\alpha_j}{\mu_j} \cdot \xi\to \alpha \cdot \xi = \pm \nu \cdot \xi$$
which is against our choice \eqref{eq:xi AGAIN} of $\xi \in \mathbb S^1 \cap D$. As a consequence $\Delta$ intersects $B^\nu$ in at most one point $z_*$.

\medskip

If $\Delta \cap \Pi_\nu = \{ z_* \}$, we define $Z_*^2 := B^\nu \cap B_{\frac{\eta}{2}}(z_*) $ while if $\Delta \cap \Pi_\nu = \emptyset$, we define $Z_*^2 = \emptyset$. The continuity of $\alpha \cdot \rm id_{\R^2} + \beta$ on the compact set $B^\nu \setminus Z_*^2 $ entails that 
$$ 0 < m_{\alpha,\beta}(\eta) := \min \{ \left\lvert \alpha \cdot y + \beta \right\rvert : \, y \in B^\nu \setminus Z_*^2 \}.$$
Set $Z_* := Z_*^1 \cup Z_*^2 \subset B^\nu$, which satisfies $\HH^1(Z_*) \leq 3 \eta$, and for all $j \in \N$, we define
$$\widehat{\mathbf T}_j:=\{T \in \mathbf{T}^{x_0, j}_{b,int} \, : \, \text{ there exists $y \in Y_j \setminus Z_*$ such that $(\mathring T \cap B)_y^\xi\neq \emptyset$}\}.
$$
Thus, for all $j \in \N$ and for each triangle $T \in \widehat{\mathbf T}_j$, there exists a point $y_T \in \Phi \circ p_\xi (\mathring T) \setminus Z_*^2  \subset B^\nu \setminus Z_*^2$ which satisfies
$$\left\lvert \alpha \cdot y_T + \beta \right\rvert \geq m_{\alpha,\beta}(\eta) > 0,$$
with $\Phi$ introduced in \eqref{eq:Phi}.

\medskip

Remembering that $ \omega(\e_{k_j}) / \varrho_j \to 0$ and that the Lipschitz constant of $\Phi$ is less than $\sqrt{1 + 4 \eta^2} \leq 2$ for $\eta$ small enough, together with the uniform convergence \eqref{eq:CVU} and \eqref{eq:C* AGAIN}, it follows that for all $ \gamma > 0$, there exists $ j_2(\gamma)  \geq j_1(\gamma)$ such that for all $ j \geq j_2(\gamma) $, 
\begin{subequations}
	\begin{empheq}[left=\empheqlbrace]{align}
	& \left\lvert \frac{1}{\mu_j} (r^+_j - r^-_j)^\xi_y - \big( \alpha \cdot y + \beta \big) \right\rvert \leq \frac{m_{\alpha,\beta}}{8} \gamma \quad \text{ for all }  y \in B^\nu \setminus Z_*, \label{seq:CVU to infty} \\
	& \HH^1(\Phi \circ p_\xi(T) ) \leq 2 \omega(\e_{k_j}) / \varrho_j \leq \frac{m_{\alpha,\beta}}{8} \gamma \quad \text{ for all } T \in \mathbf T^{x_0,j}.\label{seq:y near y_T}
	\end{empheq}
\end{subequations}
Therefore, for all $j \geq j_2(\gamma)$ and all $T \in \widehat{\mathbf T}_j$, we introduce the following quantities :
\begin{equation}\label{eq:Lref Lmax AGAIN}
 \begin{cases}
\ds L^{\rm ref}(T) := \frac{ \left\lvert [u](x_0) \cdot \xi + \mu_j \big( \alpha \cdot y_T + \beta \big) \right\rvert - \big( C_\eta + \frac{m_{\alpha,\beta}}{2} \mu_j  \big)  \gamma }{\left\lvert e(v_j)_{|T} : (\xi \otimes \xi) \right\rvert} \text{ the reference length of } T, \\
\ds L^{\rm max}(T) := \underset{ z \in p_\xi (T) }{\max} \, \LL^1(T^\xi_z) \text{ the maximal section's length of } T \text{ along the direction } \xi.
\end{cases}
\end{equation}
Note that $L^{\rm ref}(T)$ is well defined (since  $ |e(v_j)_{|T} \xi \cdot \xi |^2 \geq \kappa\varrho_j^2 / (\alpha \e_{k_j})> 0$ as $T \in \mathbf{T}^{x_0, j}_b$) and positive for $j$ large enough since $\mu_j \to + \infty$ and $ \lvert \alpha \cdot y_T + \beta \rvert >  m_{\alpha,\beta} \, \gamma / 2 > 0$. Moreover, we have $L^{\rm max}(T) > L^{\rm ref}(T)$. Indeed, if such would not be the case, denoting by $y \in Y_j \setminus Z_*$ a point such that $(\mathring T \cap B)_y^\xi\neq \emptyset$, then $ \LL^1(T^\xi_{p_\xi(y)}) = \LL^1(T^\xi_y) = b_j(y) - a_j(y) \leq L^{\rm max}(T) \leq L^{\rm ref}(T) $, entailing that 
\begin{multline*}
\lvert (v_j)^\xi_y (b_j(y)) -  (v_j)_y^\xi(a_j(y)) \rvert = \left\lvert e(v_j)_{|T} : \left( \xi \otimes \xi \right) \right\rvert \left( b_j(y) - a_j(y) \right) \\
\leq \left\lvert [u](x_0) \cdot \xi + \mu_j \big( \alpha \cdot y_T + \beta \big) \right\rvert - \big( C_\eta + \frac{m_{\alpha,\beta}}{2}  \mu_j \big)  \gamma,
\end{multline*}
by definition \eqref{eq:Lref Lmax AGAIN} of $L^{\rm ref}(T)$. Therefore, we would obtain that 
$$ 
\begin{aligned}
C_\eta \gamma + \frac{m_{\alpha,\beta}}{2} \mu_j \gamma  \leq & \left\lvert [u](x_0) \cdot \xi + \mu_j \big( \alpha \cdot y_T + \beta \big)  \right\rvert - \left\lvert (v_j)^\xi_y (b_j(y)) -  (v_j)_y^\xi(a_j(y))  \right\rvert  \\
 \leq & \left\lvert  (v_j)^\xi_y (b_j(y)) -  (v_j)_y^\xi(a_j(y)) - [u](x_0)\cdot \xi - (r^+_j - r^-_j)^\xi_y  \right\rvert \\
& \,  + \, \left\lvert (r^+_j - r^-_j)^\xi_y - \mu_j \big( \alpha \cdot y + \beta \big)  \right\rvert \, + \, \left\lvert  \mu_j  \alpha \cdot ( y - y_T ) \right\rvert \\
 \leq & C_\eta \gamma + \frac{m_{\alpha,\beta}}{8} \mu_j \gamma + \frac{m_{\alpha,\beta}}{8} \mu_j \gamma,
\end{aligned}
$$
where we used \eqref{eq:C* AGAIN}, \eqref{seq:CVU to infty} and \eqref{seq:y near y_T} (since $y,y_T \in \Phi \circ p_\xi (T)$), leading to a contradiction.

\medskip

Therefore, arguing as in the proof of Lemma \ref{lemma3}, there are exactly one or two points $z^1_{\rm ref}$, $z^2_{\rm ref} \in p_\xi(T)$, only depending on $j$ and $T$, such that $\LL^1(T^\xi_{z^1_{\rm ref}}) =\LL^1(T^\xi_{z^2_{\rm ref}}) = L^{\rm ref}(T).$ Then, as in \eqref{eq:Horizontal Tubes}, we introduce the following segments (orthogonal to $\xi$) associated to $T$ (see Figure \ref{fig:F}),
\begin{equation*}
\mathfrak{T}_i(T) := \left\{ z \in \Pi_\xi: \; \left\lvert z - z^i_{\rm ref} \right\rvert \leq C'_\eta \, \frac{  \varrho_j \LL^2(T) }{\e_{k_j} } \, \gamma \right\} \quad \text{ for } i \in \{ 1,2 \},
\end{equation*}
where the constant $C'_\eta$, only depending on $\eta$, now changes into
$$C'_\eta:= \frac{8}{\sin \theta_0  } \left( \frac{2 C_\eta}{m_{\alpha,\beta}} + 1 \right).$$
For every $j \geq j_2(\gamma)$ and every $y \in Y_j \setminus Z_*$, let $T \in  \mathbf{T}^{x_0, j}_{b,int}$ be such that $(\mathring T \cap B)_y^\xi\neq \emptyset$. In particular, note that $T \in \widehat{\mathbf T}_j$. Arguing in the same way as in the proof of Lemma \ref{lemma3}, we get that there exists $i \in \{1,2\}$ such that
\begin{eqnarray*}
&&\hspace{-0.5cm}\lvert p_\xi(y) - z^i_{\rm ref} \rvert\\
&  & \leq \frac{2 \LL^2(T)}{h_T} \, \frac{  \left\lvert \left( b_j(y) - a_j(y) \right) - L^{\rm ref}(T) \right\rvert  }{L^{\rm ref}(T) } \\ 
& &= \frac{2 \LL^2(T)}{h_T} \, \frac{  \left\lvert   | (v_j)^\xi_y (b_j(y)) - (v_j)_y^\xi(a_j(y))| - \left\lvert  [u](x_0) \cdot \xi + \mu_j (\alpha \cdot y_T + \beta) \right\rvert + \big( C_\eta + \frac{m_{\alpha,\beta}}{2} \mu_j  \big)  \gamma  \right\rvert  }{ \left\lvert [u](x_0) \cdot \xi + \mu_j \big( \alpha \cdot y_T + \beta \big) \right\rvert - \big( C_\eta + \frac{m_{\alpha,\beta}}{2} \mu_j  \big)  \gamma  } \\
&& \leq  \frac{2 \LL^2(T)}{h_T} \, \frac{  C_\eta \gamma + \left\lvert ( r^+_j - r^-_j)^\xi_y - \mu_j (\alpha \cdot y_T + \beta) \right\rvert + \big( C_\eta + \frac{m_{\alpha,\beta}}{2} \mu_j  \big)  \gamma  }{ \mu_j \, m_{\alpha,\beta} /4  } \\
&& \leq \frac{2 \LL^2(T)}{h_T} \, \frac{ \big( 2 C_\eta + m_{\alpha,\beta} \mu_j \big) \gamma}{ \mu_j \, m_{\alpha,\beta} / 4} \leq C_\eta' \, \frac{  \varrho_j \LL^2(T) }{\e_{k_j} } \, \gamma,
\end{eqnarray*}
where we used \eqref{eq:C* AGAIN}, \eqref{seq:CVU to infty}, \eqref{seq:y near y_T} and the fact that 
$$  \left\lvert [u](x_0) \cdot \xi + \mu_j \big( \alpha \cdot y_T + \beta \big) \right\rvert - \left( C_\eta + \frac{\mu_j\,  m_{\alpha,\beta}}{2}  \right)  \gamma  \geq \frac{\mu_j \, m_{\alpha,\beta}}{4}$$
 up to enlarging $j_2(\gamma) \in \N$. As in the proof of Lemma \ref{lemma3}, we deduce that for all $j \geq j_2(\gamma)$,
$$
\HH^1 ( Y_j \setminus Z_* )  \leq \sum_{ T \in \widehat{\mathbf T}_j } \HH^1( \Phi \left( \mathfrak{T}_1(T) \cup \mathfrak{T}_2(T) \right) ) \leq 8 C'_\eta \, \gamma \,  \frac{\varrho_j}{ \e_{k_j}}\sum_{ T \in \widehat{\mathbf T}_j } \LL^2(T)  \leq  \frac{8 C'_\eta\gamma}{\kappa }  \frac{\kappa\varrho_j}{\e_{k_j}} \int_{B} \chi_j \, dx.$$
Recalling \eqref{eq:lambda_k greater than AGAIN} and possibly taking a larger $j_2(\gamma) \in \N$, we finally get that for all $j \geq j_2(\gamma)$,
$$ 
\HH^1 ( Y_j \setminus Z_* ) \leq \frac{ 8 C'_\eta}{\kappa}  \, \left(2 \frac{d\lambda}{d\HH^1\res \partial^* \mathcal{P}}(x_0) + 1\right) \gamma =: C_* \gamma,
$$
for some constant $C_*>0$ only depending on $\eta$, which settles Lemma \ref{lemma3 AGAIN}.
\end{proof}

Arguing exactly as in the proof of Lemma \ref{lemma2}, having Lemma \ref{lemma3 AGAIN}  at hand, we deduce the following result.

\begin{lem}\label{lemma2 AGAIN}
For all $\eta>0$, there exist $ Z'' \subset B^\nu$ containing $Z'$ with $\HH^1(Z'') \leq 4 \eta$, and a (not relabeled) subsequence such that for all $j \in \N$ and for all $y \in(B_{1-\frac\eta2})^\nu\setminus Z''$,
$$ \# \left\{ T \in \mathbf{T}^{x_0, j}_{b,int}: \; ( \mathring{T} \cap B )^\xi_y \neq \emptyset \right\} \geq 2.$$
\end{lem}
\noindent
Finally, owing to Lemma \ref{lemma2 AGAIN}, the proof of Proposition \ref{prop Jump part in essential bdry P} is identical to that of Proposition \ref{prop Jump part}.

\bigskip

In the following result, we prove the existence of minimizers of the discrete brittle damage energy $\mathcal G_\e$ on $V^{\rm Dir}_\e(\O')$.

\begin{lem}\label{rem:existence_Minimizers}
Assume that $\O$ and $\O'$ are connected. For $\e >0$ sufficiently small, there exists a minimizer $u_\e  \in V^{\rm Dir}_{\e}(\O')$ of $\G_\e$.
\end{lem}

\begin{proof}
Let $\e_0:= \kappa  / ( \beta \norme{\nabla w}^2_{L^\infty(\R^2;\mathbb M^{2\times 2})} )$ and fix $\e<\e_0$. Since $\G_\e(w_{\mathbf T_\e}) < + \infty$, we can consider a minimizing sequence $\{ u_n \}_{n\in \N} \subset V^{\rm Dir}_\e (\O')$ satisfying 
\begin{equation}\label{eq:minseq}
\lim_{n\to \infty} \G_\e(u_n) = \inf_{L^0(\O;\R^2)} \G_\e \in [0,+\infty).
\end{equation}
By definition of the finite element space $V^{\rm Dir}_{\e}(\O')$, there exists a triangulation $\mathbf T^n \in \mathcal T_{\e}(\O')$ such that $u_n$ is affine on each $T \in \mathbf T^n$ and $u_n = w_{\mathbf T^n}$ on every triangle $T \in \mathbf T^n$ such that $T \cap ( \O' \setminus \bar \O) \neq \emptyset$.

\medskip

Let $\O''$ be a bounded open set such that $\O' \subset \subset \O''$ and $\bigcup_{ T \in \mathbf T_n} T \subset\O''$ for all $n \in \N$. Since, for all $T \in \mathbf T^n$, $\LL^2(T) \geq \e^2\sin\theta_0/2$, it is easily seen that
$$\#\mathbf T^n \leq \frac{2 \LL^2(\O'')}{\e^2\sin\theta_0}.$$
As a consequence, the sequence of integers $\{ \# \mathbf T^n  \}_{n \in \N}$ admits a subsequence converging as $n \to +\infty$ to an integer $N \in \N$. We can thus assume, without loss of generality, that 
$$\# \mathbf T^n =N \quad \text{ for all }n \in \N.$$
We write $\mathbf T^n=\{ T^n_1,\ldots, T^n_{N} \} $ for all $n \in \N$. Up to a subsequence, we can check that for all $i \in \{ 1,\ldots, N \}$, the closed triangle $T^n_i$ converges to a closed limit triangle $T_i$ in the sense of Hausdorff, with the property that the limit triangulation $\mathbf T:=\{ T_1,\ldots,T_{N} \} \in \mathcal T_\e(\O')$ remains an admissible triangulation of $\O'$. 

\medskip

Introducing the characteristic functions $ \chi_n := \mathds{1}_{\left\{\e \mathbf A  e(u_n):e(u_n) \geq \kappa \right\}} \in L^\infty(\O'; \lbrace0,1 \rbrace),$ we can write the energy as
\begin{equation}\label{energychi_n}
\G_\e(u_n) = \int_\O  (1-\chi_n) \mathbf A e(u_n):e(u_n) \, dx + \frac{\kappa}{ \e} \int_\O \chi_n \, dx.
\end{equation}

First, by definition of $\e_0$ and since $\e<\e_0$, we have that $\chi_n=0$ in $\O' \setminus \bar \O$ for all $n \in \N$, since
$$\mathbf A e(u_n):e(u_n)=\mathbf A e(w_{\mathbf T_n}):e(w_{\mathbf T_n}) \leq \beta |e(w_{\mathbf T_n})|^2 \leq \beta |\nabla w_{\mathbf T_n}|^2 \leq \beta \|\nabla w\|^2_{L^\infty(\R^2;\mathbb M^{2 \times 2})} < \frac{\kappa}{\e}$$
on that set. Being constant equal to $1$ or $0$ on each triangle of $\mathbf T^n$, $\chi_n$ can be identified with a vector $V_n \in \{ 0,1 \} ^{N}$. Hence, up to a subsequence, there exists $V \in \{ 0,1 \}^{N}$ such that $V_n \to V$ in $\R^{N}.$ In particular,  there exists $n_0 \in \N$ such that $V_n = V$ for all $n \geq n_0$. Up to reordering the triangles, we can thus find a integer $0 \leq M < N$ such that
$$\{\chi_n=1\}=\bigcup_{i=1}^{M} T_i^n, \quad \{\chi_n=0\}=\bigcup_{i=M+1}^{N}T_i^n \quad \text{ for all }n \geq n_0.$$
By the Hausdorff convergence property, we infer that 
\begin{equation}\label{eq:1540}
\chi_n \to \chi:=\mathds{1}_{\bigcup_{i=1}^{M} T_i} \quad \text{strongly in }L^1(\O').
\end{equation}

We next show some compactness on the sequence of displacements $\{u_n\}_{n \in \N}$, carefully overcoming the lack of control on $\{ \chi_n e(u_n) \}_n$ in $L^2(\O';\mathbb M^{2 \times 2}_{\rm sym})$.
Remembering that $(1-\chi_n)| e(u_n)|^2 \leq \kappa / (\alpha \e)$ for all $n \geq n_0$ and that the sequence $\{(1-\chi_n) e(u_n) \}_{n \in \N}$ lives in the finite dimensional space $(\mathbb M^{2 \times 2})^N$, up to a new subsequence (not relabeled), there exists a function $\xi \in L^\infty(\O'; \mathbb M^{2 \times 2}_{\rm sym})$ which is constant on each triangle $T \in \mathbf T$ such that
\begin{equation}\label{eq:CV linearized strain}
 (1 - \chi_n) e(u_n) \to \xi \text{ strongly in } L^2(\O';\mathbb M^{2 \times 2}_{\rm sym}),
\end{equation}
and $\xi=0$ on $\bigcup_{i=1}^{M} T_i$. Let us define the set
$$ \omega_0 := \bigcup_{i=M+1}^{N}T_i.$$
Note that $\O' \setminus \bar \O \subset  \omega_0$. Indeed, if $x \in \O' \setminus \bar \O$, then for all $n \geq n_0$, there exists  $M+1\leq i_n \leq N$ such that $x \in T_{i_n}^n$. At the expense of extracting a further subsequence, there is no loss of generality to assume that $i_n=i$ is independent of $n$. By the Hausdorff convergence of $T_i^n$ to $T_i$, we infer that $x \in T_i \subset \omega_0$. By connectedness of $\O' \setminus \bar \O \subset \omega_0$, we can consider $\omega$ the connected component of $\omega_0$ containing $\O' \setminus \bar \O$. Let $M \leq K< N$ be such that $\omega=\bigcup_{i=K+1}^N T_i$, up to reordering the triangles again.

Observe that for all $T \in \mathbf T$ such that $T \cap (\O' \setminus \bar \O) \neq \emptyset$, then $T \in \{T_{K+1},\ldots, T_N \}$. Thus, for all $T \in \{ T_1,\ldots, T_{K} \}$, $T \cap (\O' \setminus \bar \O) = \emptyset$ so that $T \subset \bar \O $. Therefore, for all open set $ W \subset \subset \O'$ with $  \bigcup_{i=1}^K  T_i \subset W  $, having that
$$\bigcup_{i=1}^K  T_i^n \to \bigcup_{i=1}^K  T_i \quad \text{ in the sense of Hausdorff,}$$
there exists $n_1 \geq n_0$ such that $ \bigcup_{i=1}^K  T_i^n \subset W$ for all $n \geq n_1$. Since 
$$\O' \setminus \overline W \subset \bigcup_{i=K+1}^N  T_i^n \subset \bigcup_{i=M+1}^N  T_i^n = \{\chi_n=0\},$$
owing to \eqref{eq:minseq}, \eqref{energychi_n} and that $u_n=w_{\mathbf T^n}$ in $\O' \setminus \bar \O$, we infer that
$$\int_{\O' \setminus \overline W} |e(u_n)|^2\, dx \leq C_*,$$
for some constant $C_*>0$ independent of $n$ and $W$. Using that $u_n - w_{\mathbf T^n} \in H^1(\O' \setminus \overline W;\R^2)$ is equal to $0$ on the open set $(\O' \setminus \overline W) \cap (\O' \setminus \bar \O) \neq \emptyset$, the Poincar\'e--Korn inequality ensures that (up to a subsequence) there exists $u \in H^1(\O' \setminus \overline W;\R^2)$ such that
$$u_n \rightharpoonup u \text{ weakly in } H^1(\O' \setminus \overline W;\R^2),$$
$u = w_{\mathbf T} $ on $(\O' \setminus \overline W ) \cap (\O' \setminus \bar \O)$ since $w_{\mathbf T^n} \to w_{\mathbf T}$ strongly in $H^1(\O';\R^2)$ and, thanks to \eqref{eq:CV linearized strain}, $e(u)=\xi$ in $\O' \setminus \overline W$. In addition, by weak lower semicontinuity of the norm, we get that
$$\int_{\O' \setminus \overline W} |e(u)|^2\, dx \leq \liminf_{n \to \infty} \int_{\O' \setminus \overline W} |e(u_n)|^2\, dx\leq C_*.$$

Considering a decreasing sequence of open sets $\{W_j\}_{j \in \N}$ such that  $\bigcup_{i=1}^K  T_i \subset W_j \subset\subset \O'$ for each $j \in \N$, and $\bigcap_j W_j=\bigcup_{i=1}^K T_i $, we deduce through a diagonalisation argument that there exists $u \in H^1(\mathring \omega\cap \O';\R^2)$ such that $u = w_{\mathbf T}$ on $\mathring \omega \cap (\O' \setminus \bar \O)= \O' \setminus \bar \O $ and $e(u)=\xi$ in $\mathring\omega\cap \O'$. In particular, since $\xi$ is constant in each triangle of $\omega$, we infer that $u$ is affine in the interior of each triangle of $\omega$. Being in $H^1(\mathring \omega\cap \O';\R^2)$, we get that $u$ is continuous at the interfaces of each triangle in $\omega$. Moreover, since $u = w_{\mathbf T}$ on $ \O' \setminus \bar \O$, we deduce that $u_{|T} = w_{\mathbf T}$ on each triangle $T \in \mathbf T$ such that $ T \cap ( \O' \setminus \bar \O) \neq \emptyset$. Note that $u$ is defined on such triangles $T$, as they are included in $\omega$.

\medskip 

In order to extend $u$ outside $\omega$, we introduce the family of triangles which are at a distance of at least one triangle from $\omega$, i.e.
$$\mathbf T^{\rm far} := \{ T \in \mathbf T : \;  T \cap \omega = \emptyset \} \subset \{T_1,\ldots,T_K\},$$
so that every remaining triangle $T \not\in \mathbf T^{\rm far}$ and such that $T \not\subset\omega$, has its three vertices in ${\rm Vertices}(\omega) \cup {\rm Vertices}\left( \mathbf T^{\rm far} \right)$.  Note that $\{T_{M+1},\ldots T_K \} \subset \mathbf T^{\rm far}$ since, by construction of the connected component $\omega$ of $\omega_0$, each triangle $T \in \mathbf T$ included in $\omega_0 \setminus \mathring\omega$ is at a distance of at least one triangle from $\omega$. We extend the function $u$ to all triangles by setting $u \equiv 0$ on every triangle $T \in \mathbf T^{\rm far}$, and by interpolating on each remaining triangle which happens to have its three vertices' values imposed. It defines a function $u \in V^{\rm Dir}_\e(\O')$ which satisfies $e(u) = \xi$ on $\omega$ and $e(u) \equiv 0$ on each triangle $T \in \{T_{M+1},\ldots T_K \} \subset \mathbf T^{\rm far}$. 

On the one hand, $\xi=0$ in $\{ \chi = 1 \}$, hence $\xi=(1-\chi)\xi$. On the other hand, $e(u) = \xi$ in $\omega$ and $e(u) = 0$ in $ \omega_0 \setminus \omega$, so that $(1-\chi) \mathbf A \xi : \xi \geq (1-\chi) \mathbf A e(u):e(u)$ by positivity of $\mathbf A$. Thus, by \eqref{eq:1540} together with \eqref{eq:CV linearized strain},
\begin{multline*}
\inf_{L^0(\O;\R^2)} \G_\e = \lim_{n\to \infty}\left\{ \int_\O (1-\chi_n) \mathbf A e(u_n):e(u_n) \, dx + \frac{\kappa}{\e} \int_\O \chi_n\, dx\right\} \\
 =  \int_\O (1-\chi)  \mathbf A \xi : \xi \, dx + \frac{\kappa}{\e} \int_\O \chi \, dx \geq \int_\O (1- \chi) \mathbf A e(u) : e(u) \, dx + \frac{\kappa}{\e} \int_\O \chi \, dx = \G_\e (u),
\end{multline*}
which settles that $u$ is a minimizer of $\G_\e$.
\end{proof}

\begin{rem}
The above proof strongly relies on the choice of the density $f(t) = \kappa \wedge t$, mainly because of the identification \eqref{energychi_n}, which would unfortunately result into a too low lower bound on the energy for a general $f$. Indeed for a generic $f$ satisfying \eqref{eq:f} one only gets, for all $\delta >0$, the existence of a constant $0 < K^\delta < \kappa$ such that
$$
\G_\e(u_n) \geq (1-\delta) \int_\O  (1-\chi^\delta_n) \mathbf A e(u_n):e(u_n) \, dx + \frac{K^\delta}{ \e} \int_\O \chi^\delta_n \, dx
$$
where the characteristic function $ \chi^\delta_n := \mathds{1}_{\left\{\e (1 - \delta) \mathbf A  e(u_n):e(u_n) \geq K^\delta \right\}} \in L^\infty(\O'; \lbrace0,1 \rbrace)$ depends on $\delta$. Even in the case where the above proof could be adapted to show the existence of a displacement $u^\delta \in V^{\rm Dir}_\e(\O')$ and a characteristic function $\chi^\delta$ such that (up to a subsequence, not relabeled)
\begin{multline*}
 \lim_{n\to \infty}\left\{ (1-\delta) \int_\O (1-\chi^\delta_n) \mathbf A e(u_n):e(u_n) \, dx + \frac{K^\delta}{\e} \int_\O \chi^\delta_n\, dx\right\} \\
 \geq (1 - \delta) \int_\O (1- \chi^\delta) \mathbf A e(u^\delta) : e(u^\delta) \, dx + \frac{K^\delta}{\e} \int_\O \chi^\delta \, dx ,
\end{multline*}
the above lower bound might be too low since $f(t) > \sup_{\delta >0} \, \left\{ K^\delta \wedge (1-\delta) t \right\}$ a priori.
\end{rem}

\medskip

We are now in position to prove the fundamental property of $\Gamma$-convergence.
\begin{proof}[Proof of Corollary \ref{thm:CV_minimizers}]
On the one hand, for all $\e >0$, we remark that $\G_\e (u_\e) \leq \G_\e(w_{\mathbf T_\e })$ is uniformly bounded due to \eqref{eq:w_eps}. Therefore, Proposition \ref{prop:Comp} implies that, up to a subsequence, there exist a sequence of piecewise rigid motions $\{r_\e\}_{\e>0}$ and a function $u \in GSBD^2(\O')$ with $u=w$ $\LL^2$-a.e. in $\O' \setminus \overline \O$, such that $u_\e - r_\e \to u$ in measure in $\O'$ and $\liminf_\e \G_\e (u_\e) \geq \G(u)$.

On the other hand, the $\Gamma$-convergence of $\G_\e$ to $\G$ ensures that, for all $v \in GSBD^2(\O')$ with $v=w$ $\LL^2$-a.e. in $\O' \setminus \overline \O$, there exists a recovery sequence $v_\e \in L^0(\O;\R^2)$ such that $v_\e \to v$ in measure in $\O'$ and $\G_\e(v_\e) \to \G(v)$. Hence 
$$
\G(v) = \lim_{\e\to 0} \G_\e(v_\e) \geq \limsup_{\e\to 0} \G_\e(u_\e) \geq \liminf_{\e\to 0} \G_\e (u_\e) \geq \G(u),
$$
implying both that $u \in {\rm arg \, min} \, \G$ and $\G_\e(u_\e) \to \G(u)$.
\end{proof}

\section{Appendix}\label{appendix:Cortesani-Toader Dirichlet}

\begin{prop}\label{prop:Cortesani-Toader Dirichlet}
Let $N,m\in \N \setminus \{ 0 \}$, $p \in (1,2]$, $k \in \N \setminus \{0,1\}$, $w \in W^{k,\infty}(\R^N;\R^m)$ and $\O \subset \R^N$ be a bounded open set with Lipschitz boundary. For all $u \in SBV^p(\O;\R^m) \cap L^\infty(\O; \R^m)$ such that $u = w $ in an open bounded neighborhood of $ \partial \O$, there exist a sequence $\{u_h\}_{h \in \N}$ in $SBV^p(\O;\R^m) \cap L^\infty(\O;\R^m)$ as well as $N_h$ disjoint closed $(N-1)$-dimensional simplexes $\Sigma^h_1,\ldots,\Sigma_{N_h}^h \subset \O$ satisfying:
$$\bar{J_{u_h}} = \bigcup_{i=1}^{ N_h} \Sigma^h_i  , \quad \HH^{N-1}( \bar{J_{u_h}} \setminus J_{u_h}) = 0, \quad u_h \in W^{k,\infty}(\O \setminus \bar{J_{u_h}};\R^m),$$
\begin{equation}\label{eq1549}
\begin{cases}
u_h = w \text{ in an open bounded neighborhood of } \partial \O, \\
u_h \to u \quad \text{ strongly in } L^1(\O;\R^m), \\
\nabla u_h \to  \nabla u \quad \text{ strongly in } L^p(\O; \mathbb M^{m \times N}), \\
\limsup_{h \to \infty} \HH^{N-1}(J_{ u_h} ) \leq \HH^{N-1}(J_u).
\end{cases}
\end{equation}
\end{prop}

We do not detail the proof of this result which follows the steps of the constructive proofs of \cite[Lemma 5.2]{BrChiad}, \cite[Theorem 3.1]{CT} and \cite[Theorem 3.9, Corollary 3.11]{Cort} with minor adaptations to the Dirichlet setting. The key point here is that, due to our definition of $V_\e^{\rm Dir}(\O')$, we need the approximating sequence to coincide with $w$ in an open neighborhood of the boundary and not only on $\partial\O$ (as in Theorem 4.2, Remark 4.3 and formula (4.1) in \cite{CC5}). 

\section*{Acknowledgements}

The authors wish to thank Antonin Chambolle for many useful discussions on the topic of this work, and for sharing his ideas.

\end{document}